\documentclass{amsart}

\usepackage{amsmath,amsthm,amssymb,amscd}
\usepackage{tikz}
\usepackage[all,cmtip,color]{xy}
\usepackage[german,english]{babel}
\usepackage{amsmath}
\usepackage{hyperref}
\usepackage{mathrsfs} 
\usepackage{todonotes}
\usepackage[all,cmtip]{xy}
\usepackage{geometry}
\usepackage{enumerate}
\usepackage{mathtools}
\usepackage{wasysym}
\usepackage{colonequals}
\usepackage[export]{adjustbox}
\usepackage{graphicx}
\usepackage{wrapfig}
\usepackage{units}
\newcommand{\bd}{\begin{displaymath}}
\newcommand{\ed}{\end{displaymath}}
\newcommand{\on}{\operatorname}

\DeclareMathOperator{\Hom}{\mathsf{Hom}}

\DeclareMathOperator{\AbGrp}{\mathsf{AbGrp}}

\DeclareMathOperator{\pf}{\mathfrak{p}}
\DeclareMathOperator{\ev}{ev}

\DeclareMathOperator{\sing}{sing}

\DeclareMathOperator{\Hb}{\mathbb{H}}

\DeclareMathOperator{\Pic}{\mathsf{Pic}}

\DeclareMathOperator{\tr}{\mathsf{tr}}

\DeclareMathOperator{\mon}{\mathrm{mon}}

\DeclareMathOperator{\spe}{\mathrm{sp}}

\DeclareMathOperator{\const}{const}

\DeclareMathOperator{\Ho}{\mathsf{Ho}}
\DeclareMathOperator{\Nb}{\mathbb{N}}

\newcommand{\Sh}{\mathsf{Sh}^{\wedge}}
\renewcommand{\Pr}{\mathsf{Pr}}

\DeclareMathOperator{\Sheaves}{\mathsf{Sh}}

\DeclareMathOperator{\E}{\mathcal{E}}

\DeclareMathOperator{\Cb}{\mathbb{C}}
\DeclareMathOperator{\Fb}{\mathbb{F}}
\DeclareMathOperator{\Bc}{\mathcal{B}}
\DeclareMathOperator{\Xc}{\mathcal{X}}
\DeclareMathOperator{\Yc}{\mathcal{Y}}
\DeclareMathOperator{\Zc}{\mathcal{Z}}
\DeclareMathOperator{\Oc}{\mathcal{O}}

\newcommand{\Cc}{\mathcal{C}}

\newcommand{\Hc}{\mathcal{H}}

\DeclareMathOperator{\Spaces}{Space}
\DeclareMathOperator{\Sets}{Set}

\DeclareMathOperator{\Rat}{\mathsf{Rat}}

\newcommand{\Ktop}{\mathsf{K}^{\rm top}}
\newcommand{\uKtop}{\underline{\mathsf{K}}^{\rm top}}
\DeclareMathOperator{\KtopS}{\mathsf{K}^{\rm top}_S}
\DeclareMathOperator{\KtopC}{\mathsf{K}^{\rm top}_{\Cb}}

\DeclareMathOperator{\an}{\mathsf{an}}

\DeclareMathOperator{\colim}{\mathsf{colim}}
\renewcommand{\lim}{\mathsf{lim}}

\DeclareMathOperator{\perf}{perf}
\DeclareMathOperator{\Perf}{\mathsf{Perf}}

\newcommand{\C}{\mathsf{C}}

\newcommand{\Spectra}{\mathsf{Sp}}

\newcommand*\cocolon{%
        \nobreak
        \mskip6mu plus1mu
        \mathpunct{}%
        \nonscript
        \mkern-\thinmuskip
        {:}%
        \mskip2mu
        \relax
}

\DeclareMathOperator{\Br}{\mathsf{Br}}

\DeclareMathOperator{\Pb}{\mathbb{P}}
\DeclareMathOperator{\Tb}{\mathbb{T}}

\DeclareMathOperator{\dual}{\vee}
\DeclareMathOperator{\id}{id}

\DeclareMathOperator{\Sp}{Sp}

\DeclareMathOperator{\Mm}{\mathcal{M}}
\DeclareMathOperator{\cMm}{\hat{\mathcal{M}}}
\DeclareMathOperator{\hMm}{\check{\mathcal{M}}}

\DeclareMathOperator{\Sb}{\mathbb{S}}

\DeclareMathOperator{\Tt}{\mathcal{T}}
\DeclareMathOperator{\Aa}{\mathcal{A}}

\DeclareMathOperator{\Qb}{\mathbb{Q}}
\DeclareMathOperator{\Zb}{\mathbb{Z}}
\DeclareMathOperator{\Db}{\mathbb{D}}

\DeclareMathOperator{\Fun}{\mathsf{Fun}}
\DeclareMathOperator{\Jac}{\mathsf{Jac}}

\DeclareMathOperator{\op}{\mathsf{op}}

\DeclareMathOperator{\Ab}{\mathbb{A}}

\DeclareMathOperator{\Cat}{\mathsf{Cat}}

\DeclareMathOperator{\F}{\mathcal{F}}

\DeclareMathOperator{\image}{im}

\DeclareMathOperator{\Map}{\mathsf{Map}}
\DeclareMathOperator{\G}{\mathbb{G}}
\DeclareMathOperator{\Eb}{\mathbb{E}}

\DeclareMathOperator{\GL}{GL}
\DeclareMathOperator{\PGL}{PGL}
\DeclareMathOperator{\SL}{SL}

\DeclareMathOperator{\Z}{\mathcal{Z}}
\DeclareMathOperator{\D}{\mathsf{D}}
\DeclareMathOperator{\Ee}{\mathcal{E}}
\DeclareMathOperator{\Kk}{\mathcal{K}}

\DeclareMathOperator{\Spec}{Spec}

\DeclareMathOperator{\Oo}{\mathcal{O}}

\renewcommand{\top}{\mathsf{top}}

\DeclareMathOperator{\Ext}{Ext}

\DeclareMathOperator{\pr}{pr}

\newcommand{\FM}{\mathsf{FM}}
\newcommand{\KU}{\mathsf{KU}}

\newcommand{\MS}{\mu}

\newcommand{\coh}{\mathsf{coh}}

\newcommand{\Lc}{\mathcal{L}}
\newcommand{\gr}{\mathsf{gr}}
\newcommand{\Orb}{\mathsf{Orb}}
\newcommand{\Cov}{\Orb}
\newcommand{\std}{\mathrm{std}}
\newcommand{\act}{\mathrm{act}}

\newcommand{\Maps}{\mathsf{Map}}
\newcommand{\Open}{\mathsf{Open}}

\newcommand{\Top}{\mathsf{Top}}
\newcommand{\TopSt}{\mathsf{TopSt}}
\newcommand{\GTopSt}{\mathsf{G-TopSt}}
\newcommand{\Gc}{\mathcal{G}}

\newcommand{\Sing}{\mathsf{Sing}\;}
\newcommand{\Latt}{\mathsf{SubGrps}}

\AtBeginDocument{%
   \def\MR#1{}
}

\begin{document}

\theoremstyle{definition}
\newtheorem{definition}{Definition}[section]
\newtheorem{rmk}[definition]{Remark}

\newtheorem{example}[definition]{Example} 

\newtheorem{ass}[definition]{Assumption}
\newtheorem{warning}[definition]{Warning}
\newtheorem{porism}[definition]{Porism}
\newtheorem{hope}[definition]{Hope}
\newtheorem{situation}[definition]{Situation}
\newtheorem{construction}[definition]{Construction}

\theoremstyle{plain}

\newtheorem{theorem}[definition]{Theorem}
\newtheorem{proposition}[definition]{Proposition}
\newtheorem{corollary}[definition]{Corollary}
\newtheorem{conj}[definition]{Conjecture}
\newtheorem{lemma}[definition]{Lemma}
\newtheorem{claim}[definition]{Claim}
\newtheorem{cl}[definition]{Claim}
\newtheorem{lemma-definition}[definition]{Lemma-Definition}

\title{Complex $K$-theory of moduli spaces of Higgs bundles}
\author{Michael Groechenig \and Shiyu Shen}
\address{Department of Mathematics, University of Toronto}
\email{michael.groechenig@utoronto.ca}
\address{Department of Mathematics, University of Toronto}
\email{shiyu.shen@utoronto.ca}
\begin{abstract}
We establish an isomorphism of complex $K$-theory of the moduli space $\hMm$ of $``\SL_n"$-Higgs bundles of degree $d$ and rank $n$ (in the sense of Hausel--Thaddeus) and twisted complex $K$-theory of the orbifold $\cMm$ of $\PGL_n$-Higgs bundles of degree $e$, where $(n,d)=(n,e)=1$. Along the way we prove the vanishing of torsion for $H^*(\hMm)$ and certain twisted complex $K$-theory groups of $\cMm$. We also extend Arinkin's autoduality of compactified Jacobian to a derived equivalence between $\SL_n$ and $\PGL_n$-Hitchin systems over the elliptic locus. In the appendix we develop a formalism of $G$-sheaves of spectra, generalising equivariant homotopy theory to a relative setting.
\end{abstract}
\thanks{
Michael Groechenig was supported by an NSERC discovery grant and an Alfred P. Sloan fellowship. Shiyu Shen has received funding from the European Union's Horizon 2020 research and innovation program under the Marie Skłodowska-Curie grant agreement No. 101034413.
}

\maketitle
\tableofcontents

\section{Introduction}

We fix a smooth and projective complex curve $X$ and denote by $\Mm = \Mm^d_e$ the moduli space of rank $n$ and degree $d$ Higgs bundles, where $(n,d)=1$. Following Hausel--Thaddeus \cite{ht}, we fix a degree $d$ line bundle $L \in \Pic^d(X)$ and let $\hMm=\hMm^L_n$ be the moduli space of $``\SL_n"$-Higgs bundles $(E,\theta)$ with determinant $L$. 

For $e \in \Zb$ satisfying $(n,e) = 1$ we let $\cMm = \cMm^e_n$ be the quotient stack 
$$\cMm = [\Mm^{\tr=0}/J]$$
of the moduli space of \emph{trace-free} Higgs bundles, with respect to the natural action of the Jacobian $J$. According to Hausel--Thaddeus \cite{ht}, this quotient stack is endowed with a gerbe $\alpha_L$  (which depends on the choice of $L$).

In this article we will compare complex $K$-theory groups of the underlying complex manifolds/orbifolds associated with the moduli problems above. A superscript ${}^{\an}$ indicates passage from the world of schemes to the complex-analytic world. The topological gerbe associated with an algebraic gerbe $\alpha$ will be denoted by $\tilde{\alpha}$.

Our main results are the following three statements, corresponding to Theorems \ref{thm:main2}, \ref{thm:SLn} \& \ref{thm:PGLn}:

\begin{theorem}\label{main}
\begin{itemize}
\item[(a)] There is an isomorphism of complex $K$-groups $\KU^*(\hMm^{\an}) \simeq \KU^*(\cMm^{\an},\tilde{\alpha})$, which is induced by an equivalence of complex $K$-theory \emph{spectra} $\KU(\hMm^{\an}) \simeq \KU(\cMm^{\an},\tilde{\alpha})$.
\item[(b)] The abelian groups $H^*(\hMm^{\an},\mathbb{Z})$ are free. This implies vanishing of torsion in $\KU^*(\hMm^{\an})$.
\item[(c)] The twisted complex $K$-theory groups $\KU^*(\cMm^{\an},\tilde{\alpha})$ are torsion-free.
\end{itemize}
\end{theorem}

The $n=2$ case of assertion (b) was communicated to us by Tom Baird (unpublished). 

\begin{rmk}
We define the twisted complex $K$-theory $\KU^*(\cMm^{\an},\tilde{\alpha})$ of the orbifold $\cMm^{\an}$ to be $\tilde{\alpha}$-twisted $J^{\an}$-equivariant complex $K$-theory of $(\Mm^{\tr=0})^{\an}$.

Equivalently, one can define it as $\tilde{\alpha}$-twisted $\Gamma$-equivariant complex $K$-theory of $(\hMm^e_n)^{\an}$, where $\Gamma=J[n]=H^1(X^{\an},\mu_n(\Cb))$.
\end{rmk}

The rationalisation of statement (a) above, is expected to hold for numerical reasons. Indeed, the rank of $\KU^i(\hMm^{\an})_{\Qb}$ (for $i=0,1$) is given by 
$\sum_{j-i\text{ even}} b_j(\hMm^{\an}),$
whereas the rank of $\KU^j(\cMm^{\an},\tilde{\alpha})$ is equal to
$\sum_{j-i\text{ even}} b_j^{\rm orb}(\cMm^{\an},\tilde{\alpha}),$
where $b_j^{\rm orb}$ denotes Betti numbers for twisted orbifold cohomology. The first equation is a well-known consequence of rational degeneration of the Atiyah-Hirzebruch spectral sequence, whereas the second equality follows from the work of Freed--Hopkins--Telemenan \cite{fht}. We therefore see that the equality of ranks is already known to hold, as a consequence of the main result of \cite{gwz17}, which resolves the Hausel--Thaddeus conjecture.

The interest of (a) lies in the fact that we can exhibit an explicit isomorphism, which stems in fact from an equivalence of $K$-theory spectra.

From a philosophical point of view, our main results confirm the expectation that there should be a twisted derived equivalence 
$$D^b_{\coh}(\hMm) \cong D^b_{\coh}(\cMm,\alpha),$$
extending the usual Fourier-Mukai transform between an abelian variety and its dual. Indeed, if such a derived equivalence was known to exist, one could use $K$-theoretic Fourier-Mukai transform to show the existence of an equivalence of spectra as in Theorem \ref{main}(a) above. We emphasise that our equivalence of $K$-theory spectra is indeed induced by a Fourier-Mukai transform, with the one caveat that we also need to apply the vanishing cycle construction of Maulik \& Junliang Shen \cite{MS}.

Another viewpoint on or main result goes right back to the original motivation for \cite{ht}. According to modern developments in the subject, one expects mirror partners to have isomorphic complex $K$-groups, possibly up to shift, and including possible torsion parts. Such an isomorphism is not expected to exist for \emph{integral singular cohomology}. This is related to the physics slogan that \emph{$D$-brane charges are valued in $\KU$}. We refer the reader to Treumann's \cite{treumann} for a mathematical exposition of this conjectural statement and some non-trivial examples. Treumann emphasises that such an isomorphism stems from an equivalence of complex $K$-theory spectra, an important quality that also plays a key role in our construction of the isomorphism.

In light of \cite{ht} we should think of the hyperk\"ahler manifolds $\hMm^{\an}$ and $\cMm^{\an}$ as being mirror partners (up to a hyperk\"ahler rotation). Our equivalence (a) above, can therefore be thought of as a $K$-theoretic mirror test for moduli spaces of Higgs bundles.

Section \ref{sec:torsion} is devoted to our second main result: vanishing of torsion  for singular cohomology of the moduli space $\check{\mathcal{M}}$ of ``$\SL_n$" Higgs bundles (Theorem \ref{main}(b)). This relies on the Garcia-Prada--Heinloth--Schmitt method \cite{GPHS,gph13}, which exhibits an algorithm expressing the class of the moduli space of Higgs bundles in a certain completion of the Grothendieck group of varieties. This yields a formula equating $[\check{\mathcal{M}}]$ in terms of of the motives of certain finite covers over symmetric powers of the base curve $C$. It then suffices to prove vanishing of torsion for the singular cohomology of these finite covers.

 After this part of the paper was completed we learnt from Victoria Hoskins and Simon Pepin-Lehalleur that they obtained closely related results on the Voevodsky motive of $\hMm$ \cite{HPL}. Their paper contains a description of the Voevodsky motive (with rational coefficients) of $\hMm$, which resembles the results in our section \ref{sec:torsion}. Nonetheless, it would not be possible to obtain our result from theirs and vice versa. 
 
\subsubsection{Methods}

The overall strategy underlying the proof of our main result should be applicable to other conjectural mirrors or Fourier-Mukai partners. Section \ref{sec:spreadingout} contains a potentially widely applicable theorem about spreading out a derived equivalence as an isomorphism of complex $K$-theory groups. Our work relies heavily on the following results by other authors:

\begin{itemize}
    \item Complex $K$-theory for DG categories (due to Blanc and Moulinos)
    \item the Decomposition Theorem by Beilinson--Bernstein--Deligne--Gabber,
    \item Arinkin's autoduality for integral spectral curves,
    \item de Cataldo's work on supports for ``$\SL_n$" Higgs bundles,
    \item the vanishing cycle method by Maulik and Junliang Shen.
\end{itemize}

 In addition to the above, we had to extend some known results in a non-trivial way:
 
 \begin{itemize}
    \item Generalisation of the main result of Moulinos's \cite{moulinos} to finite abelian quotient stacks (\ref{ssec:moulinos2})
    \item Extension of Arinkin's derived equivalence to ``$\SL_n$" and $\PGL_n$ Higgs bundles (\ref{ssec: coherent duality})
    \item Formalism of $G$-sheaves of spectra and vanishing cycles, i.e., a relative version of equivariant homotopy theory (\ref{sec:app})
\end{itemize}

Concerning the second item, we remark that a \textit{fibrewise} version of this equivalence was recently constructed by Franco--Hanson--Ruano  (see \cite{FHR}). 

The last item is the content of the appendix. This formalism is required to apply the vanishing cycle computation of Maulik \& Junliang Shen to twisted equivariant complex $K$-theory.

\subsubsection*{Acknowledgements} 
It is a pleasure to thank Tom Baird for sharing his insights about vanishing of torsion for $H^*(\hMm_2^1)$. Furthermore, we would like to thank him for bringing \cite{gs96} to our attention. We also thank Alexander Kupers for enlightening conversations about the Atiyah-Hirzebruch spectral sequence and for pointing out a reference. We are grateful to Victoria Hoskins and Simon Pepin-Lehalleur for sharing a preprint of their recent paper on a motivic version of topological mirror symmetry and for useful remarks on section \ref{sec:torsion}.

\section{Preliminaries}\label{sec:prelim}

\subsection{Sheaves of spectra}\label{ssec:sheaves}

This subsection loosely follows \cite[Section 1.3.1]{sag}. In \emph{loc. cit.} the more general framework of sheaves on $\infty$-topoi is utilised, whereas we will limit our exposition to sheaves on topological spaces.

Recall that \cite[Definition 6.2.2.6]{htt} associates to a topological space (or more generally a site) an $\infty$-topos $\Sheaves(X)$ of space-valued sheaves on $X$. In 6.5.2 of \emph{loc. cit.} the $\infty$-topos $\Sh(X)$ is defined as the hypercompletion of $\Sheaves(X)$. The objects in $\Sh(X)$ are called hypersheaves. We remark that for the underlying topological spaces of complex varieties (with respect to the standard topology), sheaves and hypersheaves agree (see \ref{sssec:hyper}), and we will therefore often drop the prefix hyper.

\begin{rmk}
According to \cite[Proposition 6.5.2.14]{htt}, $\Sh(X)$ is equivalent to the $\infty$-category of fibrant-cofibrant objects in the category of simplicial presheaves on $X$, endowed with the \emph{local model structure}.
\end{rmk}
By definition, there is an embedding
$$\Sh(X) \subset \Sheaves(X).$$
See \cite[Definition 1.3.1.4]{sag} for the definition below.
\begin{definition}
Let $X$ denote a topological space and $\Cc$ an $\infty$-category admitting arbitrary small limits. We denote by $\Sh_{\Cc}(X)$ the $\infty$-category of limit-preserving functors
$$F\colon \Sh(X)^{\op} \to \Cc.$$
We will refer to it as the $\infty$-category of $\Cc$-valued hypersheaves.
\end{definition}

Of particular importance will be the cases where $\Cc$ is the $\infty$-category of spectra $\Spectra$, or the derived $\infty$-category of chain complexes of rational vector spaces $D({\Qb})$.

The stable $\infty$-categories are endowed with natural $t$-structures $\Sheaves_{\Spectra}(X)$ and $\Sh_{\Spectra}(X)$ (see \cite[Proposition 2.1.1.1]{sag}). Before expanding further on the role of those $t$-structures, we alert the reader of a possible source of confusion.
\begin{warning}\label{warning:grading-conventions}
This article uses both homological and cohomological grading conventions when dealing with $t$-structures on stable $\infty$-categories. The former is in line with conventions in stable homotopy theory and the latter with standard usage in algebraic geometry. In practice this means that for a $t$-structure on a stable $\infty$-category $\Cc$, we may choose to work with the induced homological functors
$$\pi_i\colon\Cc \to \Cc^{\heartsuit}$$
or the induced cohomological functors
$$\Hc^i = \pi_{-i}\colon \Cc \to \Cc^{\heartsuit}.$$
\end{warning}

\begin{example}
For the stable $\infty$-category of $\Spectra$, the homological functors $\pi_i$ compute the stable homotopy groups of a spectrum. Likewise, for $\Sp$-valued hypersheaves we also have a system of homological functors $\pi_i\colon \Sh_{\Sp}(X) \to \Sh_{\Sp}(X)^{\heartsuit}=\Sh_{\AbGrp}(X)$.  Those functors associate to $\F \in \Sh_{\Sp}(X)$ the hypersheafification of the presheaf in abelian groups 
$$U \mapsto \pi_i(\F(U)),\text{ where } U\subset X \text{ is open.}$$
\end{example}

\subsection{Rationalisation}\label{ssec:rat}

\begin{lemma}[Lurie]\label{lemma:ShQ}
For a topological space $X$ and a commutative ring $R$ we denote by $\Sheaves_R(X)$ the category of sheaves of $R$-modules on $X$. There is an equivalence of $\infty$-categories $$\D(\Sheaves_R(X)) \cong \Sh_{\D(R)}(X).$$
\end{lemma}
\begin{proof}
This is a special case of \cite[Theorem 2.1.2.2]{sag}.
\end{proof}

\begin{definition}\label{def:rat}
The \emph{rationalisation} of $\F \in \Sh_{\Spectra}(X)$ is defined to be $\F \wedge \Qb \in \Sh_{\D(\Qb)}(X)$.
\end{definition}

Functoriality of the smash product yields the existence of a rationalisation functor
$$\Rat\colon \Sh_{\Spectra}(X) \to \Sh_{\D(\Qb)}(X).$$
We will also write $\Eb_{\Qb}$ to denote $\Rat(\Eb)$.

Recall from \cite[Definition 7.2.2.10]{ha} that a spectrum $M$ is called flat, if $\pi_0(M)$ is a flat $\Zb$-module and for all $n \in \Nb$ the natural map 
$$\pi_n(\Sb) \otimes \pi_0(M) \to \pi_n(M)$$
is an isomorphism. The Eilenberg--MacLane spectrum $\Qb = H\Qb$ is therefore a flat spectrum. It follows from \cite[Proposition 7.2.2.13]{ha} that there is a natural isomorphism
\begin{equation}\label{eqn:pi-rat}
\pi_n(\Rat(\F))\cong \pi_n(\F) \otimes \Qb.
\end{equation}

\begin{lemma}\label{lemma:Rat-KU}
$\Rat(\underline{\KU}_X) \simeq \underline{\Qb}_{X}[\beta,\beta^{-1}] = \bigoplus_{n \in \Zb}\underline{\Qb}_X[2n].$
\end{lemma}
\begin{proof}
Since $\Rat$ commutes with pullback it suffices to prove this for $X=\{\bullet\}$ being a singleton. Identifying $\Sh_{\Spectra}(\{\bullet\})$ with $\Spectra$, we are left to prove that the rationalisation of $\KU$ is ${\Qb}[\beta,\beta^{-1}]$. This follows from the rational degeneration of the Atiyah-Hirzebruch spectral sequence.
\end{proof}

\subsection{Topological $K$-theory for DG categories (\emph{d'apr\`es} Blanc and Moulinos)}\label{ssec:moulinos}

Let $S$ be a separated $\Cb$-scheme of finite type. We denote by $S^{\an}$ the associated analytic space. Furthermore, we will denote the $\infty$-category of $\Perf(S)$-linear stable $\infty$-categories by $\Cat^{\perf}(S)$.

\begin{theorem}[Blanc \cite{blanc}]\label{thm:blanc}
There exists a functor
$$\KtopC\colon \Cat^{\perf}(\Cb) \to \Spectra$$
sending $\Perf(S)$ to $KU(S^{\an})$.
\end{theorem}

This absolute construction for complex-linear DG-categories was extended by Moulinos to a relative set-up in \cite{moulinos}, assigning to a $\Perf(S)$-linear DG-category a sheaf of spectra on $S^{\an}$.

\begin{theorem}[Moulinos]\label{thm:moulinos}
There is a functor 
$$\KtopS\colon \Cat^{\perf}(S) \to \Sh_{\Spectra}(S^{\an}).$$
For $S=\Spec \Cb$ this functor agrees with the one constructed by Blanc.
\end{theorem}

\begin{theorem}[Moulinos]\label{thm:moulinos2}
Let $\alpha$ be a Brauer class on $S$ and $\Perf(S,\alpha)$ the corresponding DG-category of $\alpha$-twisted perfect complexes, endowed with the natural $\Perf(S)$-linear structure. Then, $$\Ktop(\Perf(S,\alpha)) \cong \underline{KU}^{\tilde{\alpha}}(S^{\an}),$$ where $\tilde{\alpha}$ denotes the associated withrsion class in $H^3(S^{\an},\Zb)$ and $\underline{KU}^{\tilde{\alpha}}(S^{\an})$ is the corresponding locally constant hypersheaf of $KU$-modules.
\end{theorem}
\begin{proof}
This is \cite[Theorem 9.1]{moulinos}. 
\end{proof}

\begin{rmk}\label{rmk:KtopDM}
There is currently no counterpart of this result for DM-stacks (even for the untwisted case). According to a result of Halpern-Leistner--Pomerleano \cite{hlp}, $\Ktop_{\Cb}(\Perf([U/G]))$ is equivalent to an equivariant $KU$-spectrum associated with $U^{\an}$. A twisted analogue thereof is produced in \cite{bm19} under additional assumptions of technical nature (see Theorem \ref{thm:brown-moulinos} below). Those results only apply to Blanc's invariant $\Ktop_{\Cb}$ and there's no known extension to a Moulinos's relative topological $K$-theory.
\end{rmk}

\begin{theorem}[Moulinos]\label{thm:moulinos3}
Let $\phi\colon S \to S'$ be proper and $\Cc \in \Cat^{\perf}_S$. Then, $$\Ktop_{S'}(\phi_*\Cc)\cong \phi_*\Ktop_{S}(\Cc),$$ where $\phi_*\Cc$ denotes the same DG-category viewed as a $\Perf(S')$-linear category by restricting along $\phi^*\colon \Perf(S') \to \Perf(S)$.
\end{theorem}

For a non-proper morphism $\phi$, we cannot expect the result above to hold. However, it is often possible to obtain a natural transformation, for instance when $\phi\colon S \to \Spec \Cb$ is smooth (see \cite[Subsection 7.4]{moulinos}).

\begin{proposition}[Moulinos]\label{prop:nat-tr}
Let $\phi\colon S \to \Spec \Cb$ be a smooth morphism of complex varieties. Then, there is a natural transformation 
$$\Ktop_{\Cb} \circ \phi_* \to \phi_* \circ \Ktop_S.$$
\end{proposition}

Utilising the formalism above, Moulinos proved the following topological counterpart of the projective bundle theorem in algebraic $K$-theory (\cite[Theorem 1.3]{moulinos}):

\begin{theorem}[Moulinos]\label{thm:moulinos4}
For a finite CW-complex $X$ and a $\Cb\Pb^n$-bundle $P$ given by a map $X \to B\PGL_{n+1}^{\an}$ there is an equivalence of spectra
$$\KU(P) \simeq \bigoplus_{i=0}^{n}\KU^{\tilde{\alpha}^i}(X),$$
where $\tilde{\alpha}$ is the torsion class in $H^3_{\sing}(X;\Zb)$ given by the composition $X \to B\PGL_{n+1}^{\an} \to K(\Zb,3)$.
\end{theorem}

\begin{theorem}\label{thm:brown-moulinos}
There is a natural transformation (see \cite[Section 5.3]{bm19}).
$$\KU_M^{\tilde{\alpha}}(X^{\an}) \to \Ktop_{\Cb}(\Perf^{\alpha}([X/G])).$$
\begin{itemize}
\item[(a)] (Halpern-Leistner--Pomerleano \cite{hlp}) If $\alpha=0$, this is an equivalence.
\item[(b)] (Brown--Moulinos \cite{bm19}) For $\alpha \neq 0$, the natural transformation is an equivalence if the auxiliary map \cite[Equation (1.6)]{bm19} is an equivalence.
\end{itemize}
\end{theorem}
It will be essential for us to dispose of a relative version of those results. In the subsequent subsection we will resolve this problem to a certain extent.

We conclude this subsection with a lemma which we record for later purpose:

\begin{lemma}\label{lemma:first-row}
Let $\iota\colon Z \subset X$ be an equivariant closed immersion of smooth complex varieties acted on by a finite group $G$. We denote the complement $X \setminus Z$ by $U$. Let $\alpha$ be an equivariant gerbe on $X$. Assume that the quotient $X/G$ is a variety. Then, there is a fibre sequence of spectra
$$\iota_*\Ktop_{Z/G}(\Perf^{\alpha}([Z/G])) \to \Ktop_{X/G}\Perf^{\alpha}([X/G])) \to j_*\Ktop_{U/G}(\Perf^{\alpha}([U/G])),$$
where $j$ denotes the inclusion of $U$.
\end{lemma}
\begin{proof}
Since the argument is analogous to the proof of \cite[Lemma 3.6]{hlp}, we will content ourselves with a sketch. For an \'etale morphism $T \to X/G$ we denote by $Z_T \to X_T \leftarrow U_T$ the $G$-equivariant immersions given by the base change 
$$Z \times_{X/G} T \hookrightarrow X \times_{X/G} T \hookleftarrow U \times_{X/G} T.$$
As in \emph{loc. cit.} we have a fibre sequence in algebraic $K$-theory
$$K(\D_{\coh}^{b,\alpha}([Z_T/G])) \to K((\D_{\coh}^{b,\alpha}([X_T/G])) \to K((\D_{\coh}^{b,\alpha}([U_T/G])).$$
Applying the exact functor 
$L_{\KU}\circ \mathsf{An}^*$ we therefore obtain a fibre sequence
$$\iota_*\Ktop_{Z/G}(\D_{\coh}^{b,\alpha}([Z/G])) \to \Ktop_{X/G}(\D_{\coh}^{b,\alpha}([X/G])) \to j_*\Ktop_{U/G}(\D_{\coh}^{b,\alpha}([U/G])).$$
By assumption, $X$, $Z$, and $U$ are smooth, and $T \to X/G$ is \'etale. Therefore, we have $\D_{\coh}^{b,\alpha}([X_T/G])\cong \Perf^{\alpha}([X_T/G])$ (and similarly for $Z_T$ and $U_T$). This concludes the proof.
\end{proof}

\subsection{Relative twisted topological $K$-theory for smooth abelian quotient stacks}\label{ssec:moulinos2}

Throughout this subsection we will assume that we are in the following situation.

\begin{situation}\label{sit0}
Let $U$ be a smooth complex variety, acted on by a finite abelian group $G$. We assume that the quotient $U/G$ is a scheme. The quotient map $U \to U/G$ will be denoted by $\pi$. Pullback along $\pi$ gives rise to a $\Perf(U/G)$-linear structure on $\Cc^{\alpha}_{U;G}=\Perf^{\alpha}([U/G])$.
\end{situation}

We will denote the hypersheaf of spectra $\Ktop_{U/G}(\Cc^{\alpha}_{U;G})$ by $\uKtop(U,\alpha;G)$. We intend to relate it to 
$\underline{\KU}^{\alpha}_{U;G},$
the hypersheaf of $G$-equivariant $\KU$-theory on $U^{\an}/G$. Informally speaking, it assigns to an open subset $V \subset U^{\an}/G$ the $G$-equivariant $\KU$-theory spectrum $\KU_{G}(\pi^{-1}(V))$. Using the formalism of $G$-sheaves of spectra developed in the appendix, this sheaf is given by $$\pi_*^G p^{-1}_G\KU_G,$$
where $\KU_G$ denotes the $G$-spectrum representing $G$-equivariant $K$-theory, identified with a $G$-sheaf on the singleton $\{\bullet\}$ and $p\colon U \to \{\bullet\}$ denotes the constant map.

At first we need to introduce a comparison map between these two sheaves of spectra.

\begin{construction}\label{construction:comparison-map}
By replacing $G$-equivariant vector bundles (algebraic and topological) in \cite[Section 5.4]{thomason} by their twisted counterparts, one obtains a \emph{comparison map} from $\alpha$-twisted $G$-equivariant $K$-theory of a complex $G$-variety $V$ to $\tilde{\alpha}$-twisted $G$-equivariant complex $K$-theory of $V^{\an}$. This map assigns to an $\alpha$-twisted $G$-equivariant vector bundle $E$ on $V$, the associated $\tilde{\alpha}$-twisted $G$-equivariant topological vector bundle $E^{\top}$. Since this assignment is compatible with direct sums, we obtain a map between $K$-theory spectra as in \emph{loc. cit.}

Using the notation from \cite[Section 7]{moulinos}, this yields a morphism of sheaves of spectra on the \'etale site of $U/G$
$$\underline{K}^{\text{\'et}}_G \to \mathsf{An}_*\underline{\KU}^{\alpha}_{U;G}.$$
The adjunction of $\mathsf{An}_*$ and $\mathsf{An}_*$ allows us to rewrite this as a morphism
$$\mathsf{An}^*\underline{K}^{\text{\'et}}_G \to \underline{\KU}^{\alpha}_{U;G}.$$
The right-hand side is already a $\KU$-module. Applying the universal property of $\KU$-localisation $L_{\KU}$, we obtain a morphism
$$\Theta_{U;G}^{\alpha}\colon \uKtop(U,\alpha;G)= L_{\KU}\mathsf{An}^*\underline{K}^{\text{\'et}}_G \to \underline{\KU}^{\alpha}_{U;G},$$
where we used Moulinos's definition of topological $K$-theory (see \cite[Definition 7.5]{moulinos}).
\end{construction}

\subsection{Equivariant gerbes and commutativity}

As in Situation \ref{sit0} we assume that $G$ is a finite abelian group acting trivially on $U$, and that $\alpha$ is a $G$-equivariant gerbe on $U$. As explained in \cite[Section 4]{ht}, the $G$-equivariant structure on $U$ yields for every $g \in G$ a line bundle $L_g^{\alpha}$. Indeed, $g^*\alpha=\alpha$, since $G$ acts trivially. The $G$-equivariant structure on $\alpha$ therefore associated with every $g \in G$ an automorphism of the gerbe $\alpha$. Such an automorphism corresponds to a line bundle $L_g^{\alpha}$.

We observe that for every pair $g,h \in G$ there is a canonical isomorphism
$$m_{g,h}\colon L^{\alpha}_g \otimes L^{\alpha}_{h} \xrightarrow{\simeq} L_{gh}^{\alpha}.$$
For every triple $(g,h,k)$, a coherence condition is satisfied, which is captured by the following diagram
\begin{equation}\label{eqn:associative}
\xymatrix{
L_g^{\alpha} \otimes L_h^{\alpha} \otimes L_k^{\alpha} \ar[r]^-{\id \otimes m_{h,k}} \ar[d]_{m_{g,h}\otimes \id} & L_g^{\alpha} \otimes L_{hk}^{\alpha} \ar[d]^{m_{g,hk}} \\
L_{gh}^{\alpha} \otimes L_k^{\alpha} \ar[r]^-{m_{gh,k}} & L_{ghk}^{\alpha}
}
\end{equation}
which is required to commute. Denoting by $L^{\times} = L \setminus \{\text{zero section}\}$ the $\G_m$-torsor associated with a line bundle, we let
$$\widetilde{G}=\bigsqcup_{g \in G} \big(L^{\alpha}_g\big)^{\times},$$
and endow it with the group structure induced by the maps $m_{g,h}\colon (L_g^{\alpha})^{\times} \times (L_h^{\alpha})^{\times} \to (L_{gh}^{\alpha})^{\times}$. By construction, there is a surjective homomorphism $\varepsilon\colon \tilde{G} \to G$ with kernel $U \times \G_m=(L_e^{\alpha})^{\times}$. The kernel lies in the centre of $\widetilde{G}$. We therefore have a central extension
\begin{equation}\label{eqn:central-extension}
1 \to \G_m \to \widetilde{G}^{\alpha} \xrightarrow{\varepsilon} G \to 1,
\end{equation}
such that for every $g \in G$, the fibre $\varepsilon^{-1}(g)$ is given by the $\G_m$-torsor $(L^{\alpha}_g)^{\times}$.

Commutativity of \eqref{eqn:associative} also implies that the locally free sheaf $\bigoplus_{g \in G} L^{\alpha}_g$ is endowed with an associative binary operation, and therefore can be viewed as a sheaf of associative algebras. 

\begin{definition}\label{defi:alpha-commutative}
We say that the $G$-equivariant structure on $\alpha$ is \emph{commutative}, if for every pair of elements $(g,h)\in G$ the square
\[
\xymatrix{
L^{\alpha}_g \otimes L_h^{\alpha} \ar[r] \ar[d] & L_{gh}^{\alpha} \ar@{=}[d] \\
L_h^{\alpha} \otimes L_g^{\alpha} \ar[r] & L_{hg}^{\alpha}
}
\] 
commutes.
\end{definition}

This property amounts to the central extension \eqref{eqn:central-extension} being commutative.
  
\begin{definition}\label{defi:Calpha} Assume that $\alpha$ is a commutative $G$-equivariant gerbe in the sense of Definition \ref{defi:alpha-commutative}. 
\begin{itemize}
\item[(a)] We define the \'etale cover $\tau\colon C^{\alpha}_{U;G} \xrightarrow{\tau} U$ as the relative spectrum $$C^{\alpha}_{U;G}=\Spec_U \big(\bigoplus_{g \in G} L_g^{\alpha}\big).$$
\item[(b)] Likewise, in the topological case we let $C^{\alpha}_{U;G}$ be the maximal spectrum of the ring of continuous sections of 
$$\big(\bigoplus_{g \in G} L_g^{\alpha}\big)$$ vanishing at infinity. By virtue of construction, $C^{\alpha}_{U;G}$ is a finite covering space.
\end{itemize}
\end{definition}

As explained in \cite[Section 4]{ht}, the line bundles $L^{\alpha}_g$ are $G$-equivariant. We therefore obtain a $G$-equivariant structure on 
$$\tau_*\Oo_{C^{\alpha}_{U;G}} = \big(\bigoplus_{g \in G} L_g^{\alpha}\big),$$
which induces a $G$-action on the total space $C^{\alpha}_{U;G}$ such that $\tau$ is a $G$-invariant map. We call this action the \emph{tautological} $G$-action.

\begin{lemma}\label{lemma:action-trivial}
Under the same assumptions as in Definition \ref{defi:Calpha} we have that the tautological $G$-action on $C^{\alpha}_G$ is trivial.
\end{lemma}
\begin{proof}
We fix a pair of elements $(g,h)$ in $G$. The $G$-equivariant structure on $L_g^{\alpha}$ yields a natural isomorphism
$L_g^{\alpha}=h^*L_g^{\alpha} \simeq L^{\alpha}_g$, where we use that $G$ acts trivially on $U$. The induced action is the trivial one.
\end{proof}

Cover $U$ by open subsets $(U_i)_{i \in I}$, over which the $\G_m$-torsors $(L_g^{\alpha})^{\times}$ and $(L_h^{\alpha})^{\times}$ can be trivialised. For every $i \in I$ we choose sections $\tilde{g}_i$ and $\tilde{h}_i$ of $(L_g^{\alpha})^{\times}$, respectively $(L_h^{\alpha})^{\times}$.

Centrality of \eqref{eqn:central-extension} implies the following:
\begin{claim}
The commutator $[\tilde{g}_i,\tilde{h}_i]$ is independent of the chosen lifts.
\end{claim}

This implies that there is a well-defined \emph{commutator pairing}
$$(-,-)\colon G \times G \to \Oo_U^\times,$$
which is locally given by the formula
$$(g,h)|_{U_i} = [\tilde{g}_i,\tilde{h}_i],$$
where we use the same notation as above. The central extension $\widetilde{G}$ is thus commutative, if and only if $(g,h)=1$ for all pairs $(g,h)$ of elements in $G$.

\begin{lemma}\label{lemma:cyclic}
If $G$ is a cyclic group then every $G$-equivariant gerbe $\alpha$ is commutative.
\end{lemma}
\begin{proof}
As remarked above, it suffices to prove that $\widetilde{G}=\widetilde{G}^{\alpha}$ is commutative. This is equivalent to the commutator pairing $(g,h)$ being the trivial pairing. To see this, observe that there exists a generator $G=\langle x \rangle$, and thus $g=x^i$ and $h=x^j$. Let $\tilde{x} \in \varepsilon^{-1}(x)$ be a lift of $x$. We then have lifts $\tilde{x}^i$ of $g$ and $\tilde{x}^j$ of $h$. It is clear that the commutator $[\tilde{x}^i,\tilde{x}^j]$ equals $1$. This concludes the proof. 
\end{proof}

The definition of the morphism $\tau$ and the cover $C^{\alpha}$ appearing in the following lemma can be found in Definition \ref{defi:Calpha}(a).

\begin{lemma}\label{lemma:induction-step}
Assume that $G$ is a finite abelian group which acts trivially on $U$. Let $\alpha$ be a $G$-equivariant $G$-gerbe on $U$, and assume that $G'\subset G$ is a subgroup such that the induced $G'$-equivariant structure on $\alpha$ is commutative. We denote the quotient group $G/G'$ by $Q$.
Then, there are canonical equivalences of sheaves of spectra
\begin{itemize}
\item[(a)] $\uKtop(U,\alpha;G) \simeq \Ktop_{U/G}(\Perf^{\tau^*\alpha}([\C^{\alpha}/Q])$ 
\item[(b)] $\underline{\KU}^{\alpha}_{U;G} \simeq \tau_*\underline{\KU}_{C^{\alpha};Q}^{\tau^*\alpha}$
\end{itemize}
such that the following diagram of sheaves of spectra
\[
\xymatrix{
\uKtop(U,\alpha;G) \ar[r]^-{\simeq} \ar[d]_{\Theta_{U;G}^{\alpha}} & \Ktop_{U/G}(\Perf^{\tau^*\alpha}([\C^{\alpha}/Q]) \ar[d]^{\Theta^{\tau^*\alpha}_{C^{\alpha};Q}} \\
\underline{\KU}^{\alpha}_{U;G} \ar[r]^-{\simeq} & \tau_*\underline{\KU}_{C^{\alpha};Q}^{\tau^*\alpha}
}
\]
commutes.
\end{lemma}
\begin{proof}
By virtue of definition, an $\alpha$-twisted $G$-equivariant vector bundle $E$ on $U$ corresponds the data $(\phi_g)_{g \in G}$ where each $\phi_g$ is
$\text{an isomorphism }\phi_g\colon L_g^{\alpha} \otimes E \to E$, and for $g,h \in G$ we have a commutative diagram
\[
\xymatrix{
 L^{\alpha}_g \otimes L^{\alpha}_h \otimes E \ar[r]^-{\phi_h} \ar[d]_-{m_{g,h}\otimes \id} & L_g^{\alpha} \otimes E \ar[d]^-{\phi_g} \\ 
L_{gh}^{\alpha} \otimes E \ar[r]^-{\phi_{gh}} & E.
}
\]
This description holds in the algebraic and the topological category, mutatis mutandis.

For the purpose of this proof we will not distinguish between a vector bundle and the corresponding locally free sheaf of sections. The properties of the maps $(\phi_g)_{g \in G'}$ outlined above, imply that the locally free sheaf $E$ is endowed with the structure of the commutative sheaf of algebras
$\bigoplus_{g \in G'} L_g^{\alpha},$
which commutes with the $G$-action. 
This actions allows us to identify $E$ with an $\tau^*\alpha$-twisted locally free sheaf on the finite \'etale cover $\tau\colon C^{\alpha}_{U;G'} \to U$, endowed with a $Q$-equivariant structure. Here, we use that $G'$ acts trivially on $C^{\alpha}_{U;G}$ according to Lemma \ref{lemma:action-trivial}.

The same reasoning can be applied to topological vector bundles, where we use that by virtue of the Serre-Swan theorem, a topological $\Cb$-vector bundle can be identified with the module over the ring of continuous complex-valued functions.

These equivalences are compatible with direct sums and restriction to open subsets, and therefore can be promoted to equivalences of sheaves of $K$-theory spectra. There are also naturally compatible with the comparison maps of Construction \ref{construction:comparison-map}.
\end{proof}

We will now promote this to a more general statement.

\begin{proposition}
Suppose that we are in Situation \ref{sit0}. Then, the comparison map $$\Theta^{\alpha}_{U;G}\colon \uKtop(U,\alpha;G) \to \underline{\KU}^{\alpha}_{U;G}$$ of Construction \ref{construction:comparison-map} is an equivalence.
\end{proposition}
\begin{proof}
We will show this by nested induction. Let us assume that the assertion is already known for groups of order $< |G|$. The base case where the group is trivial, is covered by Theorem \ref{thm:moulinos2}.

We now choose an ascending sequence of open subvarieties
$$\emptyset=U_{-1} \subset U_0 \subset \cdots \subset U_s = U$$
such that for every $i$ we have a smooth complement $Z_i=U_i\setminus U_{i-1}$ on which $G$ acts with constant stabiliser group $G_i$ and such that $Z_i$ is closed in $U_i$. We will prove by induction on $i$ that the comparison map $\Theta_{U_i;G}^{\alpha}$ is an equivalence. The base case is $i=-1$, where there is nothing to prove ($U_{-1}=\emptyset$).

We denote the closed immersion of $Z_i \hookrightarrow U_i$ by $\iota_i$ and the open inclusion $U_{i-1} \to U_i$ by $j_i$. There is a homotopy commutative diagram of distinguished triangles of sheaves of spectra:
\[
\xymatrix{
(\iota_{i})_*\uKtop(Z_i,\alpha;G) \ar[r] \ar[d]_{\Theta_{Z_i;G}^{\alpha}} & \uKtop(U_i,\alpha;G) \ar[r] \ar[d]_{\Theta_{U_i;G}^{\alpha}} & (j_i)_*\uKtop(U_{i-1},\alpha;G) \ar[r] \ar[d]_{\Theta_{U_{i-1};G}^{\alpha}} & \Sigma \uKtop(Z_i,\alpha;G) \ar[d]_{\Sigma\Theta_{Z_i;G}^{\alpha}} \\
(\iota_i)_*\underline{\KU}^{\alpha}_{Z_i;G} \ar[r] & \underline{\KU}^{\alpha}_{U_i;G} \ar[r] & (j_i)_*\underline{\KU}^{\alpha}_{U_{i-1};G} \ar[r] & \Sigma \underline{\KU}^{\alpha}_{Z_i;G} \\
}
\]
The first row stems from Lemma \ref{lemma:first-row} and the second row is given by the distinguished triangle associated with the pair $(U_i,Z_i)$:
$$\underline{\KU}^{\tilde{\alpha}}_{G}(U_i^{\an},Z^{\an}_i) \to \underline{\KU}_{U^{\an}_i;G}^{\tilde{\alpha}} \to \underline{\KU}^{\tilde{\alpha}}_{U^{\an}_i;G} \to \Sigma \underline{\KU}^{\tilde{\alpha}}_{G}(U^{\an}_i,Z^{\an}_i).$$
The left-most term is equivalent to $(\iota_i)_*\underline{\KU}^{\alpha}_{Z_i^{\an};G}$ by virtue of the equivariant Thom isomorphism (see \cite[Theorem 8.1]{KM}). The last step requires smoothness of $Z_i$.

\begin{claim} 
The morphism $\Theta_{Z_{i};G}^G$ is an equivalence.
\end{claim}
To prove the claim we choose an arbitrary cyclic subgroup $G' \subset G$. We then have according to Lemma \ref{lemma:cyclic} that $\alpha$ is a commutative $G$-equivariant gerbe. We can now apply Lemma \ref{lemma:induction-step}, and replace $G$ by $Q=G/G'$. The outer induction hypothesis now implies validity of the claim.

Furthermore, by inner induction, we may assume that $\Theta_{U_{i-1};G}^G$ is an equivalence. It follows from the 2-out-of-3 rule (or the five lemma) that the middle morphism $\Theta_{U_i;G}^{\alpha}$ is an equivalence of sheaves of spectra. This concludes the proof.
\end{proof}

\section{Spreading out twisted derived equivalences in topological $K$-theory}\label{sec:spreadingout}

\subsection{The decomposition theorem for $\KU_{\Qb}$}

Let $\Xc$ be a stack. We recall that $I\Xc$ denotes the inertia stack
$$I\Xc = \Xc \times_{\Xc \times \Xc} \Xc.$$
For a presentation $\Xc = [U/G]$ as a finite quotient stack, we have 
$$I\Xc = \bigsqcup_{[g] \in G/\text{conj}} [X^g/C_g].$$

\begin{proposition}\label{prop:decthm}
Let $p\colon \Xc \to B$ be a proper morphism from a smooth finite quotient stack $\Xc$ to a $\Cb$-variety $B$, and $\alpha$ a Brauer class on $\Xc$. Then, there exist semi-simple perverse $\Qb$-sheaves $P_{\Xc}^{\alpha}$ and $Q_{\Xc}^{\alpha}$ on $B^{\an}$ such that 
$$p_*(\underline{\KU}^{\tilde{\alpha}}_{\Xc^{\an}})_{\Qb} \simeq \bigoplus_{n\in \Zb} \left( P_{\Xc}^\alpha[2n] \oplus Q_{\Xc}^{\alpha}[2n+1] \right).$$
\end{proposition}
\begin{proof}
It suffices to prove that the theorem holds after tensoring the complexes of sheaves with $\Cb$. According to \cite[Theorem 3.9]{fht}, complexified $\alpha$-twisted $G$-equivariant topological $K$-theory of $\Xc = [U/G]$ is equivalent to 
$$\KU^{\alpha,*}_G(U^{\an}) = \bigoplus_{n\in \Zb}\bigoplus_{g \in G/\text{conj}} H^{*+2n}_{C(g)}\big((U^{\an})^g,\Lc^{\alpha}\big),$$
where $\Lc^{\alpha}$ is a $C(g)$-equivariant complex local system on $(U^{\an})^g$ of rank $1$. Furthermore, since $\alpha$ has finite order, so do the equivariant sheaves $\Lc^{\alpha}$.

We therefore have an equivalence of sheaves 
\begin{equation}\label{eqn:Qbeta}
p_*\big((\underline{\KU}_{I\Xc})_{\Qb}\big) \otimes \Cb \simeq p_*\Lc^{\alpha}[\beta,\beta^{-1}].
\end{equation}
This concludes the proof, as $p_*\Lc^{\alpha}[\beta,\beta^{-1}] \simeq \bigoplus_{n \in \Zb} p_*\Lc^{\alpha}$, and $p_*\Lc^{\alpha}$ is a direct sum of shifts of $\Cb$-linear semi-simple perverse sheaves by virtue of the decomposition theorem. The decomposition theorem can be applied to $\Lc^{\alpha}$ since it has finite monodromy, and therefore is of geometric origin.
\end{proof}

\begin{porism}\label{porism}
There exists a complex local system $\Lc^{\alpha}_{\Xc}$ of rank $1$ on $I\Xc$ such that
the complexified perverse sheaves $P_{\Xc}^{\alpha}$ and $Q_{\Xc}^{\alpha}$ are isomorphic to
$$P_{\Xc}^{\alpha} \simeq \bigoplus_{n \in \Nb}\Hc^{2n}_p(p_*\Lc^{\alpha}_{I\Xc}) \text{ and } Q_{\Xc}^{\alpha} \simeq \bigoplus_{n \in \Nb}\Hc^{2n+1}_p(p_*\Lc^{\alpha}_{I\Xc}).$$
\end{porism}
\begin{proof}
This follows from \eqref{eqn:Qbeta} and the fact that $\beta$ has cohomological degree $-2$.
\end{proof}

We record the following result for later use.
\begin{lemma}\label{lemma:p-retraction}
Let $f\colon \mathsf{F} \to \mathsf{G}$ and $g\colon \mathsf{G} \to \mathsf{F}$ be morphisms of complexes of constructible sheaves of $\Qb$-vector spaces on a variety $B$. Assume that 
\begin{itemize}
\item $\mathsf{F}$ is a direct sum $\bigoplus_{i \in \Zb} T_i[-i]$ of shifts of semisimple perverse sheaves,
\item for every $i \in \Zb$ the support of $T_i$ intersects a fixed open subvariety $B^{\heartsuit}$,
\item $(g \circ f)|_{B^{\heartsuit}} \simeq \id_{\mathsf{F}|_{B^{\heartsuit}}}$. 
\end{itemize}
Then, the composition 
\[
\xymatrix{
\Gamma(B,\mathsf{F}) \ar[r]^{f^{\Gamma}_{\Qb}} \ar[rd] & \Gamma(B,\mathsf{G}) \ar[d]^{g^{\Gamma}_{\Qb}} \\
& \Gamma(B,\mathsf{F})
}
\] 
is a unipotent endomorphism. In particular, $\Gamma(B,\mathsf{F})$ is a retract of $\Gamma(B,\mathsf{G})$.
\end{lemma}
\begin{proof}
By construction, the composition of morphisms above is the identity when restricted to $\Bc^{\heartsuit}$. Since the perverse sheaves $(T_i)_{i \in \Zb}$ are supported on $B^{\heartsuit}$, we may conclude that 
the map above is an equivalence everywhere, and that it induces the identity map on the associated graded groups (with respect to the perverse filtration)
$$\gr_p^j\Hb^k(B,\mathsf{F}) \to \gr_p^j\Hb^k(B,\mathsf{F}),\text{ where $j,k \in \Zb$,}$$
Therefore, this map is unipotent.
\end{proof}

\subsection{Main construction}\label{ssec:first}

\begin{situation}\label{sit1}
Assume that we have the following diagram of smooth DM stacks over $\Cb$
\begin{equation}\label{diag:cube}
\xymatrix{
 & \Xc^{\heartsuit} \times_{\Bc^{\heartsuit}} \Yc^{\heartsuit} \ar@{^(.>}@[gray][d] \ar[ld]_{q'} \ar[rd]^{p'} & \\
\Xc^{\heartsuit} \ar@{^(->}[d] \ar[rd]_{p} & {\color{gray}\Xc \times_{\Bc} \Yc}  \ar@{.>}@[gray][rd] \ar@{.>}@[gray][ld] & \Yc^{\heartsuit} \ar[ld]^q \ar@{^(->}[d] \\
\Xc  \ar[rd] & \Bc^{\heartsuit} \ar@{^(->}[d] & \Yc \ar[ld] \\
& \Bc &
}
\end{equation}
such that $p$ and $q$ are proper morphisms and $\Bc^{\heartsuit} \subset \Bc$ is an open subset. Let $\alpha \in \Br(\Xc)$ and $\beta \in \Br(\Yc)$ be Brauer classes, and $\Kk \in \D^b_{\coh}(\Xc \times_{\Bc} \Yc,\alpha^{-1}\boxtimes \beta)$ be an $\alpha^{-1} \boxtimes \beta$-twisted bounded complex of coherent sheaves. The restriction to $\Xc^{\heartsuit} \times_{\Bc^{\heartsuit}} \Yc^{\heartsuit}$ is denoted by $\Kk^{\heartsuit}$. Furthermore, we assume that $\Xc$ and $\Yc$ are finite quotient stacks. The course moduli space of $\Xc$ (respectively $\Yc$) will be denoted by $X$ (respectively $Y$).
\end{situation}

In section \ref{sec:app}, we start with a twisted perfect complex $\Kk^{\heartsuit}$ and choose any extension $\Kk$. This is possible by virtue of the following lemma.

\begin{lemma}
Let $W$ be a Noetherian scheme acted on by a finite group $G$, $\alpha$ a $G$-equivariant Brauer class, $j\colon U \hookrightarrow W$ a $G$-invariant open subscheme. Then, an object $\F \in \D^b_{\coh}([U/G],\alpha)$ can be extended to a complex $\overline{\F}^{*} \in \D^b_{\coh}([W/G],\alpha)$.
\end{lemma}
\begin{proof}
We assume first $\alpha = 0$. In this case, we may represent $\F^{*}$ by a bounded complex of coherent sheaves on $U$:
$$0 \to \F^a \xrightarrow{d^a} \F^{a+1} \to \cdots \to \F^{b-1} \xrightarrow{d^{b-1}} \F^b \to 0.$$
According to \cite[Proposition 15.4]{LMB}, the quasi-coherent sheaf $j_*\F^b$ is equal to the union of its subsheaves of finite type. We can therefore find a coherent extension $\overline{\F}^b$ to $W$ of $\F^b$.
 Applying the same reference again, we can find an extension $\overline{\F}^{b-1}$ of $\F^{b-1}$ together with a morphism $$\overline{d}^{b-1}\colon \F^{b-1} \to \F^b$$
extending $d^{b-1}$. For the same reason, there is a morphism of coherent sheaves on $W$
$$\widetilde{\overline{\F}}^{b-2} \xrightarrow{\overline{d}^{b-2}} \overline{\F}^{b-1} \xrightarrow{\overline{d}^{b-1}} \overline{\F}^b,$$
extending $d^{b-2}$. We define
$${\overline{\F}}^{b-2} = \widetilde{\overline{\F}}^{b-2} / \ker(\overline{d}^{b-1} \circ \overline{d}^{b-2}).$$
This is still an extension of $\F^{b-2}$, and by construction we have $\overline{d}^{b-1} \circ \overline{d}^{b-2} = 0$. We continue like this by descending induction.

For the proof of the twisted case one argues as above, but with sheaves of $\Oo$-modules replaced by $\Aa$-modules, where $\Aa$ is an Azumaya algebra representing the Brauer class $\alpha$.
\end{proof}

\begin{lemma}\label{lemma:FM}
We use the notation from Situation \ref{sit1}. Let $\F \in \Perf^{\alpha}(X)$ be an $\alpha$-twisted perfect complex. Then, 
\begin{equation}\label{eqn:FM}p'_*(q'^*\F \otimes \Kk)\end{equation}
is a $\beta$-twisted perfect complex on $\Yc$.
\end{lemma}
\begin{proof}

Let us assume first that $\alpha=\beta=0$. The tensor product $q'^*\F \otimes \Kk$ is a perfect complex on $\Xc \times \Yc$ (since every bounded complex of coherent sheaves on a smooth variety is perfect) which is set-theoretically supported on $\Xc \times_{\Bc} \Yc$. Since the map $p'\colon \Xc \times_{\Bc} \Yc \to \Yc$ is proper, and $\Xc \times \Yc \to \Xc$ is flat, $p'_*(q'^*\F \otimes \Kk)$ is perfect by virtue of \cite[Proposition 2.7]{tt90}. 
In the twisted case we may argue similarly by replacing perfect complexes of $\Oc$-modules by perfect complexes of $\Aa$-modules, where $\Aa$ denotes an Azumaya algebra representing the Brauer class $\alpha$.
\end{proof}

\begin{definition}\label{def:FM}
We denote by $\FM_{\Kk}\colon \Perf^{\alpha}(\Xc) \to \Perf^{\beta}(\Yc)$ the $\Perf(\Bc)$-linear functor which is informally given by \eqref{eqn:FM}. Formally, it is defined as a composition of pullback $p'^*$, tensoring with $\Kk$, and pushforward $p'_*$.
\end{definition}

$\Perf(\Bc)$-linearity of $\FM_{\Kk}$ follows from the projection formula. We refer the reader to \cite[Theorem 4.14]{bfn} for a more detailed account.

\begin{situation}\label{sit2}
We use the same notation as in Situation \ref{sit1}. Assume that the following holds:
\begin{enumerate}
\item[(a)] $\FM_{\Kk^{\heartsuit}}\colon \Perf^{\alpha}(\Xc^{\heartsuit}) \xrightarrow{\cong} \Perf^{\beta}(\Yc^{\heartsuit})$ is an equivalence,
\item[(b)] for every $i \in \Nb$ the perverse sheaves $\Hc_p^i(p_*\underline{\Qb}_{I\Xc})$ and $\Hc_p^i(p_*\underline{\Qb}_{I\Yc})$ are middle extensions of perverse sheaves on $\Bc^{\heartsuit}$.
\end{enumerate}
\end{situation}

The decomposition theorem \cite{bbdg} guarantees that $\Hc_p^i(p_*\underline{\Qb}_{I\Xc})$ and $\Hc_p^i(p_*\underline{\Qb}_{I\Yc})$ are semi-simple perverse sheaves. The importance of (b) lies in stipulating that the supports of all irreducible factors intersect non-trivially with $\Bc^{\heartsuit}$.

\begin{construction}\label{construction}
Applying the functor $\Ktop_{\Bc}$ to the $\Perf(\Bc)$-linear exact functor $\FM_{\Kk}\colon \Perf^{\alpha}(\Xc) \to \Perf^{\beta}(\Yc)$  we obtain a morphism of sheaves of spectra on $\Bc^{\an}$
$$\Ktop_{\Bc}(\Perf^{\alpha}(\Xc)) \to \Ktop_{\Bc}(\Perf^{\beta}(\Yc).$$
Since $p$ and $q$ are proper, we may apply Theorem \ref{thm:moulinos3} (due to Moulinos) and rewrite this as a morphism
$$p_*\Ktop_{X}(\Perf^{\alpha}(\Xc)) \to q_*\Ktop_{Y}(\Perf^{\beta}(\Yc).$$
By Theorem \ref{thm:moulinos2} (again due to Moulinos) this amounts to a morphism
\begin{equation}\label{eqn:FM-KU-sheaves}p_*\underline{\KU}^{\tilde{\alpha}}_{\Xc^{\an}} \to q_*\underline{\KU}^{\tilde{\beta}}_{\Yc^{\an}}.\end{equation}
Taking global sections, we obtain a morphism of spectra
\begin{equation}\label{eqn:FM-KU}\KU^{\tilde{\alpha}}(\Xc^{\an}) \to \KU^{\tilde{\beta}}(\Yc^{\an}).\end{equation}
\end{construction}

We are now ready to state one of our main results.

\begin{theorem}\label{thm:main}
Suppose that we are in Situation \ref{sit2}. Then, \eqref{eqn:FM-KU} gives rise to an isomorphism of rationalised topological $K$-theory:
$$\KU^{{\alpha}}(\Xc^{\an})_{\Qb} \simeq \KU^{\beta}(\Yc^{\an})_{\Qb}.$$ 
Furthermore, \eqref{eqn:FM-KU} induces an isomorphism of the torsion-free parts of the abelian groups
$$\KU^{{\alpha,*}}(\Xc^{\an})/{\text{torsion}} \simeq \KU^{\beta,*}(\Yc^{\an})/\text{torsion}.$$
\end{theorem}

\begin{proof}
We apply the rationalisation functor $\Rat$ to \eqref{eqn:FM-KU-sheaves} and obtain a morphism 
\begin{equation}\label{eqn:FM-KU-Sh-Q}(\underline{\KU}^{\tilde{\alpha}}_{\Xc})_{\Qb} \to (\underline{\KU}^{\tilde{\beta}}_{\Yc})_{\Qb}.\end{equation}
in $\D_{\const}(\Qb)$ by virtue of Lemma \ref{lemma:ShQ}.

Furthermore, by assumption (a) in Situation \ref{sit1}, the map \eqref{eqn:FM-KU-Sh-Q} is an equivalence when restricted to $\Bc^{\heartsuit}$. Since the perverse sheaves $P^{\alpha}_{\Xc}$, $P^{\beta}_{\Yc}$, $Q^{\alpha}_{\Xc}$, and $Q^{\beta}_{\Yc}$ of Proposition \ref{prop:decthm} are supported on $\Bc^{\heartsuit}$, we may conclude that 
\eqref{eqn:FM-KU-Sh-Q} is an equivalence everywhere.

By assumption, $\FM_{\Kk^{\heartsuit}}$ is an equivalence. For this reason, there exists 
$$\widetilde{\Kk}^{\heartsuit} \in \Perf_{\Yc^{\heartsuit} \times_{\Bc} \Xc^{\heartsuit}}(\Yc^{\heartsuit} \times \Xc^{\heartsuit})$$
such that $\FM_{\widetilde{\Kk}^{\heartsuit}}$ is inverse to $\FM_{\Kk^{\heartsuit}}$. The existence of $\widetilde{\Kk}^{\heartsuit}$ as a complex of quasi-coherent sheaves is guaranteed by \cite[Corollary 4.10]{bfn}. It is a perfect complex, since it is the Fourier-Mukai kernel of a derived equivalence between smooth DM stacks.

We now have the following commutative diagram of $\Zb/2\Zb$-graded abelian groups
\[
\xymatrix{
\KU^{{\alpha,*}}(\Xc^{\an})/{\text{torsion}} \ar[r]^f \ar[d] & \KU^{\beta,*}(\Yc^{\an})/\text{torsion} \ar[d] \ar[r]^g & \KU^{{\alpha,*}}(\Xc^{\an})/{\text{torsion}} \ar[d] \\
\KU^{{\alpha,*}}(\Xc^{\an})_{\Qb} \ar[r]^{f_{\Qb}} &  \KU^{\beta,*}(\Yc^{\an})_{\Qb} \ar[r]^{g_{\Qb}} & \KU^{{\alpha,*}}(\Xc^{\an})_{\Qb}
}
\]

It follows from Lemma \ref{lemma:p-retraction} that $g_{\Qb} \circ f_{\Qb}$ is a unipotent endomorphism of $(\KU^{\tilde{\alpha},*}(\Xc^{\an}))_{\Qb}$.
This implies that $g \circ f\colon \KU^{{\alpha,*}}(\Xc^{\an})/{\text{torsion}} \to \KU^{{\alpha,*}}(\Xc^{\an})/{\text{torsion}}$ is a unipotent map of abelian groups. We infer that $g \circ f$ is invertible. This implies that $f$ is injective. Applying the same argument to $f \circ g$ we obtain surjectivity of $f$. This shows that $f$ is an isomorphism of free abelian groups.
\end{proof}

\begin{corollary}\label{cor:torsion-free1}
Assume that we are in Situation \ref{sit2} and that $\KU^{\alpha,*}(\Xc^{\an})$ and $\KU^{\beta,*}(\Yc^{\an})$ are torsion-free. Then, \eqref{eqn:FM-KU} is an equivalence of spectra
$$\KU^{\alpha}(\Xc^{\an})\simeq \KU^{\beta}(\Yc^{\an}).$$
\end{corollary}
\begin{proof}
By Whitehead's lemma, \eqref{eqn:FM-KU} is an equivalence if and only if the map of $\Zb/2\Zb$-graded abelian groups
$$f\colon \KU^{\tilde{\alpha},*}(\Xc^{\an}) \to \KU^{\tilde{\beta},*}(\Yc^{\an}).$$
is an isomorphism. We know from Theorem \ref{thm:main} that $f_{\Qb}$ is an isomorphism and that both sides have vanishing torsion. Therefore, $f$ must be an isomorphism.
\end{proof}

In order to apply Corollary \ref{cor:torsion-free1}, it is important to dispose of a criterion for torsion-freeness of $\KU^{*,\alpha}(\Xc^{\an})$. If the twist $\alpha$ is trivial, it is known that vanishing of torsion in integral cohomology implies vanishing of torsion in $KU^{*}$ (see \cite[Bemerkung 14.18]{dold}).  
For Brauer-twists of complex varieties, this statement extends to the twisted topological $K$-theory:

\begin{proposition}\label{prop:cohomology-KU}
Let $\Xc$ be a smooth $\Cb$-variety endowed with a Brauer class $\alpha$. If $H^3(\Xc^{\an}; \mathbb{Z})$ is free, then $KU^{\tilde{\alpha},*}(\Xc^{\an}) \simeq KU^{*}(\Xc^{\an})$. If $H^{*}_{\sing}(\Xc^{\an};\Zb)$ is torsion-free, then $KU^{\tilde{\alpha},*}(\Xc^{\an})$ is torsion-free.
\end{proposition}

\begin{proof}
Since $\alpha \in \Br(\Xc)$ has finite order, so does the induced element $\tilde{\alpha} \in H^3(\Xc^{\an}; \mathbb{Z})$. We deduce that the topological gerbe $\tilde{\alpha} = 0$ is trivial, which directly implies the first statement. The second assertion now follows from the aforementioned reference \cite[Bemerkung 14.18]{dold}).  
\end{proof}

\section{A derived equivalence for $\SL_n$ and $\PGL_n$-Hitchin systems}\label{sec:autoduality}
\subsection{Coherent duality between the Prym scheme and its dual}\label{ssec: coherent duality}

Let $k$ be an algebraically closed field of characteristic $0$. Let $S$ be a $k$-scheme locally of finite type. Let $X\to S$ be a flat family of integral curves with planar singularities and arithmetic genus $g$. We assume $X$ admits a finite and flat map $X\to C\times S$ of $S$-schemes, where $C$ is a smooth projective curve over $k$ with genus $g_C$. Let $\overline{\on{Pic}}^0(X/S)$ be the relative compactified Jacobian of $X$ over $S$. We denote by $\check{\mathcal{P}}$ the kernel of the norm map $\overline{\on{Pic}}^0(X/S)\to \on{Pic}^0(C\times S/S)$, and we denote $\overline{{J}}=\overline{\on{Pic}}^0(X/S)$ and $\hat{\mathcal{P}}=[\overline{J}/\on{Pic}^0(C\times S/S)]$. The goal of this Subsection is to construct a Poincar\'e sheaf $\mathcal{F}$ over $\check{\mathcal{P}}\times_S \hat{\mathcal{P}}$ such that the Fourier-Mukai transform
\[
\Phi_{\mathcal{F}}: D_{\on{coh}}^{b}(\check{\mathcal{P}})\to D_{\on{coh}}^{b}(\hat{\mathcal{P}})
\]
with kernel $\mathcal{F}$ induces a fully faithful functor of derived categories. Subsequently, we will show that it gives rise to an equivalence of categories in the case we are interested in.

\begin{rmk}
A fibrewise version of the derived equivalences considered in this section were constructed by Franco--Hansen--Ruano in \cite{FHR}.
\end{rmk}

Recall that in \cite{autoduality}, the author constructed a Poincar\'e sheaf $\mathcal{G}$ on $\overline{J}\times_S \overline{J}$ which satisfies the following properties
\begin{theorem}\label{thm: arinkin autoduality}
\begin{itemize}
  \item[(a)] $\mathcal{G}$ is a maximal Cohen-Macaulay sheaf on $\overline{J}\times_S \overline{J}$ over each geometric point of $S$,
  \item[(b)] $\mathcal{G}$ is flat with respect to both projections $\overline{J}\times_S \overline{J}\to \overline{J}$,
  \item[(c)] $\mathcal{G}$ is equivariant with respect to the permutation of the two factors in $\overline{J}\times_S \overline{J}$,
  \item[(d)] the Fourier-Mukai functor with kernel $\mathcal{G}$ induces an auto-equivalence of the derived category $D_{\on{coh}}^{b}(\overline{J})$. In particular, we have \[Rp_{13*}(p_{12}^*\mathcal{G}\otimes p_{23}^*R\mathcal{H}om(\mathcal{G}, p_2^*\omega_{\overline{J}/S}[g])) \simeq \mathcal{O}_{\Delta_{\overline{J}}}.
  \]
\end{itemize}
\end{theorem}
Let $\iota: \check{\mathcal{P}}\times_S  \overline{J}\to \overline{J}\times_S \overline{J}$ be the inclusion map. The goal of the subsection is to show that $\iota^{*}\mathcal{G}$ descends to $\check{\mathcal{P}}\times_S  \hat{\mathcal{P}}$, i.e. $\iota^{*}\mathcal{G}$ admits a $\on{Pic}^0(C)$-equivariant structure, where $\on{Pic}^0(C)$ acts on the second factor. The sheaf $\iota^{*}\mathcal{G}$ satisfies the following properties:

\begin{lemma}
\begin{itemize}
    \item[(a)] $\iota^{*}\mathcal{G}$ is a maximal Cohen-Macaulay sheaf on $\check{\mathcal{P}}\times_S \overline{J}$ over each geometric point of $S$, 
    \item[(b)] $\iota^{*}\mathcal{G}$ is flat with respect to both projections $\check{\mathcal{P}}\times_S \overline{J}\to \check{\mathcal{P}}$ and $\check{\mathcal{P}}\times_S \overline{J}\to \overline{J}$.
\end{itemize}
\end{lemma}

\begin{proof}
Let $s\in S(k)$ be a closed point. Part (a) can be proved using Lemma 2.3 in \cite{autoduality}. The fact that all conditions in Lemma 2.3 \cite{autoduality} are satisfied follows from the following pull-back diagram:
\bd
\xymatrix{
 \check{\mathcal{P}}_s\times \on{Pic}^0(C)\ar[d]\ar[rr] &&\on{Pic}^0(C)\ar[d]^{[n]}\\
\overline{J}_s\ar[rr]^{\on{Nm}}&&\on{Pic}^0(C).
}
\ed
The map on the right-hand side is given by taking the $n$-th tensor product, where $n$ is the degree of the finite morphism $X_s\to C$. 

For part (b), it is clear that $\iota^{*}\mathcal{G}$ is flat with respect to the projection to $\check{\mathcal{P}}$. The other statement follows from Theorem \ref{thm: arinkin autoduality} part (c). 
\end{proof}

We denote by $\check{\mathcal{P}}^0\subset \check{\mathcal{P}}$ and $\overline{J}^0\subset \overline{J}$ the open subset which represents line bundles on $X$. Since $\iota^{*}\mathcal{G}$ is a Maximal Cohen-Macaulay sheaf, it is enough to construct the $\on{Pic}^0(C)$-equivariant structure on $\check{\mathcal{P}}^0\times_S  \overline{J}\bigcup\check{\mathcal{P}}\times_S\overline{J}^0$.

Let $\mathcal{L}$ be a universal sheaf over $X\times_S \overline{J}$. Recall that the restriction of the Poincar\'e sheaf $\mathcal{G}$ to $\overline{J}^0\times \overline{J}\bigcup\overline{J}\times \overline{J}^0$ is isomorphic to 
\begin{equation}\label{eq: def of poincare sheaf}
     \mathcal{D}_{p_{23}}(p^*_{12}\mathcal{L}\otimes p^*_{13}\mathcal{L})\otimes  \mathcal{D}_{p_{23}}(p^*_{12}\mathcal{L})^{-1}\otimes \mathcal{D}_{p_{23}}(p^*_{13}\mathcal{L})^{-1}\otimes \mathcal{D}_{p_{23}}(\mathcal{O}).
\end{equation}
Here $p_{12}$, $p_{13}$ and $p_{23}$ stand for the projection from $X\times_S \overline{J}\times_S \overline{J}$ to the corresponding factors, and $\mathcal{D}_{p_{23}}$ stands for the determinant of cohomology with respect to the morphism $p_{23}$. Basic properties of the determinant of cohomology can be found in \cite{est}, Section 6.1.
Now we show:
\begin{proposition}\label{prop: equivariant structure}
Let $K$ be a line bundle on $C$. Let $\mathcal{M}$, $\mathcal{N}$ be flat families of torsion-free sheaves on $X$ over $\overline{J}\times_S \overline{J}$. Then we have
\begin{equation}
\begin{split}
    &\mathcal{D}_{p_{23}}(\mathcal{M}\otimes \pi^{*}p_1^*K)\otimes \mathcal{D}_{p_{23}}(\mathcal{M})^{-1}\otimes  \mathcal{D}_{p_{23}}(\mathcal{N}\otimes \pi^{*}p_1^*K)^{-1}\otimes \mathcal{D}_{p_{23}}(\mathcal{N})\\
    &\simeq \mathcal{D}_{p_{23}}(\on{det}\pi_*(\mathcal{M})\otimes p_1^*K)\otimes \mathcal{D}_{p_{23}}(\on{det}\pi_*(\mathcal{M}))^{-1}\otimes \mathcal{D}_{p_{23}}(\on{det}\pi_*(\mathcal{N})\otimes p_1^*K)^{-1}\otimes \mathcal{D}_{p_{23}}(\on{det}\pi_*(\mathcal{N})).
\end{split}
\end{equation}
\end{proposition}
\begin{proof}
Let $D$ and $E$ be effective divisors on $C$ such that $K\cong \mathcal{O}_{C}(D-E)$. Consider the following short exact sequences
\begin{equation}
   0\to (\pi_*\mathcal{M}) (-E)\to \pi_*\mathcal{M}\to (\pi_*\mathcal{M})|_E\to 0
\end{equation}
\begin{equation}
    0\to (\pi_*\mathcal{M}) (-E)\to (\pi_*\mathcal{M})(D-E)\to (\pi_*\mathcal{M})(D-E)|_D\to 0.
\end{equation}
Since the determinant of cohomology is additive, we have
\[
\mathcal{D}_{p_{23}}(\mathcal{M}\otimes \pi^{*}p_1^*K)\otimes \mathcal{D}_{p_{23}}(\mathcal{M})^{-1}\simeq \mathcal{D}_{p_{23}}((\pi_*\mathcal{M})(D-E)|_D)\otimes \mathcal{D}_{p_{23}}((\pi_*\mathcal{M})|_E)^{-1}.
\]
We have similar short exact sequences associated with $\on{det}(\pi_*\mathcal{M})$:
\begin{equation}
   0\to \on{det}(\pi_*\mathcal{M})(-E)\to \on{det}(\pi_*\mathcal{M})\to\on{det}(\pi_*\mathcal{M})|_E\to 0
\end{equation}
\begin{equation}
    0\to \on{det}(\pi_*\mathcal{M}) (-E)\to \on{det}(\pi_*\mathcal{M})(D-E)\to (\on{det}(\pi_*\mathcal{M})(D-E))|_D\to 0. 
\end{equation}

Note that $\mathcal{D}_{p_{23}}((\pi_*\mathcal{M})|_E)\simeq \mathcal{D}_{p_{23}}(\on{det}(\pi_*\mathcal{M})|_E)$. If we fix a trivialisation of $\mathcal{O}_X(D-E)$ at $D$, we also have 
\[\mathcal{D}_{p_{23}}((\pi_*\mathcal{M})(D-E)|_D)\simeq \mathcal{D}_{p_{23}}((\on{det}\pi_*\mathcal{M})(D-E)|_D).\]
Combining the isomorphisms above, we obtain 
\[
\mathcal{D}_{p_{23}}(\mathcal{M}\otimes \pi^{*}p_1^*K)\otimes \mathcal{D}_{p_{23}}(\mathcal{M})^{-1}\simeq \mathcal{D}_{p_{23}}(\on{det}\pi_*(\mathcal{M})\otimes p_1^*K))\otimes \mathcal{D}_{p_{23}}(\on{det}\pi_*\mathcal{M})^{-1}. 
\]
Applying the same arguments and replacing $\mathcal{M}$ by $\mathcal{N}$, we obtain the desired isomorphism, which does not depend on the choice of the trivialisation.
\end{proof}

We apply Proposition \ref{prop: equivariant structure} by setting $\mathcal{M}=p_{12}^{*}\mathcal{L}\otimes p_{13}^{*}\mathcal{L}$ and $\mathcal{N}=p_{13}^{*}\mathcal{L}$. When restricted to $\check{\mathcal{P}}^0\times_S  \overline{J}\bigcup\check{\mathcal{P}}\times_S\overline{J}^0$, we have $\on{det}\pi_*(p_{12}^{*}\mathcal{L}\otimes p_{13}^{*}\mathcal{L})\simeq \on{det}\pi_*(p_{13}^{*}\mathcal{L})$. It follows that for any point $(F,H)\in \check{\mathcal{P}}^0\times_S  \overline{J}\bigcup\check{\mathcal{P}}\times_S\overline{J}^0$ and any line bundle $K\in \on{Pic}(C)$, the isomorphism in Proposition \ref{prop: equivariant structure} induces and isomorphism of fibres $\mathcal{G}_{(F,H)}\simeq\mathcal{G}_{(F,H\otimes\pi^{*}K)}$, which gives the desired equivariant structure. We denote the corresponding sheaf on $\check{\mathcal{P}}\times_S \hat{\mathcal{P}}$ by $\mathcal{F}$. 

The Fourier-Mukai transform $\Phi_{\mathcal{F}}$
admits a right-adjoint which is the Fourier-Mukai transform with kernel $$R\mathcal{H}om(\mathcal{F}, p_2^*\omega_{{\hat{\mathcal{P}}}/S}[g-g_C]).$$
In order to show $\Phi_{\mathcal{F}}$ is fully faithful, it is enough to prove the following
\begin{proposition}\label{prop: Fourier-Mukai composite}
   $Rp_{13*}(p_{12}^*\mathcal{F}\otimes p_{23}^*R\mathcal{H}om(\mathcal{F}, p_2^*\omega_{{\hat{\mathcal{P}}}/S}[g-g_C]))\simeq\mathcal{O}_{\Delta_{\check{\mathcal{P}}}}$.
\end{proposition}
Let $\mathcal{G}$ be the Poincar\'e sheaf on $\overline{J}\times_S \overline{J}$. Consider the following pull-back diagram
\bd
\xymatrix{
 \check{\mathcal{P}}\times_S {\overline{J}}\times_S \check{\mathcal{P}}\ar[d]_{\pi}\ar[rrr]^{\iota'} &&& \overline{J}\times_S {\overline{J}}\times_S \overline{J}\ar[dd]^{p_{13}}\\
\check{\mathcal{P}}\times_S {\hat{\mathcal{P}}}\times_S \check{\mathcal{P}}\ar[d]_{p_{13}'} \\
\check{\mathcal{P}}\times_S  \check{\mathcal{P}} \ar[rrr]^{\iota} &&&\overline{J}\times_S \overline{J}}
\ed
Since $\iota$ is a closed immersion and the maps $\iota$ and $p_{13}$ are tor independent, we can apply the base change theorem and obtain the following 
\begin{equation} \label{eq: base change}
L\iota^*Rp_{13*}(p_{12}^*\mathcal{G}\otimes p_{23}^*R\mathcal{H}om(\mathcal{G}, p_2^*\omega_{\overline{J}/S}[g]))
\simeq Rp_{13*}'R\pi_{*}L\iota'^*(p_{12}^*\mathcal{G}\otimes p_{23}^*R\mathcal{H}om(\mathcal{G}, p_2^*\omega_{\overline{J}/S}[g])).
\end{equation}
The left-hand side of (\ref{eq: base change}) is isomorphic to $L\iota^*\mathcal{O}_{\Delta_{\overline{J}}}$ by Theorem \ref{thm: arinkin autoduality} (3). Now we simplify the right-hand side. Since $\mathcal{G}$ is maximal Cohen-Macaulay, by Lemma 2.3 in \cite{autoduality}, we have 
\[
L\iota'^*(p_{12}^*\mathcal{G})\simeq \iota'^*(p_{12}^*\mathcal{G})\simeq \pi^*(p_{12}^*\mathcal{F})
\]
and 
\[
L\iota'^*(p_{23}^*R\mathcal{H}om(\mathcal{G}, p_2^*\omega_{\overline{J}/S}[g]))\simeq p_{23}^*(R\mathcal{H}om(\pi^{*}\mathcal{F}, p_2^*\omega_{\overline{J}/S}[g]))
\]
Applying projection formula and Grothendieck duality, the right-hand side of (\ref{eq: base change}) becomes
\begin{equation} 
\begin{split}
&Rp_{13*}'R\pi_{*}(\pi^*(p_{12}^*\mathcal{F})\otimes p_{23}^*(R\mathcal{H}om(\pi^{*}\mathcal{F}, p_2^*\omega_{\overline{J}/S}[g]))\\
&\simeq Rp_{13*}'(p_{12}^*\mathcal{F}\otimes R\pi_{*}p_{23}^*(R\mathcal{H}om(\pi^{*}\mathcal{F}, p_2^*\omega_{\overline{J}/S}[g]))\\
&\simeq Rp_{13*}'(p_{12}^*\mathcal{F}\otimes p_{23}^*R\mathcal{H}om(\mathcal{F}\otimes R\pi_*\mathcal{O}_{\overline{\mathcal{P}}}, p_2^*\omega_{{\hat{\mathcal{P}}}/S}[g-g_C])).
\end{split}
\end{equation}
Since the map $\pi$ gives $\overline{J}$ the structure of a $\on{Pic}^0(C)$-torsor over $\hat{\mathcal{P}}$, we have 
\begin{lemma}\label{lemma: pushforward free}
For any $q\geq 0$, $R\pi_*^q\mathcal{O}_{\overline{\mathcal{P}}}$ is a free sheaf of rank ${g \choose q}$. In particular, $R\pi_*^0\mathcal{O}_{\overline{\mathcal{P}}}\simeq \mathcal{O}_{\hat{\mathcal{P}}}$.
\end{lemma}

\begin{proof}[Proof of Proposition \ref{prop: Fourier-Mukai composite}]
Let $\Psi=Rp_{13*}(p_{12}^*\mathcal{F}\otimes p_{23}^*R\mathcal{H}om(\mathcal{F}, p_2^*\omega_{{\hat{\mathcal{P}}}/S}[g-g_C]))$. The goal is to show $\Psi\simeq\mathcal{O}_{\Delta_{\check{\mathcal{P}}}}$.
Consider the Grothendieck spectral sequence
\begin{equation}\label{eq: spectral sequence for F-M}
\begin{split}
    & E^{p,q}_2=\mathcal{H}^{p}(Rp_{13*}'(p_{12}^*\mathcal{F}\otimes p_{23}^*R^q\mathcal{H}om(\mathcal{F}\otimes R\pi_*\mathcal{O}_{\overline{\mathcal{P}}}, p_2^*\omega_{{\hat{\mathcal{P}}}/S}[g-g_C]))) \\ 
     &\Rightarrow \mathcal{H}^{p+q}(Rp_{13*}'(p_{12}^*\mathcal{F}\otimes p_{23}^*R\mathcal{H}om(\mathcal{F}\otimes R\pi_*\mathcal{O}_{\overline{\mathcal{P}}}, p_2^*\omega_{{\hat{\mathcal{P}}}/S}[g-g_C])))\\
     &\simeq L^{p+q}\iota^*\mathcal{O}_{\Delta_{\overline{\mathcal{P}}}}.
\end{split}
\end{equation}
Since $\mathcal{F}$ is a maximal Cohen-Macaulay sheaf, we have $R^q\mathcal{H}om(\mathcal{F}, p_2^*\omega_{{\hat{\mathcal{P}}}/S})=0$ when $q\neq 0$. Combining Lemma \ref{lemma: pushforward free}, we obtain
\[
R^q\mathcal{H}om(\mathcal{F}\otimes R\pi_*\mathcal{O}_{\overline{\mathcal{P}}}, p_2^*\omega_{{\hat{\mathcal{P}}}/S}[g-g_C])= R\mathcal{H}om(\mathcal{F}\otimes R\pi_*^{g_C-g-q}\mathcal{O}_{\overline{\mathcal{P}}}, p_2^*\omega_{{\hat{\mathcal{P}}}/S}[g-g_C])
\]

Since $E_2^{p,q}$ is concentrated in the locus where $0\leq p \leq g-g_C$ and $-g\leq q \leq g_C-g$, we have $E_2^{g-g_C,g_C-g}\simeq \iota^*\mathcal{O}_{\Delta_{\overline{\mathcal{P}}}}\simeq \mathcal{O}_{\Delta_{\check{\mathcal{P}}}}$. By Lemma \ref{lemma: pushforward free}, we have $\mathcal{H}^p(\Psi)=E_2^{p,g_C-g}$. It follows that in order to show $\Psi\simeq\mathcal{O}_{\Delta_{\check{\mathcal{P}}}}$, it is enough to show $\Psi$ is a sheaf. By Lemma 7.7 in \cite{autoduality}, it is enough to show $\on{codim(supp(\Psi))} \geq g-g_C$. 

Now we show $\on{supp}(E_2^{p,q})\subseteq \Delta_{\check{\mathcal{P}}}$. We prove by induction on $p$. By Lemma \ref{lemma: pushforward free} and the isomorphism $E_2^{g-g_C,g_C-g}\simeq \mathcal{O}_{\Delta_{\check{\mathcal{P}}}}$, we have $\on{supp}(E_2^{g-g_C,q})\subseteq \Delta_{\check{\mathcal{P}}}$. Now we assume the statement is true for any $p>s$. Since $E_r^{p,q}$ is constructed from $E_2^{p,q}$ by consecutively taking sub-quotients, we have $\on{supp}(E_r^{p,q})\subseteq \Delta_{\check{\mathcal{P}}}$ for any $p>s$ and $r\geq 2$. Our goal is to show $\on{supp}(E_2^{s-1,q})\subseteq \Delta_{\check{\mathcal{P}}}$, and by Lemma \ref{lemma: pushforward free}, it is enough to prove the special case when $q=g_C-g$. Let $t$ be the largest number such that $d_t$ is non-zero at position $(s-1, g_C-g)$. From the spectral sequence (\ref{eq: spectral sequence for F-M}), we have the following short exact sequence
\[
0\to E_{\infty}^{s-1,g_C-g}\xrightarrow[]{d_t} E_{t}^{s-1,g_C-g}\to \on{Im}d_t\to 0. 
\]
Since the spectral sequence $\{E_2^{p,q}\}$ converges to $L^{p+q}\iota^*\mathcal{O}_{\Delta_{\overline{\mathcal{P}}}}$, the support of  $E_{\infty}^{s-1,q}$ must lie in  $\Delta_{\check{\mathcal{P}}}$. We also have $\on{Im}d_t\subseteq \Delta_{\check{\mathcal{P}}}$, which follows from the induction hypothesis. It follows that $\on{supp}(E_{t}^{s-1,g_C-g})\subseteq \Delta_{\check{\mathcal{P}}}$. Applying the same argument consecutively, we deduce the existence of an inclusion $\on{supp}(E_{2}^{s-1,g_C-g})\subseteq \Delta_{\check{\mathcal{P}}}$.
\end{proof}

\subsection{A derived equivalence for Hitchin systems}\label{ssec:autoduality}
We fix a line bundle $D$ such that either $\on{deg}D>2g-2$ or $D=K$ the canonical line bundle of $C$. We fix a line bundle $L$ of degree $1$ on $C$. Let $\check{\mathcal{M}}_d$ be the moduli space of semistable $D$-twisted $\on{SL}_n$-Higgs bundles with determinant $L^d$. More precisely, $\check{\mathcal{M}}_d$ represents semistable triples $(E, \varphi, \theta)$, where $E$ is a vector bundle of rank $n$ on $X$, $\varphi$ is a morphism $\varphi: E\to E\otimes D$ and $\theta$ is an isomorphism $\theta: \on{det}E\simeq L^d$. The group of $n$-torsion points $\Gamma \subset \on{Pic}(C)$ acts on the moduli space $\check{\mathcal{M}}_d$, and the quotient $[\check{\mathcal{M}}_d/\Gamma]$ gives the moduli stack of semistable $\on{PGL}_n$-Higgs bundles of degree $\bar{d}\in \mathbb{Z}/n\mathbb{Z}$, which we denote by $\hat{\mathcal{M}}_d$. Both $\check{\mathcal{M}}_d$ and $\hat{\mathcal{M}}_e$ admit the Hitchin map $\check{h}:\check{\mathcal{M}}_d\to \mathcal{A}$ and $\hat{h}: \hat{\mathcal{M}}_e\to \mathcal{A}$ to the same affine space $\mathcal{A}$. Let $X\to C\times \mathcal{A}$ be the family of spectral curves. We denote by $\mathcal{A}^{\diamondsuit}\subset \mathcal{A}$ the locus where the spectral curves are smooth, and by $\mathcal{A}^{\heartsuit}\subset \mathcal{A}$ the locus where the spectral curves are integral(i.e. the elliptic locus). Let $\check{\mathcal{M}}^{\heartsuit}_d\coloneqq \check{h}^{-1}(\mathcal{A}^{\heartsuit})$ and $\hat{\mathcal{M}}^{\heartsuit}_d\coloneqq \hat{h}^{-1}(\mathcal{A}^{\heartsuit})$ be the Hitchin fibre over the elliptic locus. Let $\check{\mathcal{P}}$ be the compactified Prym scheme of $X^{\heartsuit}\to C\times \mathcal{A}^{\heartsuit}$ defined in Subsection \ref{ssec: coherent duality}, and let $\check{\mathcal{P}}_d=\on{Nm}^{-1}(L^d)$ for the norm map $\on{Nm}:\overline{\on{Pic}}^0(X^{\heartsuit}/\mathcal{A}^{\heartsuit})\to \on{Pic}^0(C\times \mathcal{A}/\mathcal{A})$. We denote $\hat{\mathcal{P}}=[\check{\mathcal{P}}/\Gamma]$ and $\hat{\mathcal{P}}_d=[\check{\mathcal{P}}_d/\Gamma]$. The spectral description of Hitchin fibres implies that $\check{\mathcal{M}}^{\heartsuit}_d$ (resp.  $\hat{\mathcal{M}}^{\heartsuit}_d$) is isomorphic to $\check{\mathcal{P}}_{d'}$ (resp. $\hat{\mathcal{P}}_{d'}$) for some $d'$, and it is \'etale locally isomorphic to $\check{\mathcal{P}}$ (resp. $\hat{\mathcal{P}}$). Let $\check{\alpha}_e$ (resp. $\hat{\alpha}_d$) be the gerbe on $\check{\mathcal{M}}_d$ (resp. $\hat{\mathcal{M}}_e$) that appears in the topological mirror symmetry for $E$-polynomials. We refer the readers to \cite{gwz17} Section 7.4 for the precise definition of those gerbes. The main goal of this subsection is to prove the following:

\begin{theorem}\label{thm: twisted F-M between Hitchin systems}
There exists $\mathcal{K}\in\on{Coh}(\check{\mathcal{M}}^{\heartsuit}_d\times_{\mathcal{A}^{\heartsuit}}\hat{\mathcal{M}}^{\heartsuit}_e, \check{\alpha}_e\boxtimes\hat{\alpha}_d)$ such that the corresponding Fourier-Mukai transform
\[
\Phi_\mathcal{K}: D^b_{\on{coh}}(\check{\mathcal{M}}_d^{\heartsuit},\check{\alpha}^{-1}_e) \xrightarrow{\simeq} D^b_{\on{coh}}(\hat{\mathcal{M}}_e^{\heartsuit},\hat{\alpha}_d)
\]
induces an equivalence of derived categories. 
\end{theorem}

We denote by $\overline{\mathfrak{Pic}}(X^{\heartsuit}/\mathcal{A}^{\heartsuit})$ the moduli stack of rank 1 torsion-free sheaves on $X^{\heartsuit}$ over $\mathcal{A}^{\heartsuit}$, and by  $\mathfrak{Pic}(X^{\heartsuit}/\mathcal{A}^{\heartsuit})$ the moduli stack of line bundles on $X^{\heartsuit}$ over $\mathcal{A}^{\heartsuit}$. Let $\mathcal{L}$ be the Poincar\'e line bundle on $\mathfrak{Pic}(X^{\heartsuit}/\mathcal{A}^{\heartsuit})\times \overline{\mathfrak{Pic}}(X^{\heartsuit}/\mathcal{A}^{\heartsuit})\bigcup \overline{\mathfrak{Pic}}(X^{\heartsuit}/\mathcal{A}^{\heartsuit})\times \mathfrak{Pic}(X^{\heartsuit}/\mathcal{A}^{\heartsuit})$ defined by (\ref{eq: def of poincare sheaf}). The moduli stack $\overline{\mathfrak{Pic}}(X^{\heartsuit}/\mathcal{A}^{\heartsuit})$ is \'etale locally isomorphic to $\overline{\on{Pic}}(X^{\heartsuit}/\mathcal{A}^{\heartsuit})\times B\mathbb{G}_m\times \mathbb{Z}$, where the $\mathbb{Z}$-component corresponds to the degree of the torsion-free sheaf. When restricted to the component of degree $(d,e)\in \mathbb{Z}\times\mathbb{Z}$, the Poincar\'e sheaf $\mathcal{L}$ has weight $(e,d)$ for the $\mathbb{G}_m\times \mathbb{G}_m$-action corresponding to the $B\mathbb{G}_m$ factors. Consider the map $\eta: \overline{\mathfrak{Pic}}(X^{\heartsuit}/\mathcal{A}^{\heartsuit})\to \on{Pic}^0(C\times \mathcal{A}^{\heartsuit}/ \mathcal{A}^{\heartsuit})$ defined by mapping a line bundle $M$ to $\on{Nm}(M)\otimes L^{-\on{deg}(M)}$. The kernel $\on{Ker}(\eta)$ is \'etale locally isomorphic to $\check{\mathcal{P}}\times B\mathbb{G}_m\times\mathbb{Z}$ such that the $\mathbb{Z}$-component corresponds to the degree of the torsion-free sheaf and the $B\mathbb{G}_m$-action corresponds to the gerbe $\check{\alpha}$. Let $\on{Ker}(\eta)^{reg}\subset \on{Ker}(\eta)$ be the subset representing line bundles. Similar arguments as in Proposition \ref{prop: equivariant structure} implies that when restricted to $\on{ker}(\eta)^{reg}\times \overline{\mathfrak{Pic}}(X^{\heartsuit}/\mathcal{A}^{\heartsuit})\bigcup  \on{ker}(\eta)\times{\mathfrak{Pic}}(X^{\heartsuit}/\mathcal{A}^{\heartsuit})$, $\mathcal{L}$ admits a $\on{Pic}^0(C)$-equivariant structure, therefore descends to a line bundle on $\on{ker}(\eta)^{reg}\times [\overline{\mathfrak{Pic}}(X^{\heartsuit}/\mathcal{A}^{\heartsuit})/\on{Pic}^0(C)]\bigcup \on{ker}(\eta)\times[{\mathfrak{Pic}}(X^{\heartsuit}/\mathcal{A}^{\heartsuit})/\on{Pic}^0(C)]$, which we still denote by $\mathcal{L}$. Note that  $[\overline{\mathfrak{Pic}}(X^{\heartsuit}/\mathcal{A}^{\heartsuit})/\on{Pic}^0(C)]$ is \'etale locally isomorphic to $\hat{\mathcal{P}}\times B\mathbb{G}_m\times\mathbb{Z}$ such that the $\mathbb{Z}$-component corresponds to the degree of the line bundle and the $B\mathbb{G}_m$-action corresponds to the gerbe $\hat{\alpha}$. It follows that the Poincar\'e line bundle $\mathcal{L}$ restricts to a $\check{\alpha}_{e}\boxtimes\hat{\alpha}_d$ twisted sheaf on $ (\check{\mathcal{M}}^{\heartsuit}_d)^{reg}\times_{\mathcal{A}^{\heartsuit}}\hat{\mathcal{M}}_e^{\heartsuit}\bigcup\check{\mathcal{M}}^{\heartsuit}_d\times_{\mathcal{A}^{\heartsuit}}(\hat{\mathcal{M}}_e^{\heartsuit})^{reg}$, where $(\check{\mathcal{M}}^{\heartsuit}_d)^{reg}\subset (\check{\mathcal{M}}^{\heartsuit}_d)$ is the open subset defined by the spectral sheaf being a line bundle. 

\begin{proof}[Proof of Theorem \ref{thm: twisted F-M between Hitchin systems}]
We define $\mathcal{K}=j_{*}\mathcal{L}$, where $j: (\check{\mathcal{M}}^{\heartsuit}_d)^{reg}\times_{\mathcal{A}^{\heartsuit}}\hat{\mathcal{M}}_e^{\heartsuit}\bigcup\check{\mathcal{M}}^{\heartsuit}_d\times_{\mathcal{A}^{\heartsuit}}(\hat{\mathcal{M}}_e^{\heartsuit})^{reg}\hookrightarrow\check{\mathcal{M}}^{\heartsuit}_d\times_{\mathcal{A}^{\heartsuit}}\hat{\mathcal{M}}_e^{\heartsuit}$ is the open immersion. The Fourier-Mukai transform $\Phi_{\mathcal{K}}$ admits a right-adjoint which is a Fourier-Mukai transform with kernel $\mathcal{K}'=R\mathcal{H}om(\mathcal{K}, p_2^*\omega_{\hat{\mathcal{M}}_e^{\heartsuit}/\mathcal{A}^{\heartsuit}})$. We claim that both adjunction maps $id\to \Phi_{\mathcal{K}'}\circ\Phi_{\mathcal{K}}$ and $ \Phi_{\mathcal{K}}\circ\Phi_{\mathcal{K}'}\to id$ are isomorphisms, which can be checked \'etale locally. By taking an \'etale cover of $\mathcal{A}^{\heartsuit}$ that trivialises all the torsors and gerbes involved, $\mathcal{K}$ becomes isomorphic to the Fourier-Mukai kernel 
$\mathcal{F}$ constructed in Section \ref{ssec: coherent duality} for the flat family of curves $X^{\heartsuit}/\mathcal{A}^{\heartsuit}$. Therefore it is enough to show that $
\Phi_{\mathcal{F}}: D_{\on{coh}}^{b}(\check{\mathcal{P}})\to D_{\on{coh}}^{b}(\hat{\mathcal{P}})
$ induces an equivalence of derived categories. Since $\check{\mathcal{P}}$ is isomorphic to $\check{\mathcal{M}}^{\heartsuit}_{a}$ for some degree $a\in\mathbb{Z}$, it admits a $\Gamma$-invariant top-degree form, therefore $D_{\on{coh}}^{b}(\hat{\mathcal{P}})$ is a Calabi-Yau triangulated category over $\mathcal{A}^{\heartsuit}$ in the sense of Section 3.5 \cite{bmr}. By Lemma 4.2 in \cite{bkr}, it is also indecomposible. The desired statement of $\Phi_{\mathcal{F}}$ being an equivalence follows from Proposition \ref{prop: Fourier-Mukai composite} and Lemma 3.5.2 in \cite{bmr}.
\end{proof}

\section{Proofs of the main results}\label{sec:app}

\subsection{The case of $D$-twisted Higgs bundles}\label{case-D} 

\begin{theorem}\label{thm:main2twisted}
Let $D$ be a divisor of degree $> 2g-2$. There is an equivalence of complex $K$-theory \emph{spectra} $\KU(\hMm^{D,\an}) \simeq \KU(\cMm^{D,\an},\tilde{\alpha})$.
\end{theorem}
\begin{proof}
This follows from Corollary \ref{cor:torsion-free1}. The base $S$ is chosen to be $\Aa^D$, $\Xc = \hMm^D$, and $\Yc = \cMm^D$. The subset $\Aa^{D,\heartsuit}$ corresponds to the open subset of integral spectral curves. As we showed in Theorem \ref{thm: twisted F-M between Hitchin systems}, Arinkin's integral kernel can be defined on the fibre product 
$$\hMm^{D,\heartsuit} \times_{\Aa^{D,\heartsuit}} \cMm^{D,\heartsuit},$$
and induces a derived equivalence over this locus. Vanishing of torsion of the complex $K$-theory groups in question, which is established in Section \ref{sec:torsion}, allows us to conclude that all hypotheses of Corollary \ref{cor:torsion-free1} are met.
\end{proof}

\subsection{The case of the canonical divisor}\label{ssec:canonical}

Let $\Cc$ be a presentable stable $\infty$-category and $\Eb \in \Cc$ an object. For a topological space $S$, we denote by $\underline{\Eb}_X$ the corresponding constant sheaf on $X$.

We denote by $i\colon \hMm^{\an}_D\hookrightarrow \hMm^{\an}_{D+p}$ the canonical inclusion. 
Recall from \cite[Equation (103) \& Theorem 4.5]{MS} that there is a function $\MS \colon \Aa_{D+p} \to \Ab^1$ with the critical locus of $f \circ \chi$ being equal to the image of $i$.

\begin{proposition}\label{prop:phi-SLn}
There exists an integer $s$ such that we have an equivalence of $\Gamma$-sheaves of spectra (see Appendix \ref{sec:app})
$$
\phi_{\MS \circ \chi}\big((\underline{\KU}_{\Gamma}^{\tilde{\alpha}})_{\hMm^{\an}_{D+p}}\big)\simeq i_*(\underline{\KU}_{\Gamma}^{\tilde{\alpha}})_{\hMm^{\an}_D}[s]
$$
\end{proposition}
\begin{proof}
We follow the proof of \cite[Theorem 4.5]{MS}, using the formalism of $G$-sheaves of spectra developed in Appendix \ref{sec:app}, where we let $G$ be the group $\Gamma$ of $n$-torsion points of the Jacobian.

We have the following commutative diagram of stacks endowed with a $\Gamma$-action:
\[
\xymatrix{
\Mm'  \ar[r]^-{\ev}_-{\text{smooth}} \ar@/_/[rd]_{\mu} & [\mathfrak{sl}_n/\PGL_n] \ar@/^/[rd]^{p} \ar[d]^{c_2} &  \\
& \Ab^1 \ar[r]_{p} & \bullet
}
\]
The map $\ev = \ev_p$ restricts a $D+p$-twisted Higgs bundle on $C$ to the point $p$. This map is well-defined up to choice of a trivialisation of $\Oo_C(D+p)_p$. 
Since $\ev_p$ is a smooth morphism for $\deg D > 2g-2$ or $D=\Omega_C^1$ (see \cite[Proposition 4.8]{MS}), we can apply Proposition \ref{prop:smooth-basechange} and obtain a natural isomorphism $\phi^{\Gamma}_{\mu}p^{-1}_{\Gamma} \KU_{\Gamma}^{\alpha} \simeq \ev^{-1}_G\phi_{c_2}^{\Gamma} p^{-1}_{\Gamma}\KU_{\Gamma}^{\alpha}$, where $\KU^{\alpha}_{\Gamma}(\bullet) \in \Sheaves(\bullet;\Gamma)$ denotes the naive $\Gamma$-spectrum representing $\alpha$-twisted $\Gamma$-equivariant complex $K$-theory.

It remains therefore to compute $\phi_{c_2}^{\Gamma} p^{-1}_{\Gamma}\KU_{\Gamma}^{\alpha}$. Since $c_2$ is a quadratic polynomial, we can argue as in \cite[?]{MS} and identify its Milnor fibre (as a $\Gamma$-space) with an $r$-sphere. Note that $\Gamma$ acts trivially on all spaces involved. Thus, we have $\phi_{c_2}^{\Gamma} p^{-1}_{\Gamma}\KU_{\Gamma}^{\alpha} \simeq p^{-1}_{\Gamma}\Sigma^r\KU_{\Gamma}^{\alpha}$.
\end{proof}

\begin{corollary}
For $s$ as above we have $$\phi_{\MS \circ \chi}(\underline{\KU}^{\tilde{\alpha}}_{\hMm^{\an}_{D+p};\Gamma})\simeq i_*\underline{\KU}^{\tilde{\alpha}}_{\hMm^{\an}_D;\Gamma}[s].$$
\end{corollary}
\begin{proof}
We apply the equivariant pushforward functor $q_*^{\Gamma}$ defined in the appendix (see Definition \ref{defi:equi-push}), where $q$ denotes the quotient map $\hMm^{\an} \to \cMm^{\an}$.

According to Proposition \ref{cor:proper-push}, we have
$q_*^{\Gamma} \circ \phi_{\mu \circ \chi} \simeq \phi_{\mu \circ \chi} \circ q_*^{\Gamma}$. Furthermore, $q_*^{\Gamma}$ maps $(\underline{\KU}_{\Gamma})_U$ to $\underline{\KU}_{U;\Gamma}$. This yields the requisite equivalence
\end{proof}

\begin{corollary}
Let $s$ be the same integer as above. Then, we have an equivalence of sheaves of spectra 
$$
\phi_{\MS \circ \chi}\big((\underline{\KU})_{\hMm^{\an}_{D+p}}\big)\simeq i_*(\underline{\KU})_{\hMm^{\an}_D}[s]
$$
\end{corollary}
\begin{proof}
This follows from the previous result by forgetting the $\Gamma$-structure (or alternatively, could be prove the same way, \emph{mutatis mutandis}). The absence of $\alpha$ in the statement is explained by the fact that the gerbe $\alpha$ is neutral on $\hMm$, if its $\Gamma$-equivariant structure is not taken into account.
\end{proof}

\begin{theorem}\label{thm:main2}
There is an equivalence of complex $K$-theory \emph{spectra} $\KU(\hMm^{\an}) \simeq \KU(\cMm^{\an},\tilde{\alpha})$.
\end{theorem}

\begin{proof}
Let $K$ be a canonical divisor on $X$, and $p$ an auxiliary point. We denote by $\Mm$ the moduli space of $K$-twisted Higgs bundles, and by $\Mm^{\prime}$ the moduli space of $(K+p)$-twisted Higgs bundles.
Similar notation will be applied to the ``$\SL_n$" and $\PGL_n$-variants, as well as the Hitchin bases.
We therefore have an embedding and regular $\Gamma$-invariant functions
\[
\xymatrix{
\hMm \ar[r]^i \ar[d]_{\check{\pi}} & \hMm^{\prime} \ar[r]^{\mu} \ar[d]_{\check{\pi}^{\prime}} & \Ab^1 \\
\Aa \ar[r] & \Aa^{\prime}\text{.} \ar[ur] &
}
\]
We now let $F\colon \hat{\pi}^{\prime}_*\underline{\KU}_{\cMm^{\prime,\an}}^{\tilde{\alpha}}\to \check{\pi}^{\prime}_* \underline{\KU}_{\hMm^{\prime,\an}}$ and $G\colon \check{\pi}^{\prime}_* \underline{\KU}_{\hMm^{\prime,\an}} \to \hat{\pi}^{\prime}_*\underline{\KU}_{\cMm^{\prime,\an}}^{\tilde{\alpha}}$ be the morphisms of sheaves induced by arbitrary extensions of Arinkin's kernel and the inverse kernel (see \ref{case-D}). By virtue of the previous corollaries, we obtain morphisms 
$\phi_{\mu}(F)$ and $\phi_{\mu}(G)$ relating the sheaves
$\hat{\pi}_*\underline{\KU}_{\cMm^{\an}}^{\tilde{\alpha}}$ and $\check{\pi}_*\underline{\KU}_{\hMm^{\an}}^{\tilde{\alpha}}$. As in the proof of Theorem \ref{thm:main} one shows that after rationalisation these morphisms are inverse to each other up to a unipotent map. As before, vanishing of torsion allows us to conclude that the induced map of spectra 
$\KU(\hMm^{\an}) \simeq \KU(\cMm^{\an},\tilde{\alpha})$
is an equivalence.
\end{proof}

\section{Vanishing of torsion for moduli spaces of Higgs bundles}\label{sec:torsion}

The last section is devoted to establishing torsion-freeness for complex $K$-theory for smooth moduli spaces of $G$-Higgs bundles, where $G=\GL_n$, $\SL_n$, or $\PGL_n$. For the first two groups, one can prove more generally that singular cohomology is torsion-free, whereas this does not hold for $\PGL_n$-Higgs bundles. This is another instance of the well-documented phenomenon that passing to complex $K$-theory annihilates the appropriate amount of torsion in singular cohomology.

We will begin by establishing torsion-freeness for $\GL_n$ and $\SL_n$-Higgs bundles, by using information about the $\G_m$-fixpoints which is provided by work of Garcia-Prada--Heinloth--Schmidt \cite{GPHS} and Garcia-Prada--Heinloth \cite{gph13}. Subsequently, we will show that torsion-freeness of singular cohomology of $\Mm$ implies vanishing of torsion in (twisted) complex $K$-theory of $\cMm$.

\subsection{$\GL_n$}
Let $D$ be a line bundle on $C$ such that either $\on{deg}(D)>2g-2$ or $D=K$ the canonical bundle. We denote by $\mathcal{M}$ the moduli space of semistable $D$-twisted $\on{GL}_n$-Higgs bundles of degree $d$ over $C$. We assume $(d,n)=1$ so that $\mathcal{M}$ is smooth. The goal of this subsection is to prove the following 
\begin{theorem}\label{thm:GLn}
The singular cohomology groups $H^*(\mathcal{M}^{an},\Zb)$ are torsion-free.
\end{theorem}

We consider the $\mathbb{G}_m$-action on $\mathcal{M}$ defined by scaling the Higgs field. The fixed point scheme $(\mathcal{M})^{\mathbb{G}_m}$ is a disjoint union of connected, smooth projective varieties $F_i$. By the work of Frankel \cite{Frankel}, in order to show the cohomology of $\mathcal{M}$ is torsion-free, it is enough to show the cohomology of each $F_i$ to be torsion-free. 

Let $K$ a field. For a smooth complex projective variety $X$, we define $P_K(X, t)\in \mathbb{Z}[t]$ to be its Poincar\'e polynomial with $K$-coefficients:
\[
P_K(X, t)\coloneqq\sum_{i}\on{dim}_K (H^{i}(X,K))t^i.
\]
By the work of Gillet and Soul\'e in \cite{gs96}, Section 3.3, there exists an extension of $P_K(X, t)$ to all complex algebraic varieties which satisfies 
\[
 P_K(Y, t)=P_K(Z, t)+P_K(Y-Z, t)
\]
for any complex algebraic variety $Y$ and closed subvariety $Z\subset Y$. In other words, this construction can be viewed as a map
\[
P_K: K_{0}(\on{Var}_{\mathbb{C}})\to \mathbb{Z}[t],
\]
were $K_{0}(\on{Var}_{\mathbb{C}})$ is the Grothendieck ring of complex algebraic varieties. The map $P_K$ is a ring homomorphism by \cite{gs96}, Proposition 5. Let $\mathbb{L}=[\mathbb{A}_{\mathbb{C}}^1]\in K_{0}(\on{Var}_{\mathbb{C}})$ be the Lefschetz motive and let $\hat{K}_{0}(\on{Var}_{\mathbb{C}})$ be the dimensional completion of $K_{0}(\on{Var}_{\mathbb{C}})[\mathbb{L}^{-1}]$. Then $P_K$ extends to a map
\[
P_K: \hat{K}_{0}(\on{Var}_{\mathbb{C}})\to \mathbb{Z}[t][[t^{-1}]]. 
\]

\begin{proof}[Proof of Theorem \ref{thm:GLn}]
It is enough to show that for any $\mathbb{G}_m$-fixed component $F_i$ and any prime $p$, we have the following equality
\begin{equation}\label{eq: poincare}
P_{\mathbb{Q}}(F_i, t)=P_{\mathbb{F}_p}(F_i, t).
\end{equation}
Each $F_i$ corresponds to the moduli of chains on $C$ of a certain type, see Lemma 9.2 in \cite{ht}. It is proved in \cite{gph13}, Theorem B that the class $[F_i]\in \hat{K}_{0}(\on{Var}_{\mathbb{C}})$ can be expressed in terms of $\mathbb{L}$ and the symmetric powers $[C^{(j)}]$ of the curve $C$. Therefore it is enough to prove 
(\ref{eq: poincare}) for $\mathbb{L}$ and $[C^{(j)}]$. The equalities for $\mathbb{L}$ are obvious. The equalities for $[C^{(j)}]$ follow from the fact that the cohomology groups of $[C^{(j)}]$ are torsion-free, see \cite{Mcdonald}. 
\end{proof}

\begin{corollary}\label{cor:GLn}
For $\alpha \in \Br(\Mm)$ the twisted complex $K$-theory groups $\KU^{*}(\Mm^{\an},\tilde{\alpha})$ are torsion-free.
\end{corollary}
\begin{proof}
In Proposition \ref{prop:cohomology-KU} we showed that vanishing of torsion in singular cohomology implies vanishing of torsion in complex $K$-theory twisted with respect to a Brauer twist. Theorem \ref{thm:GLn} above therefore implies that the $\Zb/2\Zb$-graded abelian group $\KU^{*}(\Mm^{\an},\tilde{\alpha})$ is torsion-free.
\end{proof}

\subsection{$\SL_n$}
Let $L$ be a line bundle of degree $d$ on $C$, and let $\hMm\coloneqq \on{det}^{-1}(L)$, where $\on{det}: \mathcal{M}\longrightarrow J^d$ is the determinant map. We assume $(d,n)=1$. The goal of this subsection is to prove the following vanishing result on torsion:
\begin{theorem}\label{thm:SLn}
The singular cohomology groups $H^*(\hMm^{\an},\Zb)$ are torsion-free.
\end{theorem}

Since the isomorphic class of $\hMm$ does not depend on the choice of $L$, we fix a closed point $c_0\in C$ and restrict to the case where $L=\mathcal{O}(dc_0)$. Since $H^{*}(J)$ is torsion-free, proving Theorem \ref{thm:SLn} is equivalent to showing the cohomology of $\hMm\times J$ is torsion-free. We have the following pull-back diagram

\begin{equation}\label{eq: pull back jacobian}
    \xymatrix{
 \hMm\times J \ar[rr] \ar[d] &&J \ar[d]^{[n]} \\
  \mathcal{M} \ar[rr]^{\on{det}\otimes L^{-1}} & &J.
}
\end{equation}

\begin{proposition}
    The class $[\hMm\times J]\in \hat{K}_{0}(\on{Var}_{\mathbb{C}})$ can be written as a linear combination of classes of the form $[\widetilde{\displaystyle\prod_{i} C^{(l_i)}}]\times \mathbb{L}^m$, where $\widetilde{\displaystyle\prod_{i} C^{(l_i)}}$ fits into the following pull-back diagram
\begin{equation}\label{eq: pull-back jacobian main prop}
  \xymatrix{
 \widetilde{\displaystyle\prod_{i} C^{(l_i)}} \ar[rr] \ar[d] &&J \ar[d]^{[n]} \\
  {\displaystyle\prod_{i} C^{(l_i)}}\ar[rr]^{\displaystyle\bigotimes_i \mu(k_i)} & &J.
}  
\end{equation}
Here $\mu{(k_i)}: C^{(l_i)}\longrightarrow J$ is the composite of the Abel-Jacobi map with the $k_i$-th multiplication map $J\xrightarrow{[k_i]}J$. More precisely, $\mu{(k_i)}$ maps the class of $(c_1, c_2,\dots, c_r)$ to $\mathcal{O}_{C}(c_1+c_2+\cdots c_r-rc_0)^k$.
\end{proposition}

\begin{proof}
By Theorem B in \cite{gph13}, the class $[\mathcal{M}]\in \hat{K}_{0}(\on{Var}_{\mathbb{C}})$ can be written as a linear combination of classes of the form $[\displaystyle\prod_{i} C^{(l_i)}]\times \mathbb{L}^m$. The procedure explained in the proof of Theorem B in \cite{gph13} and Corollary 6.10 in \cite{GPHS} reduces the computation of $[\mathcal{M}]$ to the computation of the class $[\mathcal{M}{(\underline{n})}_{\underline{d}}^{\on{gen-surj}}]$ of the moduli stack of generically surjective chains of type $(\underline{n})=(n,n,\dots,n)$ and degree $\underline{d}=(d_1,d_2,\dots, d_r)$, which classifies the following data
\begin{equation}\label{eq: chains}
    E_1\xrightarrow{\phi_1}E_2\xrightarrow{\phi_2}E_3\xrightarrow{\phi_3}\cdots E_{r-1} \xrightarrow[]{\phi_{r-1}} E_r,
\end{equation}
where each $E_i$ is a vector bundle of rank $n$ and degree $d_i$, and each $\phi_i$ is a generically surjective map. The class $[\mathcal{M}{(\underline{n})}_{\underline{d}}^{\on{gen-surj}}]$ is computed in Corollary 4.10 of \cite{GPHS}, where the data in (\ref{eq: chains}) are considered as consecutive Hecke modifications. They considered the map $\mathcal{M}{(\underline{n})}_{\underline{d}}^{\on{gen-surj}}\longrightarrow\on{Bun}_{d_r}\times\displaystyle\prod_{i=1}^{r-1} C^{(d_{i+1}-d_i)}$ which maps the data in (\ref{eq: chains}) to $(E_r, \on{supp}(E_{i+1}/\phi_i(E_i))_{i=1,2,\dots,r-1})$. By identifying fibres of this map, they proved that \[[\mathcal{M}{(\underline{n})}_{\underline{d}}^{\on{gen-surj}}]=\mathbb{L}^{\chi}[\on{Bun}_{d_r}]\displaystyle\prod_{i=1}^{r-1} [(C\times \mathbb{P}^{{n-1})^{(d_{i+1}-d_i)}]}.
\]
Combining the pull-back diagram \ref{eq: pull back jacobian}, it follows that the class $[\hMm\times J]\in \hat{K}_{0}(\on{Var}_{\mathbb{C}})$ can be written as a linear combination of classes of the form $[\widetilde{\displaystyle\prod_{j}\on{Bun}_{d_j}\times\prod_{i} C^{(l_i)}}]\times \mathbb{L}^{m}$, where $\widetilde{\displaystyle\prod_{j}\on{Bun}_{d_j}\times\prod_{i} C^{(l_i)}}$ fits into the following pull-back diagram
\bd
\xymatrix{
\widetilde{\displaystyle\prod_{j}\on{Bun}_{d_j}\times\prod_{i} C^{(l_i)}} \ar[rrrrr] \ar[d] &&&&&J \ar[d]^{[n]} \\
  {{\displaystyle\prod_{j}\on{Bun}_{d_j}\times\prod_{i} C^{(l_i)}}}\ar[rrrrr]^{\displaystyle\bigotimes_j\on{det}{(m_j)}\otimes\bigotimes_i \mu(k_i)} &&&&&J.
}
\ed
Here $\on{det}(m):\on{Bun}_d\to J$ maps $E\in \on{Bun}_d$ to $\on{det}(E)^m\otimes \mathcal{O}_C(-mdc_0)$.

It is shown in \cite{BD} that the class $[\on{Bun}_d]$ can be written as a linear sum of classes of the form $[\displaystyle\prod_{i} C^{(l_i)}]\times \mathbb{L}^{m}$.\footnote{The computation in \cite{BD} Section 6 is written for the case of vector bundles with fixed determinant, but the same arguments work for $\on{Bun}_d$.} As explained in \cite{BD} Section 6, this computation is reduced to computing the class of $\on{Div}(D)$, where $\on{Div}(D)$ is the Quot scheme classifying locally free subsheaves
\[
 E\hookrightarrow \mathcal{O}_C(D)^n,
\]
of rank $n$ and degree $d$. Here $D$ is a effective divisor which we can assume to be supported solely on $c_0$. There is a $\mathbb{G}_m^n$-action on $\on{Div}(D)$, and the computation of $[\on{Div}(D)]$ can be reduced to computing the class of the fixed point loci. Those fixed points correspond to inclusions of the form 
\[
 \bigoplus_{i=1}^n \mathcal{O}_C(D-E_i)\hookrightarrow\mathcal{O}_C(D)^n,
\]
where $E_1, E_2, \dots, E_n$ are effective divisors such that $d= n\on{deg}D-\sum \on{deg}E_i$. It follows that the components of the fixed point loci are of the form ${\displaystyle\prod_{i=1}^n C^{(l_i)}}$. Note that the determinant map $\on{det}(m):\on{Bun}_d\to J$ is invariant with respect to the $\mathbb{G}_m^n$-action. Since we assume $D$ is supported solely on $c_0$, the map $\on{det}(m)$ corresponds exactly to
\[
{\displaystyle\prod_{i=1}^n C^{(l_i)}}\xrightarrow{\displaystyle\bigotimes_i \mu(m)} J
\]
when restricted to those fixed point components. 
\end{proof}

We denote ${\displaystyle\prod_{i} C^{(l_i)}}$ by $Y(\underline{l})$ and $\widetilde{\displaystyle\prod_{i} C^{(l_i)}}$ (defined by the pull-back diagram \ref{eq: pull-back jacobian main prop}) by $\widetilde{Y}_n(\underline{l}, \underline{k})$.  Proving Theorem \ref{thm:SLn} is thus reduced 
to establishing the following statement:

\begin{proposition}\label{prop: SL_n torsion-free main prop}
The cohomology groups of $\widetilde{Y}_n(\underline{l}, \underline{k})$ are torsion-free. 
\end{proposition}

Our proof of Proposition \ref{prop: SL_n torsion-free main prop} relies on the following two technical lemmas

\begin{lemma}\label{lemma: homologically trivial}
Let $E\longrightarrow B$ be a Serre fibration over a path-connected base $B$, and let $F$ be the fibre. Assume the cohomology of $F$ and $B$ are torsion-free. If the cohomology of $E$ satisfy the following inequality
\begin{equation}\label{eq: cohomology of serre fibration}
    \on{dim}_\mathbb{Q} H^{i}(E,\mathbb{Q})\geq\on{dim}_\mathbb{Q} H^{i}(B\times F,\mathbb{Q})
\end{equation}
for each $i$, then we have 
\[
    H^{i}(E,\mathbb{Z})\cong H^{i}(B\times F,\mathbb{Z})
\]
for each $i$. 
\end{lemma}
\begin{proof}
Consider the Leray-Serre spectral sequence
\[
    E_2^{p,q}=H^{p}(B, H^{q}(F))\Longrightarrow H^{p+q}(E).
\]
Any non-trivial local system of $H^{q}(F)$ would result in \[\on{rank}H^{0}(B, H^{q}(F))< \on{rank}H^{q}(F),
\]
which contradicts (\ref{eq: cohomology of serre fibration}). It follows that $H^{p}(B, H^{q}(F))\cong H^{p}(B)\otimes_{\mathbb{Z}}H^{q}(F)$ which is torsion-free. The spectral sequence must degenerate at $E_2$, since any non-zero differential would contradict (\ref{eq: cohomology of serre fibration}). 
\end{proof}

\begin{lemma}\label{Lemma: affine open cohomology}
Let $X$ be a smooth projective variety over $\mathbb{C}$ and let $Y\subset X$ be a smooth closed subvariety such that $U\coloneqq X-Y$ is affine. If the cohomology of $X$ is torsion-free, then the cohomology of $Y$ is also torsion-free. 
\end{lemma}

\begin{proof}
Assume $X$ is of dimension $n$. Since $U$ is affine, it is homotopy equivalent to an $n$-dimensional CW complex. It follows that $H_c^{i}(U)=0$ when $i<n$ and $H_c^{n}(U)$ is torsion-free. By the long exact sequence of (compact supported) cohomology, we have $H^{i}(Y)\cong H^{i}(X)$ when $i<n-1$ and $H^{n-1}(Y)$ fits into
\[
    0\longrightarrow H^{n-1}(X)\longrightarrow H^{n-1}(Y)\longrightarrow H^{n}(U)\longrightarrow \cdots.
\]
Since both $H^{n-1}(X)$ and $H^{n}(U)$ are torsion-free, $H^{n-1}(Y)$ is also torsion-free. For $i>n-1$, $H^{i}(Y)$ is torsion-free by Poincar\'e duality.
\end{proof}

We start proving Proposition \ref{prop: SL_n torsion-free main prop} by considering the case when $\underline{l}=(l)$ has exactly one element. We have the following

\begin{lemma}\label{Lemma: one symmetric product}
The cohomology groups of $\widetilde{Y}_n(l, k)$ are torsion-free. 
\end{lemma}

\begin{proof}
Let $m=\on{gcd}(k,n)$. Note that $\widetilde{Y}_n(l, k)$ is defined by the following pull-back diagrams
\bd
\xymatrix{
 \widetilde{Y}_n(l, k) \ar[r] \ar[d] & \displaystyle\coprod^{m}_{i=1} J \ar[rr]^{[k/m]}\ar[d]^{[{n}/{m}]} &&J\ar[d]^{[n]}\\
Y(l)\ar[r]^{\mu(1)} & J \ar[rr]^{[k]}&&J.
}
\ed
It follows that $\widetilde{Y}_n(l, k)$ is a disjoint union of $m$ copies of $\widetilde{Y}_{n/m}(l, 1)$, therefore it is enough to prove the lemma when $k=1$. 

We start by proving $H^{*}(\widetilde{Y}_{n}(l, 1),\mathbb{Z})$ is torsion-free when $l>2g-2$. 
\bd
\xymatrix{
 \widetilde{Y}_n(l, 1)\ar[d]\ar[rr] &&J\ar[d]^{[n]}\\
Y(l)\ar[rr]^{\mu(1)}&&J.
}
\ed
In this case, $Y(l)\xrightarrow{\mu(1)} J$ is a homologically trivial fibration with fibre $\mathbb{P}^{l-g}$. Since $\widetilde{Y}_n(l, 1)\longrightarrow Y(l)$ is a principal bundle with finite structure group, we have
\begin{equation}
    \on{dim}_\mathbb{Q} H^{i}(\widetilde{Y}_{n}(l, 1),\mathbb{Q})\geq\on{dim}_\mathbb{Q} H^{i}(Y(l),\mathbb{Q})
\end{equation}
for any $i$. It follows from Lemma \ref{lemma: homologically trivial} that 
\[
H^{*}(\widetilde{Y}_{n}(l, 1),\mathbb{Z})\cong H^{*}(Y(l),\mathbb{Z})\cong H^{*}(\mathbb{P}^{l-g}\times J, \mathbb{Z}),
\]
which is torsion-free. 

For the case when $l\leq 2g-2$, we argue by reverse induction on $l$. Consider the closed embedding $Y(l)\hookrightarrow Y(l+1)$ defined by adding $c_0\in C$. Note that $U\coloneqq Y(l+1)-Y(l)$ is affine: it is the symmetric product of an affine curve $C-\{c_0\}$. Let $\widetilde{U}$ be the preimage $U$ in $\widetilde{Y}_{n}(l, 1)$, which is also affine. If the cohomology of $\widetilde{Y}_{n}(l+1, 1)$ is torsion-free, Lemma \ref{Lemma: affine open cohomology} implies the cohomology of $\widetilde{Y}_{n}(l, 1)$ is also torsion-free. 
\end{proof}

\begin{proof}[Proof of Proposition \ref{prop: SL_n torsion-free main prop}]

 We prove Proposition \ref{prop: SL_n torsion-free main prop} by induction on the length $|\underline{l}|$ of $\underline{l}$. We assume the cohomology of $\widetilde{Y}_n(\underline{l}, \underline{k})$ is torsion-free when $|\underline{l}|<r$, and the goal is to show the cohomology of $\widetilde{Y}_n(\underline{l}, \underline{k})$ is torsion-free when $|\underline{l}|=r$. 

\emph{Step 1.} In this step, we show the cohomology groups of $\widetilde{Y}_n(\underline{l}, \underline{k})$ are torsion-free when $l_1>2g-2$. Let $\underline{l}'=(l_2, l_3,\dots,l_r)$ and let $\underline{k}'=(k_2, k_3,\dots,k_r)$. Let $m=\on{gcd}(k_1,n)$. Note that $\widetilde{Y}_n(\underline{l}, \underline{k})$ fits into the following pull-back diagrams
\bd
\xymatrix{
 \widetilde{Y}_n(\underline{l}, \underline{k})\ar[d]\ar[rrr] &&&J\ar[d]^{[n/m]}\\
Y(l_1)\times\widetilde{Y}_m(\underline{l}', \underline{k}')\ar[d] \ar[rrr]^{\mu(k_1/m)\otimes \widetilde{\mu}(\underline{k}')}&&&J\ar[d]^{[m]}\\
{Y}(\underline{l}, \underline{k})\ar[rrr]^{\mu(\underline{k})} &&&J.
}
\ed
The cohomology of $\widetilde{Y}_m(\underline{l}', \underline{k}')$ is torsion-free by induction hypothesis. Consider the fibration $$\widetilde{Y}_n(\underline{l}, \underline{k})\longrightarrow \widetilde{Y}_m(\underline{l}', \underline{k}')$$ with fibre $\widetilde{Y}_{n/m}(l_1, k_1/m)$. Since $n/m$ and $k_1/m$ are coprime, we have $\widetilde{Y}_{n/m}(l_1, k_1/m)\cong \widetilde{Y}_{n/m}(l_1, 1)$, therefore $H^{*}(\widetilde{Y}_{n/m}(l, k_1/m),\mathbb{Z})\cong H^{*}(Y(l),\mathbb{Z})$. Similar arguments as in the proof of Lemma \ref{Lemma: one symmetric product} implies the fibration $\widetilde{Y}_n(\underline{l}, \underline{k})\longrightarrow \widetilde{Y}_m(\underline{l}', \underline{k}')$ is homologically trivial, which implies the cohomology of $\widetilde{Y}_n(\underline{l}, \underline{k})$ is torsion-free. 

\emph{Step 2.}
In this step, we prove the cohomology of $\widetilde{Y}_n(\underline{l}, \underline{k})$ is torsion-free by reverse induction on $\displaystyle\sum_{i=1}^r l_i$. When $\displaystyle\sum_{i=1}^r l_i>r(2g-2)$, the Proposition is true by Step 1. For the general case, we consider the map $Y(l)\longrightarrow Y(l+1)$ defined by adding $c_0\in C$. Let $U(l+1)=Y(l+1)\backslash Y(l)$ and let $U(\underline{l}+1)=\displaystyle\prod_{i=1}^r U(l_i+1)$. Let $\widetilde{U}(\underline{l}+1)$ be the preimage of $U(\underline{l}+1)$ in $\widetilde{Y}_n(\underline{l}+1, \underline{k})$. Then we have the following equality in $K_0(\on{Var}_{\mathbb{C}})$
\begin{equation}\label{eq: induction on symmetric product}
    [U(\underline{l}+1)]=\prod_{i=1}^r([Y(l_i+1)]-[Y(l_i)]). 
\end{equation}
After multiplying out the right hand side, (\ref{eq: induction on symmetric product}) can be considered as an equation in the Grothendieck group $K_0(\on{Var}_{{Y}_n(\underline{l}+1)})$ of varieties over ${Y}_n(\underline{l}+1)$.
By pulling back along the covering map $\widetilde{Y}_n(\underline{l}+1, \underline{k})\longrightarrow {Y}_n(\underline{l}+1)$, we get an equation in $K_0(\on{Var}_{\widetilde{Y}_n(\underline{l}+1, \underline{k})})$, which in turn gives an equation in $K_0(\on{Var}_{\mathbb{C}})$ by forgetting the relative structure. This equality implies the class $[\widetilde{Y}_n(\underline{l}, \underline{k})]$ can be written as a linear combination of $[\widetilde{Y}_n(\underline{l}+1, \underline{k})-\widetilde{U}(\underline{l}+1)]$ together with classes of the form $[\widetilde{Y}_n(\underline{l}', \underline{k}')]$ with $\displaystyle\sum_{i=1}^r l'_i> \displaystyle\sum_{i=1}^r l_i$. Since $\widetilde{U}(\underline{l}+1)$ is affine, Lemma \ref{Lemma: affine open cohomology} implies the cohomology $H^{*}(\widetilde{Y}_n(\underline{l}+1, \underline{k})-\widetilde{U}(\underline{l}+1), \mathbb{Z})$ is torsion-free. This completes the induction step.
\end{proof}

\begin{corollary}
For $\alpha \in \Br(\Mm)$ the twisted complex $K$-theory groups $\KU^{*}(\Mm^{\an},\tilde{\alpha})$ are torsion-free.
\end{corollary}
\begin{proof}
Analogous to the proof of Corollary \ref{cor:GLn}.
\end{proof}

\subsection{$\PGL_n$}

In this subsection we prove that complex $K$-theory of the DM stack $\cMm$ is torsion-free, even in the presence of a Brauer twist. It is important to emphasise that there are non-zero torsion classes in $H^*(\cMm^{\an},\Zb) = H^*_{J^{\an}}(\Mm^{\an},\Zb)$, and it is thus necessary to pass to complex $K$-theory in order to obtain vanishing of torsion.

\begin{rmk}
The moduli space $\hMm$ is simply connected by \cite{cp} (see the paragraph immediately below Theorem 1.1). Thus, the orbifold fundamental group of $\cMm^{\an}$ is isomorphic to $\Gamma = J[n]$. We infer $H^1(\cMm^{\an},\Z/n/\Zb) \simeq (\Zb/n\Zb)^{2g}$, and thus by the Bockstein sequence, $H^2(\cMm,\Zb)[n] \simeq (\Zb/n\Zb)^{2g}$. This shows that singular cohomology of $\cMm$ has non-trivial torsion elements.
\end{rmk}

\begin{theorem}\label{thm:PGLn}
The complex $K$-theory groups $\KU^{*}(\cMm^{\an},\tilde{\alpha}) = \KU^{*}_J(\Mm^{\an},\tilde{\alpha})$ are torsion-free for an arbitrary Brauer class $\alpha \in \Br(\cMm)$.
\end{theorem}

The proof will be given after the following auxiliary results. Our exposition will use the duality formalism for commutative group stacks and more precisely of Beilinson $1$-motives (see the appendices in \cite{Ar,Tr,CZ}). In these sources the base is assumed to be schematic rather than a stack. This is no problem, since definitions and results of \emph{loc. cit.} are \'etale local. We therefore say that $\mathcal{G} \to \mathcal{S}$ is 
a commutative group stack, respectively a Beilinson $1$-motive, if its pullback along an \'etale atlas of $\mathcal{S}$ satisfies those properties. 

For the following lemmas we fix a DM stack $\mathcal{S}$ and a relative Beilinson $1$-motive $\pi\colon \Aa \to \mathcal{S}$ (e.g., an abelian $\mathcal{S}$-scheme, $B_{\mathcal{S}}G$, where $G$ is a finite abelian group). 

\begin{lemma}\label{lemma:duality}
For every $\Aa$-torsor $\pi\colon \Tt \to \mathcal{S}$ there exists a $\G_m$-gerbe $\beta$ on $\Aa^{\dual}$ such that we have an $\mathcal{S}$-linear derived equivalence
$$\Perf(\Tt) \cong \Perf^{\beta}(\Aa^{\dual}).$$
\end{lemma}
\begin{proof}
Let us denote by 
\begin{equation}\label{eqn:fibre-seq}
\Aa \to P_{\Tt} \xrightarrow{g} \Zb\end{equation}
a fibre sequence of commutative group stacks such that $g^{-1}(1) = \Tt$ as an $\Aa$-torsor. The duality functor is defined as 
$$(-)^{\dual}=\Hom(-,B\G_m).$$
It sends \eqref{eqn:fibre-seq} to
$$B\G_m \to P_{\Tt}^{\dual} \to \Aa^{\dual}.$$
Let us denote $P_{\Tt}^{\dual}$ by $\mathcal{G}_{\beta}$. As is visible from the fibre sequence above, it is the underlying stack of a $\G_m$-gerbe (or $B\G_m$-torsor). The class $\beta \in \Ext^2(\Aa^{\dual},\G_m)$.

According to \cite[Theorem A.4.6]{CZ}, there is a derived equivalence
$$\FM\colon\D(P_{\Tt}) \cong \D(\mathcal{G}_{\beta}).$$
Again, we emphasise that \emph{loc. cit.} works under the assumption that the base $\mathcal{S}$ is a scheme. We must therefore replace $\mathcal{S}$ by an \'etale atlas. Faithfully flat descent theory implies that the statement also holds over the DM stack $\mathcal{S}$.
Furthermore, we have $\D(\Tt) \simeq \D^{\beta}(\Aa^{\dual})$ (see \cite[Theorem A.7.2]{CZ}). Passing to compact objects and using \cite[Example 9.3]{hr}, we obtain the sought-for equivalence.
\end{proof}

Let $\phi\colon \Aa \to \Aa'$ denote a finite \'etale morphism of abelian schemes over $\mathcal{S}$. There is a dual morphism $\phi^{\dual}\colon \Aa'^{\dual} \to \Aa$, which is also finite \'etale. Furthermore, for every $\Aa$-torsor $\Tt$, we obtain an induced $\Aa'$-torsor $\Tt'$ by virtue of the natural map
$$H^1_{\text{\'et}}(\mathcal{S},\Aa) \to H^1_{\text{\'et}}(\mathcal{S},\Aa').$$

\begin{lemma}\label{lemma:duality-quotient}
Using the same notation as in Lemma \ref{lemma:duality}, there is a derived equivalence
$$\Perf(\Tt') \cong \Perf^{\beta'}(\Aa'^{\dual}),$$
where $\beta'=(\phi^{\dual})^*\beta$.
\end{lemma}
\begin{proof}
Using the notation from the proof of Lemma \ref{lemma:duality}, we intend to establish the following duality relation of commutative group stacks:
\[
\xymatrix{
P_{\Tt} \ar[rd] \ar@{<.>}[rrrd] & & {\color{gray}\text{duality}} & & \mathcal{G}_{\beta'} \ar[ld] \\
& P_{\Tt'} \ar[rd] \ar@{<.>}[rrru] & & \mathcal{G}_{\beta} \ar[ld] & \\
& & S & &
}
\]
Hence, we are required to prove 
$$P_{\Tt'} \simeq \mathcal{G}_{\beta'}^{\dual}.$$
Let us denote the kernel of the isogeny $\phi$ by $K$. By virtue of definition, $K$ is a finite \'etale group scheme over $S$, and we have a commutative diagram where the rows are fibre sequences
\[
\xymatrix{
K \ar[r] \ar@{=}[d] & \Aa \ar[r] \ar[d] & \Aa' \ar[d] \\
K \ar[r] & P_{\Tt} \ar[r] & P_{\Tt'}.
}
\]
Applying the duality functor $(-)^{\dual} = Hom(-,B\G_m)$ we obtain the following commutative ladder (of fibre sequences):
\[
\xymatrix{
\Aa'^{\dual} \ar[r]^{\phi^{\dual}} \ar@{<-}[d] & \Aa^{\dual} \ar[r] \ar@{<-}[d] & BK^* \ar@{=}[d] \\
P_{\Tt'}^{\dual} \ar[r] & \mathcal{G}_{\beta} \ar[r] & BK^*.
}
\]
Here, we made use of the equivalence $P_{\Tt}^{\dual} \simeq \mathcal{G}_{\beta}$. The left-hand commutative square yields a morphism 
$$P_{\Tt'}^{\dual} \to (\phi^{\dual})^*\mathcal{G}_{\beta}.$$
Since it is a morphism of $B\G_m$-torsors on $\Aa'$ (i.e. $\G_m$-gerbes), it is an equivalence. This completes the proof, as we have an equivalence $(\phi^{\dual})^*\mathcal{G}_{\beta} \simeq \mathcal{G}_{\beta'}$.
\end{proof}

\begin{lemma}\label{lemma:duality-twist}
Let $\alpha$ be an arbitrary $\G_m$-gerbe on $\mathcal{S}$. Assume that the DM stacks
\[
\xymatrix{
\mathcal{U} \ar[r]^g & \mathcal{S}  & \mathcal{V} \ar[l]_h
}
\]
are \emph{Fourier-Mukai partners} relative to $\mathcal{S}$. Then, there is also a derived equivalences
$$\Perf^{g^*\alpha}(\mathcal{U}) \cong \Perf^{h^*\alpha}(\mathcal{V}).$$
\end{lemma}
\begin{proof}
Let $\F$ be a complex of quasi-coherent sheaves on $\mathcal{U} \times_{\mathcal{S}} \mathcal{V}$ which induces an equivalence $\FM_{\F}\colon \D(\mathcal{U}) \xrightarrow{\cong} \D(\mathcal{V})$. The exterior product $g^*\alpha^{-1} \boxtimes h^*\alpha$
 of gerbes on $\mathcal{U} \times_{\mathcal{S}} \mathcal{V}$
splits, since $pr_1^*g^*\alpha \simeq pr_2^*h^*\alpha$. We may therefore consider $\F$ as being $g^*\alpha^{-1} \boxtimes h^*\alpha$-twisted, and hence obtain a twisted Fourier-Mukai functor
$${}^{g^*\alpha}\FM_{\F}^{h^*\alpha}\colon \D{g^*\alpha}(\mathcal{U}) \to \D^{h^*\alpha}(\mathcal{V}).$$
\'Etale locally on $\mathcal{S}$, the gerbe $\alpha$ splits. We therefore see that ${}^{g^*\alpha}\FM_{\F}^{h^*\alpha}$ is an equivalence. Passing to compact objects, we obtain the desired equivalence of DG categories of perfect complexes $\Perf^{g^*\alpha}(\mathcal{U}) \cong \Perf^{h^*\alpha}(\mathcal{V})$. Here, we use \cite[Example 9.3]{hr}.
\end{proof}

We will also need \cite[Proposition 7.1]{ht}. For this purpose we recall some notation. Recall that elements of $\Gamma=J[n]$ are by definition $n$-torsion line bundles. By virtue of the Kummer sequence
$$0 \to \mu_r \to \G_m \xrightarrow{[r]} \G_m \to 0,$$
an element $\gamma \in \Gamma$ of order $r$ gives rise to a $\mu_{r}$-torsors $\widetilde{X}_{\gamma}\to X$. We will denote by $\widetilde{M}_{\gamma}$ the moduli space $\Mm_{\frac{n}{r}}(\widetilde{X}_{\gamma})$. The $\mu_r$-action on $\widetilde{X}_{\gamma}$ induces an action on $\widetilde{M}_{\gamma}$.

\begin{proposition}[Hausel--Thaddeus]\label{prop:ht}
The group $\mu_r$ acts freely on $\widetilde{M}_{\gamma}$. The fixpoint locus $\Mm^{\gamma}$ is isomorphic to $\widetilde{M}_{\gamma}/\mu_r$.
\end{proposition}
\begin{proof}
This is \cite[Proposition 7.1]{ht}
\end{proof}
%

We will need a further refinement of this result. Let $L_{\gamma}$ denote the line bundle induced by the $\Gamma$-equivariant gerbe $\alpha$ on $\Mm^{\gamma}$ (see \cite[Section 4]{ht}). The relation 
$$L_{\gamma_1}\otimes L_{\gamma_2} \simeq L_{\gamma_1\gamma_2}$$
for $\gamma_1,\gamma_2 \in \langle \gamma \rangle = c$ gives rise to the structure of a sheaf of algebras on $\bigoplus_{\gamma' \in c} L_{\gamma'}$. According to Lemma \ref{lemma:cyclic} it is commutative. The relative spectrum will be denoted by $C_{c}^{\alpha}$. It is a $c^*$-torsor over the fixpoint locus $\Mm^{c}= \Mm^{\gamma}$.

\begin{corollary}\label{cor:Cc}
There is an isomorphism
\[
\xymatrix{
C_{c}^{\alpha} \ar[rd] \ar[rr]^{\simeq} & & \widetilde{M}_{\gamma} \ar[ld] \\
& \Mm^{\gamma}. & 
}
\]
\end{corollary}
\begin{proof}
It suffices to construct a morphism of $\mu_r$-torsors $\widetilde{M}_{\gamma} \to C_{c}^{\alpha}$, as it will then automatically be an isomorphism. Recall that $\widetilde{M}_{\gamma}$ is the moduli space of rank $\frac{n}{r}$ degree $d'$ Higgs bundles degree on the finite \'etale cover $\widetilde{X}_{\gamma} \xrightarrow{\pi} X$. Since $d'$ is coprime to $\frac{n}{r}$, we see that there exists a universal family $(\widetilde{\mathcal{E}},\widetilde{\theta})$ on $\widetilde{M}_{\gamma} \times \widetilde{X}_{\gamma}$. The pushforward $(\pi_*\widetilde{\mathcal{E}},\tr(\widetilde{\theta}))$ is a universal family of rank $n$ and degree $d$. It is $\mu_r$-equivariant with respect to the action on $\widetilde{M}_{\gamma}$. Since the gerbe $\alpha$ measures the obstruction to the existence of a universal family, this implies that $\alpha$ splits as a $c$-equivariant gerbe when pulled back along $\widetilde{M}_{\gamma} \to \Mm^{\gamma}$. This concludes the proof.
\end{proof}

\begin{corollary}\label{cor:no-p-torsion}
Let $p$ be a prime and let $\Gamma' \subset \Gamma$ denote a subgroup of order prime to $p$. Then, we have $\KU^{*,\tilde{\alpha}}_{\Gamma'}(\Mm^{\an,\gamma})[p] = 0$.
\end{corollary}
\begin{proof}
If the rank is $1$ this assertion holds, since the moduli space in question is then a cotangent bundle of an abelian variety. Assume by induction that the assertion holds for all ranks $< n$.

We will apply the concentration theorem (a.k.a. localisation theorem). Let $\pf \subset R(\Gamma')$ be a prime ideal above $p$. We denote by $c \subset \Gamma'$ the corresponding cyclic subgroup. There are two cases to consider.

\noindent\textit{If $c$ is the trivial subgroup}: then $\KU^{*,\tilde{\alpha}}_{\Gamma'}(\Mm^{\an,\gamma})_{\pf}$ is simply given by Borel-equivariant twisted complex $K$-theory localised at $p$:
$$\KU^{*,\tilde{\alpha}}_{\Gamma'}(\Mm^{\an,\gamma})_{\pf} \simeq \pi_*\left(\KU^{\tilde{\alpha}}(\Mm^{\an,\gamma})^{h\Gamma'}\right) \otimes \Zb_{(p)}.$$
This follows from the Atiyah--Segal completion theorem. By Corollary \ref{cor:GLn}, $\KU^{\tilde{\alpha}}(\Mm^{\an,\gamma})$ does not have $p$-torsion. Since $\Gamma'$ is also assumed to be $p$-torsion-free, we see that the same property holds for homotopy invariants $\KU^{\tilde{\alpha}}(\Mm^{\an,\gamma})^{h^{\Gamma'}}$.

\noindent\textit{Let $c=\langle \gamma \rangle$, where $\gamma \neq 1$}: The derived equivalence of Lemma \ref{lemma:duality}
$$\Perf^{\alpha}([\Mm^{\gamma}/\Gamma']) \cong \Perf^{\alpha}([C_c^{\alpha}/(\Gamma'/c)])$$
implies 
$$\KU^{\tilde{\alpha}}_{\Gamma'}(\Mm^{\an,\gamma}) \simeq \KU^{\tilde{\alpha}}_{\Gamma'/c}(C^{\alpha}_c).$$
By virtue of Corollary \ref{cor:Cc}, the right-hand side can be rewritten as $\KU^{\tilde{\alpha}}_{\Gamma'/c}(\widetilde{M}_c)$. As explained in Proposition \ref{prop:ht}, this space is itself isomorphic to a moduli space of Higgs bundles of rank $\frac{n}{r}$ (where $r$ is the order of $c$) on a curve $\widetilde{X}_c$. We have therefore reduced the question of vanishing of $p$-torsion to the one for a lower rank $\frac{n}{r}$. This concludes the proof by induction.
\end{proof}

Finally, we will need the following lemma which will allow us to infer vanishing of torsion.

\begin{lemma}\label{lemma:p-torsion}
Let $\F$ be a locally constant sheaf of spectra on a topological torus $\Tb^m=(\Sb^1)^m$ with fibre $\Eb\in \Spectra$ having finitely generated homotopy groups. Assume that $\F$ is trivialisable after pullback along the map
$T \xrightarrow{[p^{\nu}]} T$ of degree $p^{\nu}$ where $p$ is a fixed prime. Then, if $\pi_*(\Eb)[p] \neq 0$, we also have $\pi_{*}(\Gamma(\Tb^m,\F))[p] \neq 0$.
\end{lemma}
\begin{proof}
We will prove this statement by induction on $m$. The starting case of a $0$-dimensional torus holds, since $\Gamma(\Tb^0,\F) \simeq \Eb$. Let us assume that the assertion holds for tori of dimension $\leq m-1$.

For $i=1,\dots,m$ we denote by $\mon_i\colon \Eb \to \Eb$ the monodromy of $\F$ along the $i$-th factor $j_i\colon \Sb^1 \hookrightarrow \Tb$. Let $q\colon \Tb^m \to \Tb^{m-1}$ denote the projection with kernel $j_m$.

The pushforward $q_*\F$ is a locally constant sheaf with fibre $\Gamma(\Sb^1,j_m^*\F)$, acted on by the monodromy operators $\mon_1,\dots,\mon_{m-1}$. By induction, it therefore suffices to prove 
$$\pi_*(\Gamma(\Sb^1,j_m^*\F))[p] \neq 0.$$
The Mayer--Vietoris sequence for $\Sb^1$ implies that $\Gamma(\Sb^1,j_m^*\F)$ fits into a fibre sequence of spectra
$$\Gamma(\Sb^1,j_m^*\F) \to \Eb \xrightarrow{\alpha} \Eb \xrightarrow{\bullet},$$
where $\alpha=\mon_m - 1$. We obtain the following long exact sequence of homotopy groups
\begin{equation}\label{eqn:exactness}
\cdots \to \pi_i(\Gamma(\Sb^1,j_m^*\F)) \to \pi_i(\Eb) \xrightarrow{\alpha} \pi_i(\Eb) \to \pi_{i-1}(\Gamma(\Sb^1,j_m^*\F)) \to \cdots.
\end{equation}
The map $\alpha$ is of the form $\mon_m - 1$, where $\mon_m$ is an automorphism. According to Claim \ref{claim:below} below, $\mon_m$ fixes a $p$-torsion class in $\pi_*(\Eb)[p]$.
This shows that the map of finite abelian groups
$$\pi_i(\Eb)[p^{\infty}] \xrightarrow{\alpha} \pi_i(\Eb)[p^{\infty}]$$
has non-trivial kernel. Thus, there exists an element $w \in \pi_i(\Eb)[p^{\infty}]$ which does not lie in the image of $\pi_i(\Eb)[p^{\infty}]$. Assume by contradiction that $w = \alpha(v)$, where $v \in \pi_i(\Eb)$. Write $v= v' + v_{p^{\infty}}$, where $v_{p^{\infty}} \in \pi_i(\Eb)[p^{\infty}]$ denotes the $p$-power-torsion part, and $v'$ belongs to the complementary component $\pi_i(\Eb)'$. Since $\mon_m$ is an automorphism, it fixes $\pi_i(\Eb)'$. Therefore, we have  $\alpha(v') = \mon_m(v')-v' \in \pi_i(\Eb)'$. By assumption on $w$ we must have $\alpha(v')=0$. This shows that $w = \alpha(v_{p^{\infty}})$, which is a contradiction.

Using exactness of \eqref{eqn:exactness} we exhibit a non-zero $p$-torsion class in $\pi_{i-1}(\Sb^1,j_m^*\F)$. This concludes the proof of the lemma modulo the claim below.
\end{proof}
\begin{claim}\label{claim:below}
Let $i \in \Zb$ be a degree such that $\pi_i(\Eb)[p] \neq 0$. Then, there exists $v \in \pi_i(\Eb)[p]$ such that $\mon_m(v) = v$.
\end{claim}
\begin{proof} To see this, we consider the $\Fb_p$-vector space $V=\pi_0(\Eb)[p]$. The group $\Sb^1[p^{\nu}]$ is a cyclic group of order $p^{\nu}$ acting on $\Eb$ through $\mon_m$. Let $v \in \bar{V} = V \otimes_{\Fb_p} \bar{\Fb}_p$ be an eigenvector of $\mon_m$. And, let $\lambda\in\bar{\Fb}_p^{\times}$ be the corresponding eigenvalue. 
Since $\mon_m^{p^{\delta}} = 1$, we have $\lambda^{p^{\delta}} = 1$. This implies $\lambda = 1$, and thus $v$ is a fixvector and therefore we have $v \in V$.
\end{proof}
We are now ready to establish vanishing of torsion for the $\PGL_n$-case.
\begin{proof}[Proof of Theorem \ref{thm:PGLn}]
Let $J=\Jac(X)$ and consider the $J$-torsor 
$$\pi\colon \Mm^0  \to \cMm$$
where $\Mm^0$ denote the subvariety corresponding to trace-free Higgs fields. The inclusion $(\Mm^0)^{\an} \hookrightarrow \Mm^{\an}$ is a deformation retract.

We denote the gerbe corresponding to the $J$-torsor $\pi$ by $\beta$ (see Lemma \ref{lemma:duality}). It is a gerbe on $\cMm \times J^{\dual} = \cMm \times J$. 
We then have a derived equivalence
$$\Perf^{\pi^*\alpha}(\Mm^0) \cong \Perf^{\pr_1^*\alpha \otimes \beta}(\cMm \times J)$$
by virtue of Lemma \ref{lemma:duality-twist}.
\begin{claim}\label{claim:n}
The gerbe $\beta$ splits when pulled back along $[n]\colon J \to J$.
\end{claim}
To see this, it suffices to observe that the $J$-torsor $\pi$ is sent to the trivial torsor by the map
$$H^1_{\text{\'et}}(\cMm,J) \xrightarrow{[n]} H^1_{\text{\'et}}(\cMm,J).$$
Indeed, we have a commutative diagram
\[
\xymatrix{
[\Mm^0 / \Gamma] \ar[r]^{(\pi,\det)}_{\simeq} \ar[d]_{\pi} & \cMm \times J \ar[ld]^{pr_2} \\
\cMm . & 
}
\]

Let us factor $n$ as 
$$n=p^{\nu}\cdot{}n',\text{ where }\gcd(p,n')=1.$$
We denote by $\beta'$ the pullback of $\beta$ along $J \xrightarrow{[n']} J$. According to Lemma \ref{lemma:duality-quotient} there is an equivalence 
$$\Perf^{\pi^*\alpha}([\Mm^0/\Gamma]) \cong \Perf^{\pr_1^*\alpha \otimes \beta'}(\cMm \times J)$$
and thus an equivalence of spectra 
\begin{equation}\label{eqn:kucm}
\KU^{\pi^*\tilde{\alpha}}_{\Gamma'}(\Mm^{\an}) \simeq \KU^{\pi^*\tilde{\alpha}\tilde{\beta'}}(\cMm^{\an} \times J^{\an}).
\end{equation}

Assume by contradiction that $p$ is a prime number such that $\KU^*(\cMm^{\an},\tilde{\alpha})[p] \neq 0$. We now consider the projection $pr_2\colon \cMm \times J \to J$ and the following sheaf of spectra
$$\F=(pr_2)_*\underline{\KU}^{\pi^*\alpha\beta}_{\cMm^{\an} \times J^{\an}}.$$

This is a locally constant sheaf of spectra with fibre $\Eb = \KU^{\tilde{\alpha}}(\cMm^{\an})$. By virtue of Claim \ref{claim:n} and construction as a pullback along $[n']$, the sheaf $\F$ is trivialised by pullback along the isogeny
$$0 \to \Gamma[p^{\nu}] \to J^{\an} \xrightarrow{[p^{\nu}]} J^{\an} \to 0.$$

We infer from Lemma \ref{lemma:p-torsion} that $\Gamma(J^{\an},\F)$ has non-trivial $p$-torsion. It follows from \eqref{eqn:kucm} that $\KU^{*}(\Mm^{\an},\pi^*\tilde{\alpha}[p])^{\Gamma'}\subset \KU^*(\Mm^{\an},\pi^*\tilde{\alpha})$ has non-vanishing $p$-torsion, which contradicts Corollary \ref{cor:no-p-torsion}.
\end{proof}

\appendix

\section{G-sheaves of spectra and equivariant vanishing cycles}

\subsection{Motivation: $G$-equivariant homotopy theory}

In this section we denote by $G$ a finite group and will develop a relative or sheaf-theoretic analogue of equivariant homotopy theory. The latter is an ambiguous term since stable and unstable homotopy theory can be extended to several distinct equivariant theories. In the stable case one counts three main contestants laying claim to the term \emph{$G$-spectra} or \emph{equivariant spectra}:

\begin{itemize}
\item[-] \emph{Borel} $G$-spectra
\item[-] \emph{naive} $G$-spectra
\item[-] \emph{honest} $G$-spectra
\end{itemize}

By definition, a Borel $G$-spectra is a spectrum endowed with a $G$-action up to homotopy. In categorical terms, one considers the $\infty$-category of functors $\Fun(BG^{\op},\Spectra)$. This approach to equivariant homotopy theory is too naive to capture most equivariant cohomology theories, such as equivariant $K$-theory. In a nutshell, the issue is that non-equivariant homotopy equivalences shape the purportedly equivariant category $\Fun(BG,\Spectra)$ into a sort of hybrid existence strung between the equivariant and non-equivariant world. 

To develop a more accurate equivariant picture it is best to rebuild equivariant homotopy theory from scratch by introducing a (model) category of topological spaces endowed with strict $G$-actions and $G$-equivariant weak equivalences. Localising along the latter produces an $\infty$-category $\Spaces_G$ of  $G$-spaces, which can be further stabilised to yield the stable $\infty$-category of \emph{naive} $G$-spectra $\Sp_G$. This framework handles equivariant $K$-theory sufficiently well, although it misses additional structure presenting itself in the guise of an additional grading of $K$-groups; hence the need for a further theory of \emph{honest} $G$-spectra $\Sp_G^{\text{honest}}$.

The stable $\infty$-category $\Sp_G^{\text{honest}}$ is constructed from $\Spaces_G$ by inverting suspension with respect to arbitrary representation spheres. That is, for every linear $G$-representation $V$ we have a $G$-space given by the one-point compactification $\mathbb{S}^V=V \cup \{\infty\}$. The corresponding reduced suspension functors $\{\Sigma^V\}_V$ are then inverted for all isomorphism classes of $G$-representations of finite rank. 

The first goal of this appendix is to develop the rudiments of a theory of $G$-sheaves (of spectra) on a topological space $X$ endowed with a $G$-action. The resulting stable $\infty$-category $\Sh_{\Sp}(X;G)$ is supposed to be equivalent to the $\infty$-category of naive $G$-spectra $\Sp_G$ when $X$ is a point. Furthermore, for a naive $G$-spectrum $\Eb$ and a sufficiently nice space $X$, we want the space of global sections $\Gamma(X,\underline{\Eb})$ of the induced $G$-sheaf $\Eb$ on $X$ to be naturally equivalent to equivariant cohomology $\Hb_G^*(X,\Eb)$.

Subsequently, we will use $G$-sheaves and their pushforward and pullback operations to define an equivariant vanishing cycle formalism which can be applied to twisted equivariant $K$-theory.

\subsection{Recollection of $G$-equivariant homotopy theory}

We denote by $G$ a finite group. A $G$-set is a set endowed with a right $G$-action and similarly for $G$-actions on topological spaces. As mentioned above, an honest approach to equivariant homotopy theory is built on a category of topological spaces endowed with strict $G$-actions, localised with respect to \emph{equivariant} homotopy equivalences. However, there is an alternative approach via orbit categories, which is more amenable to generalisation. By virtue of Elemendorf's theorem, this yields an equivalent $\infty$-category of $G$-spaces. 

While we limit ourselves to describing the necessary theory in broad strokes we refer the reader in search of a more detailed account or references to \cite[Section 3.2]{temp} and \cite{schwede}.

\begin{definition}\label{defi:orbit}
We denote by $\Orb_G$ the category of $G$-orbits. That is, objects of this category are right $G$-sets isomorphic to $H \backslash G$ where $H$ is a subgroup of $G$. Morphisms are given by $G$-equivariant maps.
\end{definition}

\begin{definition}\label{defi:G-spaces}
The $\infty$-category of presheaves $\Pr(\Orb_G)=\Fun(\Orb_G^{\op},\Spaces)$ is called the $\infty$-category of $G$-spaces $\Spaces_G$. 
\end{definition}

There is a natural functor from $BG$ to $\Orb_G$, sending the sole object $\bullet \in BG$ to $\{e\} \setminus G \in \Orb_G$. Restriction along this inclusion induces a functor $\Pr(\Orb_G,\Spaces) \to \Pr(BG)$, that is, from $G$-spaces to Borel $G$-spaces. This functor is far from being an equivalence.
This concludes the expository part of this appendix and will now turn to defining an $\infty$-category of $G$-sheaves of spaces (and later of spectra) on an equivariant space $X$.

\subsection{$G$-sheaves of spaces and spectra}

\subsubsection{Definitions}

As before we let $G$ be a finite group. Let $X$ be a topological $G$-space, i.e. a topological space endowed with a continuous $G$-action. The purpose of this section is to define an $\infty$-category of space-valued $G$-sheaves $\Sheaves(X;G)$ satisfying the following properties:
\begin{itemize}
\item For $X= \{\bullet\}$ there is an equivalence $\Sheaves\big(\{\bullet\};G\big) \cong \Spaces_G$.
\item For every open subset subset $U \subset X^G$, there is a \emph{global sections functor}
$$\Gamma_G(U,-)\colon \Sheaves(X;G) \to \Spaces_G.$$
\item If $f\colon X \to Y$ is $G$-equivariant, then there are adjoint functors
$$f^{-1}_G\colon \Sheaves(Y;G) \rightleftarrows \Sheaves(X;G) \cocolon f^G_*.$$
Furthermore, there is a fully subcategory $G$-hypersheaves $\Sh(X;G) \subset \Sheaves(X;G)$ for which we also have pullback and pushforward functors.
\end{itemize} 

Mimicking the definition of $\Orb_G$, we introduce a site of \emph{generalised $G$-orbits} $\Orb(X;G)$.

\begin{definition}\label{defi:gen-orbit}
Let $H \subset G$ be a subgroup and $U \subset X^{H}$ an open subset. We let 
$i\colon U \times H \backslash G = P \to X$
be the map $(x,H g) \mapsto x\cdot{} g$. We call $P$ a \emph{generalised $G$-orbit}.
\end{definition}

For example, one could consider the generalised $G$-orbit $H \setminus G \times X^H$ for any subgroup $H \subset G$. 
Next, we must define morphisms of generalised $G$-orbits.

\begin{definition}
\begin{itemize}
\item[(a)] For $H \subset G$ a subgroup, we denote by $X^{s(H)}$ the union $\bigcup_{g \in G}X^{gHg^{-1}}$.
\item[(b)] If $H_1, H_2 \subset G$ is a pair of subgroups and $P= U \times H_1 \backslash G$ a generalised $G$-orbit, then we denote by $P|_{s(H_2)}$ the fibre product $P \times_X X^{s(H_2)}$.
\item[(c)] A morphism of generalised $G$-orbits $P_i = U_i \times H_i \backslash G$ is given by a commutative diagram
\[
\xymatrix{
P_1 |_{s(H_2)}  \ar@/^/[rr]^{i_1} \ar[r]_{{f}} & P_2 \ar[r]_{i_2} & X 
}
\]
where $f$ is a $G$-equivariant continuous map and $H_1$ is included in $H_2$ up to conjugation.
\item[(d)] We denote the resulting category of generalised $G$-orbits in $X$ by $\Cov(X;G)$.
\end{itemize}
\end{definition}

Note that if $H_1$ and $H_2$ are conjugate subgroups, then $X^{s(H_2)} = X^{s(H_1)}$, and thus $i_1(P_1) \subset X^{s(H_1)} = X^{s(H_2)}$. This means that in this case we have $P_1|_{s(H_2)} = P_1$, and thus $f$ will simply be a $G$-equivariant morphism $P_1 \to P_2$ satisfying $i_2 \circ f = i_1$.

\begin{example}
Let $X$ be a $G$-space with $X^G = \{x\}$. We denote by $P_1$ the generalised $G$-orbit $X \times G$, and by $P_2$ the generalised $G$-orbit $\{x\} \times G \backslash G = \{x\}$. Then, there is a morphism of generalised $G$-orbits $P_1 \to P_2$ given by the constant map 
$P_1 |_{s(G)} = \{x\} \times G \to \{x\}.$
\end{example}

This category $\Cov(X;G)$ is endowed with a natural Grothendieck topology, induced by the notion of open covers in topology.

\begin{definition}
We define a Grothendieck topology, denoted $\std$ and called \emph{standard topology}, on $\Cov(X;G)$ as follows: a collection of maps $\{P_i \xrightarrow{f_i} P\}_{i \in I}$ in $\Cov(X;G)$ is a covering, if for every $i \in I$ we have that $f_i$ is an open immersion and that $\bigcup_{i \in I} \image f_i = P$. The empty covering $\{\}$ is defined to be a covering of the empty generalised $G$-orbit $\emptyset \to X$.
\end{definition}

We leave the straightforward verification that this defines a Grothendieck topology to the reader. This topology will also be referred to as the \emph{standard topology} on $\Cov(X;G)$. We emphasise that if $\{P_i \to P\}$ is a non-empty covering of a non-empty $G$-orbit, then we obtain isomorphism of $G$-sets $H_i \backslash G \simeq H \backslash G$. Thus, each $H_i$ is conjugate to $H$ in $G$.

\begin{rmk}\label{rmk:singleton}
If $X$ is a singleton $\{\bullet\}$ acted on trivially by $G$, then there is a fully faithful functor 
$$\Orb_G \hookrightarrow \Cov\big(\{\bullet\};G\big).$$
It is not an equivalence, since $\emptyset \in \Cov(\{\bullet\};G)$ is not in its essential image. However, we have equivalences 
$$\Sh_{\std}\big(\Cov(\{\bullet\};G)\big) \cong \Sheaves_{\std}\big(\Cov(\{\bullet\};G)\big) \cong \Pr(\Orb_G),$$
since a (hypersheaf) $\F$ satisfies $\F(\emptyset) \simeq \emptyset$ (as can be seen by applying the sheaf condition to the empty covering of $\emptyset$).
\end{rmk}

The $\infty$-category of sheaves on this site is the sought-for theory of $G$-sheaves mimicking the construction of $G$-spaces in equivariant homotopy theory (see Remark \ref{rmk:singleton2} below).

\begin{definition}
We define the $\infty$-category of $G$-sheaves on $X$, denoted by $\Sheaves(X;G)$, to be $$\Sheaves_{\std}(\Cov(X;G)).$$ The full subcategory of $G$-hypersheaves is denoted by $\Sh(X;G)$.
\end{definition}

\begin{rmk}\label{rmk:singleton2}
It follows from Remark \ref{rmk:singleton} and Definition \ref{defi:G-spaces} that $\Sh\big(\{\bullet\};G\big) = \Sheaves\big(\{\bullet\};G\big)$ is equivalent to $\Spaces_G$. Similarly, if $G=\{e\}$ is the trivial group, then $\Cov(X;G)$ is simply the site of open subsets of $X$ and thus $\Sheaves(X;\{e\}) = \Sheaves(X)$ and $\Sh(X;\{e\})=\Sh(X)$.
\end{rmk}

\begin{lemma}\label{lemma:subgroup-compatibility}
$\Sheaves(H \setminus G ; G) \cong \Spaces_H$
\end{lemma}
\begin{proof}
By definition, the objects in $\Cov(H \setminus G ; G)$ are given by $G$-equivariant maps $U \times H'\setminus G \to H\setminus G$ where $H' \subset G$ is a subgroup and $U \subset (H\setminus G)^{H'}$ is a subset. Expressing $U$ as a disjoint union of singletons, we see that there is a disjoint cover by objects isomorphic to 
$$P_{H'}=\{\ast\} \times H'\setminus G \xrightarrow{\act} H \setminus G,$$
where $\ast$ is a fixed element in $H \setminus G$ and $\act$ denotes the action map. Furthermore, every $\std$-cover can be refined to disjoint cover of objects isomorphic to $P_{H'}$ with $H' \subset H$. Denoting the full subcategory with objects $\{P_{H'}\}_{H' \subset H}$ by $\Cov'$ we therefore have equivalences
$$\Sheaves_{\std}(\Cov(H \setminus G; G)) \cong \Sheaves_{\std}(\Cov') = \Pr(\Cov').$$
For the last identification we used that the sheaf condition is vacuous on $\Cov'$ since every $\std$-cover of $P_{H'}$ can be refined to the trivial cover $\{P_{H'} \to P_{H'}\}$. To conclude the proof it suffices to produce an equivalence between $\Cov'$ and $\Orb_H$. This is achieved by sending $H' \setminus H$ to $P_{H'}$ and by observing that a commutative diagram of $G$-equivariant maps
\[
\xymatrix{
\{\ast\} \times H' \setminus G \ar[r]^{g} \ar[rd] & \{\ast\}\times H''\setminus G \ar[d] \\
& H \setminus G
}
\]
implies that $g|_{H'\setminus H}$ yields an $H$-equivariant map $H'\setminus H \to H'' \setminus H$ and vice versa. Therefore, we have $\Hom_{\Cov(H \setminus G; G)}(P_{H'},P_{H''}) \simeq \Hom_{\Orb_H}(H'\setminus H, H'' \setminus H)$, which shows that $\Cov'$ is equivalent to $\Orb_H$.
\end{proof}

\subsubsection{Equivariant pullback and pushforward}

\begin{construction}\label{const:pushforward}
Let $f\colon X \to Y$ be a $G$-equivariant map of topological $G$-spaces. There is a natural functor
$$f^*\colon \Cov(Y;G) \to \Cov(X;G),$$
which respects coverings.
\end{construction}
\begin{proof}[Details]
Recall that an object in $\Cov(Y;G)$ is given by a $G$-equivariant map
$$i \colon U \times H \backslash G \to Y$$
where $U \subset Y^{H}$ is an open subset. The functor $f^*$ sends this object to the generalised $G$-orbit $(f^H)^{-1}(U) \times H \backslash G$. We leave the remaining details (definition of $f^*$ on morphisms) to the reader.
\end{proof}

This continuous map of sites allows us to define the requisite pair of adjoint functors as pullback and pushforward.

\begin{definition}\label{defi:equi-push}
We define the $G$-pushforward functor
$$f^G_*\colon \Sheaves(X;G) \to \Sheaves(Y;G)$$
as restriction of sheaves along the covering-preserving functor $f^*\colon \Cov(Y;G) \to \Cov(X;G)$. We have $f^G_*(\Sh(X;G)) \subset \Sh(Y;G)$.
\end{definition}

The assertion that pushforward preserves hypersheaves follows from \cite[Proposition 6.5.2.13]{htt}. The functor $f^G_*$ preserves limits of presheaves, sheaves and hypersheaves and therefore has a left adjoint 
$$f_G^{-1}\colon  \Sheaves(Y;G) \to  \Sheaves(X;G)$$ 
by virtue of the adjoint functor theorem \cite{htt}. We emphasise that $f_G^{-1}(\Sh(Y;G))$ is not always contained in $\Sh(X;G)$. It is therefore necessary to post-compose with a hypersheafification functor to obtain a left adjoint
$$\widehat{f}_G^{-1}\colon \Sh(Y;G) \to \Sh(X;G)$$
to the restricted functor $f^G_*|_{\Sh(X;G)}$.

\begin{definition}\label{defi:enriched-sections}
Let $U \subset X$ be a subset satisfying $U \cdot{} G = U$. Then, there is a functor
$$\Gamma(U,-)\colon \Sheaves(X;G) \to \Spaces_G,$$
which is defined to be the composition
$$\Sheaves(X;G) \xrightarrow{i^{-1}_G} \Sheaves(U;G) \xrightarrow{p^G_*} \Sheaves\big(\{\bullet\};G\big) \cong \Spaces_G.$$
Here, we denote by $i\colon U \subset X$ the inclusion and $p\colon U \to \{\bullet\}$ the unique map to a singleton. The last equivalence stems from Remark \ref{rmk:singleton2}.
\end{definition}

In a similar vain, we have the following equivalence.

\begin{lemma}\label{lemma:trivial-action-case}
Assume that $G$ acts trivially on $X$, then $\Sheaves(X;G)$ is equivalent to the $\infty$-category of $G$-space valued sheaves $\Sheaves_{\Spaces_G}(X)$ and likewise for hypersheaves. Furthermore, for a continuous map of spaces $f\colon X \to Y$ acted on trivially by $G$ we have a commutative diagram
\[
\xymatrix{
\Sheaves(X;G) \ar[r]^{\cong} \ar[d]_{f^G_*} & \Sheaves_{\Spaces_G}(X) \ar[d]^{f_*} \\
\Sheaves(Y;G) \ar[r]^{\cong} & \Sheaves_{\Spaces_G}(Y)
}
\]
of $\infty$-categories and likewise for hypersheaves.
\end{lemma}
\begin{proof}
It suffices to construct an equivalence $\Pr(X;G) \cong \Pr_{\Spaces_G}(X)$ and to observe that sheaf and hypersheaf conditions are preserved by it. To simplify this process we evoke stricitfication of presheaves and work with classical models for these presheaf categories, i.e. with strict functor categories
$$\Fun_{str}(\Cov(X;G)^{\op},\Spaces^{str}) \text{ and }\Fun_{str}(\Open(X)^{\op},\Spaces_G^{str}),$$
where we denote by $\Spaces_G^{str}$ the (model) category of $G$-spaces given by $\Fun_{str}(\Orb_G^{\op},\Spaces^{str})$. Since an object in $\Cov(X;G)$ is of the sheaf $U \times H \setminus G$ where $U \in \Open(X)$ and $H\setminus G \in \Orb_G$, it is straightforward to construct an equivalence of classical categories $\Fun_{str}(\Cov(X;G)^{\op},\Spaces^{str}) \cong \Fun_{str}(\Open(X)^{\op},\Spaces_G^{str}).$ Localisation at weak equivalences then yields an equivalence between the resulting $\infty$-categories of presheaves:
$$\Pr(X;G)\cong \Pr_{\Spaces_G}(X).$$
It is clear that, again by the nature of objects in $\Cov(X;G)$, sheaves and hypersheaves are preserved by this equivalence.

The same strategy can be used to construct a square of classical categories and functors, which commutes up to a natural transformation:
\[
\xymatrix{
\Fun_{str}(\Cov(X;G)^{\op},\Spaces^{str}) \ar[r]^{\cong} \ar[d]_{f^G_*} &  \Fun_{str}(\Open(X)^{\op},\Spaces_G^{str}) \ar[d]^{f_*} \\
\Fun_{str}(\Cov(Y;G)^{\op},\Spaces^{str}) \ar[r]^{\cong} &  \Fun_{str}(\Open(Y)^{\op},\Spaces_G^{str})
}
\]
Localising at weak equivalences we obtain the requisite commutative squares of $\infty$-categories of presheaves, sheaves and hypersheaves.
\end{proof}

\subsubsection{The forgetful functor to sheaves}

As one would expect, there is a forgetful functor 
$$\Phi\colon\Sheaves(X;G) \to \Sheaves(X)$$
from $G$-sheaves on $X$ to sheaves on $X$. This functor preserves hypersheaves.

A direct construction proceeds as follows: an open subset $U \subset X$ induces as above a generalised $G$-orbit $U \times G \to X$ (the action map). A presheaf 
$$\F \colon \Cov(X;G)^{\op} \to \Spaces$$
can be composed with this functor $\Open(X)^{\op} \to \Cov(X;G)^{\op}$.

As can be confirmed by tracing through the definitions, this functor agrees with $\act^{-1}_G$, where $\act\colon X \times G \to X$ is the action on $X$, and the left-hand side is endowed with the trivial $G$-action on $X$ and right translation on the $G$-component.

More generally, for a subgroup $H \subset G$ we have a forgetful functor $\Phi^G_H\colon\Sheaves(X;G) \to \Sheaves(X^H)$. Its construction is based on a map of sites
$$F_{H}\colon\Open(X^H) \to \Cov(X;G)$$
which sends an open subset $U \subset X^{H}$ to the generalised orbit $U \times H \backslash G \to X^{H} \subset X$.

\begin{definition}\label{defi:PhiGH}
We let $\Phi^G_{H}$ denote the functor $\Sheaves(X;G) \to \Sheaves(X^{H})$ induced by composition with $F_{H}$. 
\end{definition}

\subsubsection{Stalks}
Let $X$ and $H \subset G$ be as above and assume that $x \in X^{H}$ is an $H$-fixpoint. We then obtain a $G$-equivariant map
$$i_{x;H}\colon H \setminus G \to X$$
sending $[e] \in H \setminus G$ to $x$. This yields a functor 
$$i^{-1}_{x;H}\colon \Sheaves(X;G) \to \Sheaves(H \setminus G ;G) \cong \Spaces_{H},$$
where we used the equivalence of Lemma \ref{lemma:subgroup-compatibility}. This functor will be called the \emph{$H$-stalk at $x$}. We will also use the notation $\F_{x,H}$ to denote the $H$-space $i^{-1}_{x;H}\F \in \Spaces_H = \Pr(\Orb_H)$.

\begin{definition}
Let $x$ be an $H$-fixpoint and $P \in \Orb_H$. For $\F \in \Sheaves(X;G)$ we denote by $\F_{x;H,P}$ the space $\F_{x;H}(P) \in \Spaces$.
\end{definition}

\begin{lemma}\label{lemma:enough-points}
\begin{itemize}
\item[(a)] The classical topos $\Sh_{\Sets}(X;G)$ has enough points. 
\item[(b)] If $X$ is sober, then every point of $\Sheaves_{\Sets}(X;G)$ is induced by a triple $(x,H,P)$ where $H \subset G$ is a subgroup, $x \in X^{H}$ is an $H$-fixpoint, and $P \in \Orb_H$ is an $H$-orbit. 
\end{itemize}
\end{lemma}
\begin{proof}
For the proof of (a) we observe first that for a fixed subgroup $H$ we have a functor $$F_{H}\colon \Open(X^{H}) \to \Cov(X;G).$$ Furthermore, the essential images of these functors ranging over all possible subgroups cover the site $\Cov(X;G)$. Let $f\colon \F \to \Gc$ be a map of set-valued sheaves on $\Cov(X;G)$ such that for every triple $(x;H,P)$ we have that the induced map of stalks $\F_{x;H,P} \to \Gc_{x;H,P}$ is an isomorphism. As in the non-equivariant case, this directly implies that for every $\tilde{P} \in \Cov(X;G)$ the induced map of sets $\F(\tilde{P}) \to \Gc(\tilde{P})$ is a bijection.

To prove claim (b) we use that the closed subset $X^H \subset X$ is also sober, and hence one can argue as in \cite[Proposition IX.3.2]{MMcL} to verify the second assertion. 
\end{proof}

\begin{proposition}
We assume that $X$ is a sober space. Then, a morphism $f\colon \F \to \Gc$ in $\Sh(X;G)$ is an equivalence if and only if for every subgroup $H$ and for every $H$-fixpoint $x \in X^H$ we have that $\F_{x;H} \to \Gc_{x;H}$ is an equivalence of $H$-spaces.
\end{proposition}
\begin{proof}
The implication ``$\Rightarrow$" is a direct consequence of functoriality of $i^{-1}_{x;H}$. It remains to establish the other direction ``$\Leftarrow$". Thus, we assume that $\F$ and $\Gc$ are hypersheaves and that $f\in \Hom(\F,\Gc)$ is a morphism satisfying the assumption that $f_x$ is an equivalence of $H$-spaces for all points $x \in H$ and all subgroups $H \subset G$. Since the topos $\Sh_{\Sets}(X;G)$ was shown to have enough points (see Lemma \ref{lemma:enough-points}) it follows from work of Jardine that $f$ must be an equivalence of hypersheaves (see \cite[Remark 6.5.2]{htt}).
\end{proof}

\subsubsection{Hypercompleteness}\label{sssec:hyper}

In light of the previous paragraph it is useful to have a criterion to recognise when sheaves and hypersheaves agree (= hypercompleteness of $\Sheaves(X;G)$). This happens under similar assumptions as for sheaves on topological spaces (without group action).

\begin{definition}
We say that a topological space $Z$ has \emph{covering dimension $\leq n$} if for every open covering $(U_i)_{i \in I}$ there exists an open refinement $(V_j)_{j \in J}$ such that intersection $V_{j_0} \cap \cdots \cap V_{j_{n+1}} = \emptyset$ if the indices $j_0,\dots,j_{n+1}$ are pairwise distinct.
\end{definition}

We recall that a topological space $X$ is called \emph{heriditarily paracompact} if every subset $M \subset X$ is paracompact. It is well-known that this property holds if and only if every open subset $U \subset X$ is paracompact. Therefore, separable metric spaces have this property, and thus also complex-analytic spaces.

\begin{proposition}
Assume that $X$ is a \emph{heriditarily paracompact} topological space acted on by $G$ such that for every subgroup $H \subset G$ we have that $X^H$ has covering dimension $\leq n$ for some positive integer $n$. Then, the $\infty$-topos $\Sheaves(X;G)$ is \emph{hypercomplete}, i.e. $\Sh(X;G)=\Sheaves(X;G)$.
\end{proposition}
\begin{proof}
We apply Jardine's criterion for hypercompleteness: an $\infty$-topos $\Xc$ which is locally of homotopy dimension $\leq n$ is hypercomplete. Recall that for $\Xc$ to be locally of homotopy dimension $\leq n$ we must construct a family of objects $(T_i)_{i \in I}$ generating $\Xc$ under colimits such that for each slice $\Xc_i=\Xc_{/U_i}$ we have a section $1_{\Xc_i} \to F$ for every $n$-connective object $F \in \Xc_i$.

For $\Xc = \Sheaves(\Xc;G)$ we let $\{T_i\}_{i \in I}$ be a family of representatives of isomorphism classes of generalised $G$-orbits. That is, $T_i$ can be written as the sheaf represented by a $G$-equivariant map
$$P_i = U_i \times H_i \setminus G \to X,$$
where $U_i \subset X^{H_i}$ is an open subset. More concretely, $T_i$ assigns to another generalised $G$-orbit $(P,f)$ the set of maps 
$$\Hom_{\Cov(X;G)}(P,P_i).$$

By definition, objects of the slice topos $\Sheaves(\Xc;G)_{/T_i}$ are sheaves $\F$ equipped with a morphism $\F \to T_i$. Therefore, $\F(P) \neq \emptyset$ if and only if there exists a map $P \to P_i$. This property could be described as $\F$ \emph{being supported on $P_i$}. For the $\infty$-categories of sheaves we have:
$$\Sheaves(\Xc;G)_{/T_i} \cong \Sheaves\big(\Cov(X;G)\big)_{/ T_i} \cong \Sheaves\big(\Cov(X;G)_{/T_i}\big) \cong \Sheaves\big(\Open(U_i) \times \Latt(H)\big),$$
where $\Latt$ denotes the poset of subgroups of $H$. Since the latter is a finite poset, we are now essentially reduced to showing that $\Sheaves(U_i)$ has homotopy dimension $\leq n$. Since $U_i$ is a paracompact space, this is equivalent to $U_i$ having covering dimension $\leq n$ (see \cite[Corollary 7.2.3.7]{htt}). This follows right from the assumption that $U_i$ is an open subset of $X^H$ which is assumed to have covering dimension $\leq n$.
\end{proof}

\subsubsection{Compatibility with equivariant cohomology}

The purpose of this subsection is to exhibit for a $G$-space $\E \in \Spaces_G$ a natural isomorphism
\begin{equation}\label{eqn:eqi}\underline{\Map}(|\Sing X|,\E) \simeq p_*^G\widehat{p}^{-1}_G\E,\end{equation}
where $\Map$ denotes the internal mapping space in $\Spaces_G$, $\Sing$ the functor from topological $G$-spaces to $G$-simplicial sets (given by singular chains in $X$), which are then mapped to $\Spaces_G$ via geometric realisation $|-|$ and $p\colon X \to \{\bullet\}$ denotes the constant $G$-equivariant map.

There are some technical assumptions on $X$ required, which we will explain below (see Proposition \ref{prop:equi-compatibility}). 

As a first step towards this result, we examine the following special case.

\begin{lemma}\label{lemma:HG}
Let $X$ be $H\setminus G$ endowed with the canonical right $G$-action. Then, 
$$p_*^G\widehat{p}^{-1}_G\E \simeq  \underline{\Map}(H \setminus G,\E),$$
and thus we have an equivalence of spaces as asserted in equation \eqref{eqn:eqi} if $X = H \setminus G$.
\end{lemma}
\begin{proof}
This identity clearly holds for classical set-valued sheaves, as shown by a direct computation. On the other hand, every space-valued sheaf on a classical site is a colimit of representable (and therefore, classical) sheaves by Corollary 5.1.5.8 in \cite{htt}.  Hence, it is sufficient to check that the functors $p_*^Gp^{-1}_G$ and $\underline{\Maps}(H\backslash G,-)$ commute with arbitrary colimits. Being a left adjoint, the functor $p^{-1}_G$ commutes with arbitrary colimits. Finiteness of the set $H\backslash G$ implies that every object of the site $\Orb(H\backslash G; G)$ is quasi-compact and therefore that $p_*^G$ commutes with arbitrary colimits. 
\begin{claim}
The functor $\underline{\Maps}(H\backslash G,-)$ commutes with arbitrary colimits.
\end{claim}  
There is a canonical map
$$\colim_I \underline{\Maps}(H\backslash G,\mathcal{F}_i) \to \underline{\Maps}(H\backslash G,\colim_{I}\mathcal{F}_i).$$
Let $H'\backslash G \in \Orb_G$ be a test object. We obtain a map
$$\colim_I \underline{\Maps}(H\backslash G \times H'\backslash G,\mathcal{F}_i) \to \underline{\Maps}(H\backslash G \times H'\backslash G,\colim_{I}\mathcal{F}_i),$$
where the product $H\backslash G \times H'\backslash G$ is taken in the presheaf category $\Pr(\Orb_G) = \Spaces_G$. Since this product is itself a finite colimit of representable presheaves, we obtain that $\Maps(H\backslash G \times H'\backslash G,-)$ commutes with colimits. This concludes the proof.
\end{proof}

The proof of the following lemma relies on the Equivariant Proper Base Change Theorem, which we will prove in Subsection \ref{ssec:proper-base} in the slightly more general context of topological stacks.

\begin{lemma}[G-homotopy invariance]
Let $\Ee \in \Spaces_G$ be a $G$-space and $X$ a locally compact topological space endowed with a $G$-action. Assume that $\Sh(X;G) = \Sheaves(X;G)$. Then, we have that the natural pullback morphism induces an equivalence
$$(p_{X \times [0,1]})_*^G \underline{\Ee}_{X \times [0,1]} \simeq (p_X)_*^G \underline{\Ee}_X.$$
\end{lemma}

Consequently, one obtains for a $G$-homotopy $H\colon X \times [0,1] \to Y$ of locally compact topological spaces that $H_0$ and $H_1$ induce the same map in $\Ho(\Spaces_G)$ between $(p_Y)_*^G \Ee$ and $(p_X)_*^G\Ee$. Indeed, this follows from the result above: we have $H_t = H \circ \mathrm{ev}_t$, where $\mathrm{ev}_t$ is the section $X \to X \times [0,1]$ corresponding to $t \in [0,1]$. 

Therefore, if $X \simeq Y$ is a $G$-homotopy equivalence, and both $X$ and $Y$ are locally compact, then $(p_{X})_*^G \underline{\Ee}_{X} \simeq (p_Y)_*^G \underline{\Ee}_Y.$ 
\begin{proof}[Proof of the lemma]
Since every sheaf on $X$ satisfies hyperdescent, it suffices to check for every pair $(x,H)$ with $x \in X^H$ that 
$$\big((p_{X})_*^G\underline{\Ee}\big)_{x;H} \to \big((p_{X \times [0,1]})_*^G\underline{\Ee}\big)_{x;H}$$
is an equivalence of spaces. 

Let us denote by $h\colon X \times [0,1] \to X$ the $G$-equivariant map of spaces given by projection to the first component.

The Proper Base Change Theorem applied to $h_*^G$ yields
$$g^{-1}_G h_*^G \simeq h_*^{'G} g'^{-1}_G$$ 
where we denote by $g$ a $G$-map along which we base change
(see Theorem \ref{thm:proper-base}).
In particular, we have
$$(h^G_*\underline{\E}_{X \times [0,1]})_x \simeq (p_{[0,1]})_*^G \underline{\E}_{[0,1]}.$$
Therefore, it remains to show that $\E \simeq (p_{[0,1]})_*^G \underline{\E}_{[0,1]}$. This is indeed the case, since $G$ acts trivially on $[0,1]$ and $(p_{[0,1]})_*^G \underline{\E}_{[0,1]}$ is therefore equivalent to non-equivariant sheaf cohomology of the $\Spaces_G$-valued sheaf $\underline{\E}$ on $[0,1]$ (see Lemma \ref{lemma:trivial-action-case}):
$$(p_{[0,1]})_*^G \underline{\E}_{[0,1]} \simeq (p_{[0,1]})_* \underline{\E}_{[0,1]} \simeq \E$$
where we used non-equivariant homotopy equivalence for (non-equivariant) sheaf cohomology in the last step.
\end{proof}

For a topological $G$-space $X$ we denote by $q_X\colon X \to X/G$ the projection map to the quotient space.

\begin{situation}\label{sit:loc-contract}
Assume that $X$ is locally compact and that there exists an open covering $\{U_i\}_{i \in I}$ of the quotient $X/G$ such that for every $i \in I$ we have a $G$-equivariant deformation retraction $$q_X^{-1}(U_i) \cong H_i \setminus G.$$
\end{situation}

\begin{proposition}\label{prop:equi-compatibility}
Assume that we are in Situation \ref{sit:loc-contract}. Then, there is an equivalence of $G$-spaces as asserted in equation \eqref{eqn:eqi}, i.e.
$$\underline{\Map}(|\Sing X|,\E) \simeq p_*^G\widehat{p}^{-1}_G\E,$$
\end{proposition}
\begin{proof}
We will prove more generally that there is an equivalence of sheaves of $G$-spaces on $X/G$:
$$\underline{\Map}(|\Sing (X / (X/G))|,\E) \simeq (q_X)_*^G \widehat{p}^{-1}_G\E,$$
where $|\Sing(X/(X/G))|$ denotes the (strict) presheaf (taking values in the opposite category of topological $G$-spaces) associating to an open subset $U \subset X/G$ the $G$-space $|\Sing q_X^{-1}(U)|$. The asserted equivalence will then be obtained from this one by applying the global sections functor.

Since both sides are hypersheaves it suffices to 
\begin{itemize}
\item[-] construct a morphism, 
\item[-] and to show that it is an equivalence stalk-wise.
\end{itemize}
To construct the morphism we appeal to the adjunction between geometric realisation and $\Sing$, which yields a natural map (the co-unit) $|\Sing U| \to U$ for every $G$-space $U$.  

For a $G$-space $U$ we denote by $\pi_0^{G}(U)$ the $G$-set corresponding to the set-valued presheaf on $\Orb_G$ assigning to $H\backslash G$ the set of homotopy classes of $G$-equivariant maps $H\backslash G \to U$. There is a natural map 
$$|\Sing U| \to \pi^G_0(U),$$
which in fact also stems from an adjunction (the right adjoint being the inclusion functor of $G$-sets in $G$-spaces).

This yields a map of presheaves
$$\underline{\Map}(\pi_0^G (X / (X/G)),\E) \to \underline{\Map}(|\Sing (X / (X/G))|,\E),$$
which by assumption on local equivariant contractibility, induces an equivalence on stalks and thus an equivalence of associated hypersheaves.

The right-hand presheaf already being a hypersheaf, it remains to check that the hypersheafification of the left-hand presheaf is equivalent to $p^G_*\widehat{p}_G^{-1}\Ee$.
This follows directly from the claim that the presheaf $p^G_* p_G^{-1,{\rm pre}}\Ee$ is equivalent to $\underline{\Map}(\pi_0^G (X / (X/G)),\E)$ by virtue of definition. To see that this claim holds one uses the assumption to reduce to the case where $X = H \backslash G$ and uses the equivalence
$$\Pr(H\backslash G;G) \cong \Spaces_G/{H\backslash G},$$
describing the first $\infty$-category as a comma-category. With respect to this equivalence, the functor $p_G^{-1}$ corresponds to $- \times (H\backslash G)$, and $p^G_*$ is given by the relative space of sections $\underline{\Maps}_{/(H \backslash G)}(H\backslash G, -)$. The composition of these two functors is then given by the asserted internal mapping space.
\end{proof}

\subsubsection{$\Cc$-valued $G$-sheaves}

Let $\Cc$ be an $\infty$-category. We denote by $\Pr_{\Cc}(X;G)$ the $\infty$-category of functors
$\Fun\big(\Cov(X;G)^{\op},\Cc).$
\begin{definition}
A $\Cc$-valued $G$-presheaf $\F \in \Pr_{\Cc}(X;G)$ is a \emph{sheaf} (respectively hypersheaf) if for every object $X \in \Cc$ the space-valued $G$-presheaf $\Maps(X,\F) \in \Pr(X;G)$ is a sheaf (respectively hypersheaf). 
\end{definition}

For a continuous and $G$-equivariant map $f\colon X \to Y$ we obtain as before a pushforward functor $f^G_*\colon \Pr_{\Cc}(X;G)\to \Pr_{\Cc}(Y;G)$ preserving the full subcategories of sheaves and hypersheaves. If $\Cc$ is presentable, then we also dispose of an adjoint functor $f^{-1}_G$.

\begin{proposition}
Let $\Cc$ be a presentable $\infty$-category and $\Ee \in \Cc_G:=\Sh_{\Cc}(\{\bullet\};G)$. Under the same assumptions as in Proposition \ref{prop:equi-compatibility} we have 
$$p^G_*p^{-1}_G \Ee \simeq \underline{\Maps}(\Sing X,\Ee) \in \Cc_G.$$
\end{proposition}
\begin{proof}
This follows by combining Proposition \ref{prop:equi-compatibility} with Yoneda's lemma.
\end{proof}
Using the Yoneda lemma one could prove $\Cc$-valued counterparts of almost every result on $\Sheaves(X;G)$ established so far. We will refrain from including the details here.

The most important case for us will be when $\Cc = \Spectra$ equals the stable $\infty$-category of spectra. Objects of $\Sheaves_{\Spectra}(X;G)$ will also be referred to as \emph{spectra-valued $G$-sheaves}.

\subsection{$G$-sheaves on topological stacks}

In this subsection we develop a mixed framework, combining the notion of $G$-sheaves developed above with Borel's notion. This will allow us to consider a set-up where a topological space $X$ is endowed with two actions $G_1 \curvearrowright X \curvearrowleft G_2$, and study an $\infty$-category of sheaves endowed with a $G_1$-structure and a Borel-equivariant $G_2$ structure. It turns out that combining these two notions of equivariance is essential for our application to moduli spaces of Higgs bundles (see Subsection \ref{ssec:canonical}).

Our mixed framework will be built, by allowing $X$ in $Sh(X;G)$ to be a topological stack. Following Noohi (see \cite[Section 4]{noohi}), we let $\Top$ be the category of compactly generated Hausdorff spaces, endowed with the Grothendieck topology given by open covers. 

\begin{definition}
A continuous map $f\colon Y \to X$ of topological spaces is said to be \emph{locally cartesian}, if it is an open map and for every $x \in X$ there exists an open neighbourhood $U$ such that $f^{-1}(U)$ admits an open covering $\bigcup_{i \in I_x} V_i$, such that for all $i \in I_x$ there is a topological space $B_i$ and a homeomorphism $\alpha_i\colon V_i \to f(V_i) \times B_i$, satisfying
$f|_{U_i} = \mathrm{pr_1} \circ \alpha_i.$
\end{definition}

We fix the class of local fibrations $\mathbf{LF}$ (required for the definition of a topological stack) to be \emph{locally cartesian maps}. A \emph{topological stack} is then defined to be a stack $\Xc$ on $\Top$, such that there is an atlas 
$$p\colon X \to \Xc,$$
which is a locally cartesian and representable epimorphism.

We denote the $2$-category of topological stacks by $\TopSt$. For a finite group $G$, we denote by $BG$ the groupoid consisting of a single object $\ast$ with automorphism group $G$.

\begin{definition}
\begin{itemize}
\item[(a)] A $G$-action on a topological stack $\Xc$ is defined to be a functor $F\colon BG \to \TopSt$ satisfying $F(\ast) = \Xc$. We denote the ensuing $2$-category of such functors by $\GTopSt$.
\item[(b)] For a subgroup $H \subset G$, we consider $H \backslash G$ with the usual (right) $G$-action. For $\Xc \in \GTopSt$ we define the $H$-fixpoints, denoted by $\Xc^H$, to be the stack $\Xc \times_{X_{mod}} X_{mod}^H$, where $X_{mod}$ denotes the \emph{coarse moduli space} of $\Xc$ defined in \cite[Section 4.3]{noohi}.
\end{itemize}
\end{definition}

\begin{lemma}
The stack $\Xc^H$ of $G$-fixpoints is topological for every $\Xc \in \GTopSt$ and every subgroup $H \subset G$.
\end{lemma}
\begin{proof}
This follows directly from the definition $\Xc^H=\Xc \times_{X_{mod}} X_{mod}^H$ since fibre products of topological stacks are topological according to \cite[Proposition 13.9(iii)]{noohi}.
\end{proof}

Equipped with these definitions we can define the site $\Cov_{\mathbf{LF}}(\Xc;G)$.
\begin{definition}
Let $H \subset G$ be a subgroup and let $U \subset \Xc^{H}$ be an open subset. We let 
$i\colon U \times H \backslash G \to \Xc$
be the map $(x,H g) \mapsto x\cdot{} g$. We call the map $i$ a \emph{generalised $G$-orbit}.
\end{definition}

\begin{definition}
\begin{itemize}
\item[(a)] For a subgroup $H \subset G$ we let $X^{s(H)}$ be the union $\bigcup_{g \in G} X^{gHg^{-1}}$.
\item[(b)] If $f\colon P = U \times H_1 \backslash G \to X$ is a generalised $G$-orbit and $H_2 \subset G$ is a subgroup, then we write 
$P|_{s(H_2)}$ to denote the fibre product $P \times_X X^{s(H_2)}$.
\item[(c)] A morphism of generalised $G$-orbits is given by a commutative diagram
\[
\xymatrix{
P_1|_{s(H_2)} \ar@/^/[rr]^{i_1} \ar[r]_{{f}}  & P_2 \ar[r]_{i_2} & \Xc 
}
\]
where $f$ is a $G$-equivariant map.
\item[(d)] We let $\Cov(\Xc;G)$ be the category of generalised $G$-orbits in $\Xc$.
\end{itemize}
\end{definition}

\begin{definition}
We define a Grothendieck topology, denoted by $\mathbf{LF}$, on $\Cov(\Xc;G)$ as follows: a collection of maps $\{Z_i \xrightarrow{f_i} Z\}_{i \in I}$ in $\Cov(\Xc;G)$ is a covering, if for every $i \in I$ there exists a representation of the morphism $f_i$ as $(j_i,\tilde{f}_i)$, where 
$$j_i \colon U_i \hookrightarrow U$$
is representable and locally cartesian $\forall i \in I$ and $\bigcup_{i \in I} j_i(U_i) = U$.
\end{definition}

Let $\Xc=X$ be in $\Top$. By virtue of definition, for every cover $\{U_i \to X\}_{i \in I}$ in $\Cov_{\mathbf{LF}}(X;G)$ there is an open cover $\{V_j \hookrightarrow X\}_{j \in J}$ such that for every $i \in I$ there is an $j \in J$ and a factorisation
\[
\xymatrix{
& U_i \ar[d] \\
V_j \ar@{-->}[ru] \ar[r] & X\text{.}
}
\]

\begin{lemma}
For $X \in \Top$ endowed with a $G$-action, the restriction map along $\Cov(X;G) \to \Cov_{\mathbf{LF}}(X;G)$ gives rise to an equivalence of $\infty$-categories of sheaves
$$\Sheaves(\Cov_{\mathbf{LF}}(X;G)) \to \Sheaves(\Cov(X;G))$$
as well as hypersheaves.
\end{lemma}
\begin{proof}
As we recalled above every $\mathbf{LF}$-covering has a refinement by an $\std$-covering. This implies that these two topologies are equivalent, and thus give rise to the same sieves. The statement above then follows right from the definition of sheaves (see \cite[Definition 6.2.2.6]{htt}).
\end{proof}

\subsection{Proper base change}\label{ssec:proper-base}

The proper base change result presented below is likely to hold under much weaker assumptions. The assumption of hypercompleteness below (which is always satisfied in the cases of interest to us) was added to shorten the argument.

\begin{theorem}[Equivariant Proper Base Change]\label{thm:proper-base}
Assume that we are given a $G$-equivariant cartesian square of topological stacks endowed with a $G$-action. 
\[
\xymatrix{
Y' \ar[r]^{\tau} \ar[d]_{f} & X' \ar[d]^{g} \\
\Yc \ar[r]^{\pi} & \Xc
}
\]
Furthermore, let $\pi$ be a proper morphism and assume that all objects above are locally compact and Hausdorff, $X'$ and $Y'$ are topological spaces, and that $\Sheaves(X';G) = \Sh(X';G)$.

Then, for every $\F \in \Sheaves(\Yc;G)$ the natural transformation
\begin{equation}\label{nattra}g^{-1}_G\pi^G_*\F \simeq \tau_*^Gf^{-1}_G \F \end{equation}
is an equivalence.
\end{theorem}
\begin{proof}
By virtue of assumption, it suffices to verify on all stalks that \eqref{nattra} is an equivalence. Let $H \subset G$ be a subgroup and $x \in X'^H$ a fixpoint. We abbreviate the fibre $\tau^{-1}(x)$ by $Y'_x$. By appending the cartesian square
\[
\xymatrix{ 
Y'_x \times H \backslash G \ar[r] \ar[d] & \{x\}\times H \backslash G \ar[d] \\ 
Y' \ar[r]^{\tau} & X' \\
}
\]
we see that it suffices to prove the Proper Base Change Theorem for this square of topological $G$-spaces. Therefore, we may assume without loss of generality that $\Xc$ and $\Yc$ are topological spaces and that $X' = \{x\} \times H \backslash G$ where $x \in X^H$. Recall the functor between sites
$F_H\colon \Open(X^H) \to \Cov(X;G)$
constructed in Definition\ref{defi:PhiGH}. As explained there, we obtain an induced functor $\Phi^G_H \colon \Sheaves(X;G) \to \Sheaves(X^H)$. We have the stalk $\F_{x;H,H\backslash G} = \F_{x,H}(H \backslash G) \in \Spaces$. By construction, it satisfies
$$\F_{x,H}(P)\cong \colim_U \F(U \times H \backslash G) \cong \colim_U \Phi^G_H(\F)(U) = \big(\colim_U \Phi^G_H(\F)\big)_x,$$
where $U \subset X^H$ ranges over all open neighbourhoods of $x$.
Therefore, we have 
$$\big(f_*^G\F\big)_{x;H,H \backslash G} \simeq \big(\Phi^G_H(f^G_*\F) \big)_{x}\simeq \big(f_*(\Phi^G_H(\F)) \big)_x.$$
The assertion then follows from Lurie's Proper Base Change Theorem applied to $Y^H \to X^H$ (see \cite[Corollary 7.3.1.18]{htt}).
\end{proof}

\subsection{Vanishing cycles}

Using the formalism introduced so far in this section, we can develop vanishing and nearby cycles for $G$-sheaves of spectra on complex analytic $G$-spaces (and stacks). In the non-equivariant case, i.e. for $G=\{e\}$ being the trivial group, this amounts to a spectral refinement of the classical theory of vanishing cycles. And for this case, the theory presented here is well-known and closely related constructions can be found in \cite{ayoub}.
\subsubsection{Definitions}\label{app:def}

We denote by $G$ a finite group. Let $\Db=\Db_{\varepsilon}(0) \subset \Cb$ be a complex disk endowed with the trivial $G$-action, $\Zc$ a complex analytic stack acted on by $G$ and $\mathcal{Z} \xrightarrow{f} \Db$ a $G$-equivariant morphism of complex analytic stacks. We denote by $\Db^*$ the punctured disk and by $\widetilde{\Db}^* \xrightarrow{\pi} \Db^*$ a universal covering space. Likewise, we write $\Zc^*$ for the base change $\Zc \times_{\Db} \Db^*$, and $\widetilde{\Zc}^*$ for the fibre product $\Zc \times_{\Db} \widetilde{\Db}^*$.

\[
\xymatrix{
\Zc_0 \ar@{^(->}[r]^{i_{\Zc}} \ar[d] & \Zc \ar[d]_{f} & \Zc^{*} \ar@{_(->}[l]_{j_{\Zc}} \ar[d] & \widetilde{\Zc}^* \ar[d] \ar[l]_{\pi_{\Zc}} \\
\{0\} \ar@{^(->}[r]^i & \Db & \Db^* \ar@{_(->}[l]_{j} & \widetilde{\Db}^* \ar[l]_{\pi}
}
\]

We fix a presentable stable $\infty$-category $\Cc$, and denote by $\Sh_{\Cc}(\Zc)$ the category of $\Cc$-valued $G$-hypersheaves on $\Zc$.

\begin{definition}
The $G$-equivariant \emph{nearby cycle functor} 
$\psi_{f}^G\colon \Sh_{\Cc}(\Zc;G) \to \Sh_{\Cc}(\Zc_0;G)$ 
is defined to be 
$$\psi_f=i^{-1}_G(j_{\Zc} \circ \pi_{\Zc})^G_*(j \circ \pi_{\Zc})^{-1}_G.$$
\end{definition}

By adjunction, there is a natural transformation
$$\id_{\Sh_{\Cc}(\Zc)} \to (j_{\Zc} \circ \pi_{\Zc})^G_*(j \circ \pi_{\Zc})_G^{-1}.$$
By composition with $i^{-1}_G$ we infer the existence of a natural transformation, also known as the specialisation morphism
\begin{equation}\label{eqn:van-natural}
\spe\colon i^{-1}_G \to \psi^G_f.
\end{equation}

\begin{definition}
The cofibre of \eqref{eqn:van-natural} will be referred to as the vanishing cycle functor
$$\phi_f^G\colon \Sh_{\Cc}(\Zc;G) \to \Sh_{\Cc}(\Zc_0;G).$$
\end{definition}

\subsubsection{Smooth base change}

\begin{construction}\label{con:nat}
Consider a cartesian diagram of topological stacks and $G$-equivariant morphisms
\begin{equation}\label{eqn:smoothbasechange}
\xymatrix{
\Yc' \ar[r]^{f'} \ar[d]_{g} & \Xc' \ar[d]^{g'} \\
\Yc \ar[r]^f & \Xc.
}
\end{equation}
For $\F \in \Sh_{\Cc}(\Yc;G)$ we have by adjunction a morphism
$$\F \to g^G_*g^{-1}_G\F,$$
which can be composed with $f^G_*$ to obtain
$$f^G_*\F \to f^G_*g^G_*g^{-1}_G\F.$$
Using that $f^G_*g^G_* \simeq g'^G_*f'^G_*$ by commutativity of the diagram above, we can rewrite this map as
$$f^G_*\F \to g'^G_*f'^G_*g^{-1}_G\F.$$
Applying the adjunction of $g'^{-1}_G$ and $g'^G_*$, we obtain a natural morphism
$$g'^{-1}_Gf^G_*\F \to f'^G_*g_G^{-1}\F.$$
\end{construction}

Let $\Yc \xrightarrow{h} \Zc \xrightarrow{f} \Db$ be $G$-equivariant morphisms of complex analytic stacks equipped with a $G$-action. We will use notation similar to the one employed in \ref{app:def}.
\[
\xymatrix{
\Yc_0 \ar@{^(->}[r]^{i_{\Yc}} \ar[d]_{h_0} & \Yc \ar[d]_h & \Yc^{*} \ar@{_(->}[l]_{j_{\Yc}} \ar[d]_{h} & \widetilde{\Yc}^* \ar[d]_{\tilde{h}} \ar[l]_{\pi_{\Yc}} \\
\Zc_0 \ar@{^(->}[r]^{i_{\Zc}} \ar[d] & \Zc \ar[d]_f & \Zc^{*} \ar@{_(->}[l]_{j_{\Zc}} \ar[d] & \widetilde{\Zc}^* \ar[d] \ar[l]_{\pi_{\Zc}} \\
\{0\} \ar@{^(->}[r]^i & \Db & \Db^* \ar@{_(->}[l]_{j} & \widetilde{\Db}^* \ar[l]_{\pi}
}
\]
\begin{construction}\label{con:nat-psi}
Let $\F \in \Sh_{\Cc}(\Yc;G)$. Using the natural transformation from Construction \ref{con:nat} we obtain
$$(h_0)_G^{-1}(j\circ p)^G_* \to (j \circ p)^G_*h^{-1}_G$$
From this we obtain
$$i^{-1}_Gh^{-1}_G(j\circ p)^G_*(j\circ p)^{-1}_G\F \to i^{-1}_G(j \circ p)^G_*h^{-1}_G(j\circ p)^{-1}_G\F.$$
The right-hand side is equivalent to $i^{-1}_G(j \circ p)^G_*(j\circ p)^{-1}_Gh^{-1}_G\F$, whereas the left-hand side is equivalent to $(h_0)_G^{-1}i^{-1}_G(j\circ p)^G_*(j\circ p)^{-1}_G\F$. Thus, we have constructed a natural morphism
\begin{equation}\label{eqn:nat-psi}
(h_0)_G^{-1}\psi^G_f\F \to \psi^G_{f\circ h}h^{-1}_G\F.
\end{equation}
Likewise, we obtain a natural morphism
\begin{equation}\label{eqn:nat-phi}
(h_0)_G^{-1}\phi_f^G\F \to \phi^G_{f\circ h}h^{-1}_G\F.
\end{equation}
\end{construction}

The natural morphism constructed above is required to state the following base change result for vanishing cycles along smooth (i.e. holomorphic and submersive) maps of complex analytic spaces. We remark that it seems likely that some of the assumptions can be relaxed.  

\begin{proposition}[Smooth base change]\label{prop:smooth-basechange}
Suppose that the following assumptions are satisfied: 
\begin{itemize}
\item[-] $h$ is a representable, smooth and $G$-equivariant morphism of complex analytic stacks (with $G$-actions), 
\item[-] $G$ acts trivially on $\Yc$,
\item[-] and $\F \in \Sheaves_{\C}(\Yc;G)$ is a locally constant $G$-sheaf.
\end{itemize}
Then, the morphisms \eqref{eqn:nat-psi} and \eqref{eqn:nat-phi} are equivalences.
\end{proposition}
\begin{proof}
Let $Y \xrightarrow{q} \Yc$ be an atlas of the complex analytic stack $\Yc$. We endow $Y$ with the trivial $G$-action, and hence $q$ is a $G$-equivariant morphism by the second assumption above.

For a subgroup $H \subset G$ we have that $\Yc^H = \Yc$ and likewise, $Y^H = Y$. In particular, we have that 
$$q^H\colon Y^H \to \Yc^H$$
is an atlas. Hence, the induced map between the underlying topological stacks $\underline{Y} \to \underline{\Yc}$ gives rise to an $\mathbf{LF}$-covering for each fixpoint locus $\Yc^H$. In particular, a morphism of $G$-sheaves $\Gc_1 \xrightarrow{\gamma} \Gc_2$ on $\Yc$ is an equivalence if and only if $h^{-1}_G(\gamma)$ is an equivalence.
We may therefore replace $\Yc$ by $Y$ without loss of generality.

Furthermore, by a similar argument we may assume that $Y$ is contractible. Since the $G$-action on $Y$ is the trivial one, and $\F$ is assumed to be (locally) constant, the proof can be finished by evoking a familiar argument (see the proof of Theorem 1.3.5.1 in \cite{conrad}).
\end{proof}

\subsubsection{Pushforwards along proper maps}

\begin{corollary}[Vanishing cycles and proper pushforward]\label{cor:proper-push}
Let $\Yc \xrightarrow{g} \Zc$ be a $G$-equivariant proper morphism of complex analytic stacks with $G$-action, and $f\colon \Zc \to \Db$ a $G$-invariant holomorphic map. Then, for every $\F \in \Sh_{\C}(\Yc;G)$ we have an equivalence
$$\phi^G_f\big(g^G_*(\F)\big) \simeq g^G_*\big(\phi^G_*(\F)\big)\text{ and }\psi^G_f\big(g^G_*(\F)\big) \simeq g^G_*\big(\psi^G_*(\F)\big).$$
\end{corollary}
\begin{proof}
It suffices to prove the result for nearby cycles, which is done as follows:
$$g^G_*\big(\psi^G_*(\F)\big) \simeq g^G_*i^{-1}_G(j \circ \pi_{\Yc})^G_*(j \circ \pi_{\Yc})^{-1}_G(\F) \simeq i^{-1}_G g^G_* (j \circ \pi_{\Yc})^G_*(j \circ \pi_{\Yc})^{-1}_G(\F),$$
where we used proper base change (Theorem \ref{thm:proper-base}) in the last step. The right-hand side is now equivalent to 
$$i^{-1}_G g^G_* (j \circ \pi_{\Yc})^G_*(j \circ \pi_{\Yc})^{-1}_G(\F) \simeq i^{-1}_G (j \circ \pi_{\Zc})^G_*g^G_* (j \circ \pi_{\Zc})^{-1}_G(\F) \simeq i^{-1}_G (j \circ \pi_{\Zc})^G_*(j \circ \pi_{\Zc})^{-1}_Gg^G_* (\F),$$
by proper base change.\end{proof}


\end{document}